\documentclass[12pt,a4paper, reqno]{amsart}
\usepackage{}
\usepackage{txfonts}
\usepackage{color}
\usepackage{url}
\usepackage{fourier}
\usepackage{bbm}
\usepackage{amsfonts}
\usepackage{amsxtra}
\usepackage{amssymb}
\usepackage{amscd}
\usepackage{amsthm}
\usepackage{mathrsfs}
\usepackage{leqno}
\usepackage{multirow}
\usepackage[backref=page]{hyperref}

\makeatletter

\numberwithin{equation}{section}
\newtheorem*{theorem*}{Theorem}
\newtheorem{lemma}{Lemma}[section]
\newtheorem{proposition}[lemma]{Proposition}
\newtheorem{remark}[lemma]{Remark}
\newtheorem{example}[lemma]{Example}

\newtheorem{theorem}[lemma]{Theorem}
\newtheorem{definition}[lemma]{Definition}

\newtheorem{question}[lemma]{Question}

\newtheorem*{question*}{Question}
\newtheorem*{assumption*}{Assumption}
\newtheorem*{axiom*}{Axiom}

\newtheorem*{theorem*1}{Theorem (\ref{theta1})}
\newtheorem*{theorem*2}{Theorem (\ref{theta2})}
\newtheorem*{theorem*3}{Theorem (\ref{theta3})}
\newtheorem*{theorem*4}{Theorem (\ref{theta4})}
\newtheorem*{proposition*5}{Proposition (\ref{theta5})}
\newtheorem*{proposition*6}{Proposition (\ref{theta6})}
\baselineskip 20pt \textwidth 18cm  \theoremstyle{plain}

\newcommand{\Hom}{\operatorname{Hom}}

\renewcommand{\Im}{\operatorname{Im}}

\newcommand{\Ind}{\operatorname{Ind}}

\newcommand{\Res}{\operatorname{Res}}

\newcommand{\Ha}{\operatorname{H}}
\newcommand{\Oa}{\operatorname{O}}

\newcommand{\sgn}{\operatorname{sgn}}

\newcommand{\C}{\mathbb C}

\newcommand{\R}{\mathbb R}
\newcommand{\Za}{\mathbb Z}

\newcommand{\Z}{\mathbb Z}

\newcommand{\Gal}{\operatorname{Gal}}
\newcommand{\GL}{\operatorname{GL}}
\newcommand{\GU}{\operatorname{GU}}
\newcommand{\GSp}{\operatorname{GSp}}
\newcommand{\GMp}{\operatorname{GMp}}
\newcommand{\PGSp}{\operatorname{PGSp}}

\newcommand{\Mp}{\operatorname{Mp}}

\newcommand{\SL}{\operatorname{SL}}
\newcommand{\SO}{\operatorname{SO}}
\newcommand{\Sp}{\operatorname{Sp}}
\newcommand{\Spin}{\operatorname{Spin}}

\newcommand{\U}{\operatorname{U}}

\newcommand{\Span}{\operatorname{Span}}

\newcommand{\diag}{\operatorname{diag}}

\newcommand{\id}{\operatorname{Id}}

\newcommand{\sign}{\operatorname{sign}}

\newcommand{\tr}{\operatorname{Tr}}
\newcommand{\Tr}{\operatorname{Tr}}
\newcommand{\Trd}{\operatorname{Trd}}
\newcommand{\Nrd}{\operatorname{Nrd}}
\newcommand{\Arg}{\operatorname{Arg}}

\begin{document}
\title{Extended  Weil representations: the real  field case}
\author{Chun-Hui Wang}
\address{School of Mathematics and Statistics\\Wuhan University \\Wuhan, 430072,
P.R. CHINA}
\email{cwang2014@whu.edu.cn}
\begin{abstract}
 Let F be the usual real field. Let W be a symplectic
vector space over F. It is  known that there are two different Weil representations of a Meteplectic covering
group $\widetilde{\Sp}(W)$. By some twisted actions, we reorganize  them  into a representation of $\widetilde{\Sp}^{\pm}(W)$,  a covering group over a subgroup $\Sp^{\pm}(W)$ of $\GSp(W)$. Based on the works of MVW, Kudla, and Howe on reductive dual pairs in $\Sp(W)$,  we  explore the analogous dual pairs in $\Sp^{\pm}(W)$ . Finally, following Lion-Vergne's classical book on Weil representations and theta series,  we  investigate  some simple theta series in $\Sp^{\pm}(W)$ where  $W$ has dimension two.  
\end{abstract}
\maketitle
\setcounter{secnumdepth}{5}
\tableofcontents{}
\section{Introduction}
 Let  $F$ be a local field. Let $(W, \langle, \rangle)$ be a symplectic vector space  of  dimension $2m$   over  $F$.  Let $\Ha(W)=W\oplus F$ denote the usual Heisenberg group. Let   $\Sp(W)$(resp. $\GSp(W)$) denote the corresponding symplectic group(resp. similitude symplectic group). Let $\psi$ be a non-trivial unitary  character of $F$.   Let $\widetilde{\Sp}(W)$ denote an $8$-fold metaplectic covering group over $\Sp(W)$. Let $\pi_{\psi}$ be a Weil representation of $\widetilde{\Sp}(W)\ltimes \Ha(W)$ associated to $\psi$, and let $\omega_{\psi}=\pi_{\psi}|_{\widetilde{\Sp}(W)}$.  One question  is how to extend the Weil representation to the similitude group.   This question has been studied systematically  by Barthel  in \cite{Ba} in the $p$-adic field case. In that paper, Barthel constructed the Metaplectic 2-cocycle over
$\GSp(W)$, defined the corresponding Metaplectic group $\GMp(W )$, and then obtained the extended Weil representation by induction. One can also see another two important papers: B.Robets  \cite{Ro1} and Gan-Tantono \cite{GaTa}, which  contain some similar results on the theta correspondences for the similitude groups.  Consequently,   C.Zhang   investigated unitary similitude dual pairs in \cite{Zh}. In the archimedean field  case,  one can see M.Cognet's two papers\cite{Co1,Co2} on $\GSp_8(F)$.
    In the global field case,  many people contribute to this subject. For examples, one can see Harris-Kudla-Sweet's paper\cite{HaKuSw}, S.Gelbart's book \cite{Ge},   B. Roberts'paper \cite{Ro2},  etc. 
  We mostly approach such question mainly   by following  \cite{Ba}, \cite{GePi},\cite{Wa1},\cite{Wa2}, \cite{LiVe}.   Through some twisted actions, we continue their works by constricting the Weil representation.  In this study, we continue our work into the real field case, moving from the p-adic field case(cf. \cite{Wa3}).
  
\subsection{}  Let $F$ be a real field. Let  $a\in F^{\times}$, and  let $\psi^a$ denote another character of $F$ defined as $t\longmapsto \psi(at)$, for $t\in F$. Let $\omega_{\psi^a}$ be the Weil representation of $\widetilde{\Sp}(W)$ associated to $\psi^a$. It is known that  $\omega_{\psi^{at^{2}}}\simeq \omega_{\psi^a}$, for any $t\in F^{\times}$. Loosely speaking, there  exist two different kind of Weil representations associated to $\psi$ and $\psi^{-1}$. One of our purpose is to reorganizer them as a representation of a bigger group than $\widetilde{\Sp}(W)$.  Let $F^{\times }_+$ be the multiplicative group of positive real numbers. Let $\PGSp^{\pm}(W)=\GSp(W)/F_+^{\times}$, which is isomorphic with the  subgroup  $\Sp^{\pm}(W)$  of $\GSp(W)$. Then there exists the following exact sequence:
\begin{align*}
1 \longrightarrow \Sp(W)\longrightarrow \Sp^{\pm}(W)\longrightarrow F^{\times}/F^{\times 2}\longrightarrow 1.
\end{align*}
Let $\widetilde{c}_{X^{\ast}}(-,-)$ denote the Perrin-Rao cocycle on $\Sp(W)$. Based on  Barthel's paper \cite{Ba}  in the $p$-adic case,  we can extend this cocycle  from $\Sp(W)$ to  $\GSp(W)$ and $\Sp^{\pm}(W)$ in  Section \ref{ext}. Let $\widetilde{\Sp}^{\pm}(W)$ denote the corresponding covering group. Then there exists an  exact sequence:
\begin{align*}
1  \longrightarrow \mu_8 \longrightarrow \widetilde{\Sp}^{\pm}(W) \longrightarrow \Sp^{\pm}(W) \longrightarrow 1,
\end{align*}
and $\widetilde{\Sp}^{\pm}(W) \simeq F_2\ltimes \widetilde{\Sp}(W)$, for $F_2=\{\begin{pmatrix}
1& 0\\
0& \pm 1
\end{pmatrix}\}$. As a consequence, we can define 
$$\Pi_{\psi}=\Ind_{\widetilde{\Sp}(W)\ltimes \Ha(W)}^{\widetilde{\Sp}^{\pm}(W)\ltimes \Ha(W)} \pi_{\psi}, \qquad \Omega_{\psi}=\Res_{\widetilde{\Sp}^{\pm}(W)}^{\widetilde{\Sp}^{\pm}(W)\ltimes \Ha(W)}\Pi_{\psi}.$$
The restriction of $\Omega_{\psi}$ on $\widetilde{\Sp}(W)$ consists of the exact two different Weil representations. We also consider  the $2$-fold Metaplectic covering group by following P.Perrin, R.Ranga.Rao and S.Kudla's works(cf. \cite{Pe},\cite{Ra},\cite{Ku}).
 
\subsection{} In the second part(cf. Section \ref{splitting}), we consider the  dual pair in $\Sp^{\pm}(W)$. Let us present two known results on irreducible reductive dual pairs in $\Sp(W)$.  Retain notations from Section \ref{splitting}. 
\begin{theorem}\label{scindagedugroupeR01}
Let $(\U(\mathcal{W}_1), \U(\mathcal{W}_2))$ be an  irreducible reductive dual pair in $\Sp(W)$.
\begin{itemize}
\item[(1)] If $(\mathcal{W}_1, \mathcal{W}_2)$  is not the exception case (1*), then the exact sequence $1 \longrightarrow T \longrightarrow \widehat{\U}(\mathcal{W}_i) \longrightarrow \U(\mathcal{W}_i)\longrightarrow 1$ splits, for $i=1, 2$.
\item[(2)]  If $(\mathcal{W}_1, \mathcal{W}_2)$  is  the exception case and $\epsilon_1=-1$, $\epsilon_2=1$, then
\begin{itemize}
 \item the exact sequence $1 \longrightarrow T \longrightarrow \widehat{\Oa}(\mathcal{W}_2) \longrightarrow \Oa(\mathcal{W}_2)\longrightarrow 1$ splits, but
  \item the exact sequence $1 \longrightarrow T \longrightarrow \widehat{\Sp}(\mathcal{W}_1) \longrightarrow \Sp(\mathcal{W}_1)\longrightarrow 1$  does not split.
  \end{itemize}
    \end{itemize}
\end{theorem}
The $p$-adic version of this theorem  has been proved in  \cite[Chapter 1 ]{MoViWa}   and  \cite{Ku1}. For the   real field case,  it was   proved in \cite[Prop.4.1]{Ku1}. Moreover,  explicit trivialization is also given in \cite{Ku1}. Indeed, there is  an inconspicuous question on which 2-degree covering group or  $8$-degree covering group is also splitting.
\begin{theorem}\label{scindagedugroupeR41}
Let $(\U(\mathcal{W}_1), \U(\mathcal{W}_2))$ be an  irreducible reductive dual pair in $\Sp(W)$ as given above.  Then $\widehat{\U}(\mathcal{W}_1)$  commutates with  $\widehat{\U}(\mathcal{W}_2)$ in $\widehat{\Sp}(W)$.
\end{theorem}
The $p$-adic version of this theorem  has been proved in  \cite[Chapter 1 ]{MoViWa}. For the   real field case,  one can see \cite{Ho1,Ho2}.  We generalize the above two results to the group $\Sp^{\pm}(W)$. 
\begin{proposition}\label{splitting1}
Let  $(\mathcal{W}_1, \mathcal{W}_2)$ be a pair  in Theorem \ref{scindagedugroupeR01}.
\begin{itemize}
\item[(1)] If $(\mathcal{W}_1, \mathcal{W}_2)$  is not the exception case (1*), then the exact sequence $1 \longrightarrow T \longrightarrow \widehat{\U^{\pm}}(\mathcal{W}_i) \longrightarrow \U^{\pm}(\mathcal{W}_i)\longrightarrow 1$ splits, for $i=1, 2$.
\item[(2)]  If $(\mathcal{W}_1, \mathcal{W}_2)$  is  the exception case and $\epsilon_1=-1$, $\epsilon_2=1$, then
\begin{itemize}
 \item[(i)] the exact sequence $1 \longrightarrow T \longrightarrow \widehat{\Oa^{\pm}}(\mathcal{W}_2) \longrightarrow \Oa^{\pm}(\mathcal{W}_2)\longrightarrow 1$ splits, but
  \item[(ii)] the exact sequence $1 \longrightarrow T \longrightarrow \widehat{\Sp^{\pm}}(\mathcal{W}_1) \longrightarrow \Sp(\mathcal{W}_1)\longrightarrow 1$  does not split.
  \end{itemize}
    \end{itemize}
\end{proposition}
\begin{proposition}\label{comm1}
Let  $(\mathcal{W}_1, \mathcal{W}_2)$ be a pair  in Theorem \ref{scindagedugroupeR01}.
\begin{itemize}
\item[(1)] If $D=E$ or $\mathbb{H}$,  $\widehat{\U^{\pm}}(\mathcal{W}_1) $ commutes with  $\widehat{\U^{\pm}}(\mathcal{W}_2)$ in $\widehat{\Sp}^{\pm}(W)$.
\item[(2)] If $D=F$,  assume $\epsilon_1=-1$, $\epsilon_2=1$.
\begin{itemize}
 \item[(a)] If $2\mid \dim \mathcal{W}_2$, and $\mathcal{W}_2$ is not  a hyperbolic space, then:
 \begin{itemize}
\item[(i)] $\Oa^{\pm}(\mathcal{W}_2)=\Oa(\mathcal{W}_2)$.
 \item[(ii)] if $4 \mid \dim \mathcal{W}_1$, then $\widehat{\Sp^{\pm}}(\mathcal{W}_1) $ commutes with  $\Oa(\mathcal{W}_2)$ in $\widehat{\Sp}^{\pm}(W)$.
  \item[(iii)] if $4\nmid \dim \mathcal{W}_1$, then $\widehat{\Sp^{\pm}}(\mathcal{W}_1) $ only commutes with  $\SO(\mathcal{W}_2)$ in $\widehat{\Sp}^{\pm}(W)$, and $\Oa(\mathcal{W}_2)$ only commutes with  $\widehat{\Sp}(\mathcal{W}_1)$ in $\widehat{\Sp}^{\pm}(W)$.
  \end{itemize}
   \item[(b)] If $2\mid \dim \mathcal{W}_2$, and $\mathcal{W}_2$ is  a hyperbolic space, then:
 \begin{itemize}
\item[(i)] $\widehat{\Oa^{\pm}}(\mathcal{W}_2)\simeq \Oa^{\pm}(\mathcal{W}_2) \times T$.
 \item[(ii)] if $4 \mid \dim \mathcal{W}_1$,  then $\widehat{\Sp^{\pm}}(\mathcal{W}_1) $ commutes with  $\widehat{\Oa^{\pm}}(\mathcal{W}_2)$ in $\widehat{\Sp}^{\pm}(W)$.
  \item[(iii)] if $4 \nmid \dim \mathcal{W}_1$ and  $4\mid \dim \mathcal{W}_2$, then $\widehat{\Sp^{\pm}}(\mathcal{W}_1) $ only commutes with  $\widehat{\SO^{\pm}}(\mathcal{W}_2)$ in $\widehat{\Sp}^{\pm}(W)$, and $\widehat{\Oa^{\pm}}(\mathcal{W}_2)$ only commutes with  $\widehat{\Sp}(\mathcal{W}_1)$ in $\widehat{\Sp}^{\pm}(W)$.
 \item[(iv)] if  $4\nmid \dim \mathcal{W}_1$ and   $4\nmid \dim \mathcal{W}_2$,  then $\widehat{\Sp^{\pm}}(\mathcal{W}_1) $ only commutes with  $\widehat{\SO}(\mathcal{W}_2)$ in $\widehat{\Sp}^{\pm}(W)$, and $\widehat{\Oa^{\pm}}(\mathcal{W}_2)$ only commutes with  $\widehat{\Sp}(\mathcal{W}_1)$ in $\widehat{\Sp}^{\pm}(W)$.
  \end{itemize}
  \item[(c)]  If $2\nmid \dim \mathcal{W}_2$, then:
\begin{itemize}
\item[(i)] $\Oa^{\pm}(\mathcal{W}_2)=\Oa(\mathcal{W}_2)$.
 \item[(ii)] if $4 \mid \dim \mathcal{W}_1$, then $\widehat{\Sp^{\pm}}(\mathcal{W}_1) $ commutes with  $\Oa(\mathcal{W}_2)$ in $\widehat{\Sp}^{\pm}(W)$.
  \item[(iii)] if $4\nmid \dim \mathcal{W}_1$, then $\widehat{\Sp^{\pm}}(\mathcal{W}_1) $ only commutes with  $\SO(\mathcal{W}_2)$, and  $\Oa(\mathcal{W}_2)$ only commutes with  $\widehat{\Sp}(\mathcal{W}_1)$ in $\widehat{\Sp}^{\pm}(W)$.
  \end{itemize}
    \end{itemize}
    \end{itemize}
\end{proposition}
It seems that Prop.\ref{comm1}(2)(c) supports J.-S.Li's works \cite{Li1},\cite{Li2} and Jiang-Soudry 's work \cite{JiSo} on the theta correspondence between $\Mp_{2n}$ and $\SO_{2r+1}$. For the finer results on   Shimura-Waldspurger correspondences, one can see Gan-Ichino, Gan-Li's recent papers \cite{GaIc}, \cite{GaLi}.
  
\subsection{} In the third part(cf. Section \ref{Parenthesis}), we consider the example when $\dim W=2$ and  some theta series associated to  the group $\SL_2^{\pm}(\R)$.  We mainly follow the classical book \cite{LiVe} to deal with theta series. In an  alternative way by using the global method, one can  see \cite{We} and \cite{Ge}. Retain  notations from Section \ref{notationss}. Let $\Gamma=\breve{\Gamma}(2)^{\pm}=\langle \omega\rangle \ltimes \Gamma(2)^{\pm}$ be 
a discrete subgroup of $\SL_2^{\pm}(\R)$.  Let $\gamma=\begin{pmatrix}
                                        a & b \\
                                        c & d 
                                      \end{pmatrix}\in \Gamma $. Let us define $\sgn(z)=1$ or $-1$ according to $z\in \mathcal{H}$ or $z\in \mathcal{H}^-$. Let us define 
                                     
\begin{align*}
\theta_{1/2}(z, \epsilon)&=\sum_{n\in \Z} e^{i \epsilon  \pi n^2 z},
\end{align*}
\begin{align*}
\theta_{3/2}(z, \epsilon)&=\sum_{n\in \Z} n e^{i \epsilon  \pi n^2 z},
\end{align*}
for $z\in \mathcal{H}^{\pm}$, $\epsilon \in \mu_2$. Then: 
\begin{align}\label{theta1}
 \theta_{1/2}(\tfrac{az+b}{cz+d}, \epsilon)=\lambda^{\pm}(\gamma, \epsilon) \sqrt{\det \gamma(cz+d)}\sum_{n\in \Z} e^{i (\det \gamma)\epsilon  \pi n^2 z},
\end{align}
\begin{align}\label{theta2}
 \theta_{3/2}(\tfrac{az+b}{cz+d}, \epsilon)=\lambda^{\pm}(\gamma,  \epsilon) (\sqrt{\det \gamma(cz+d)})^3\sum_{n\in \Z} ne^{i (\det \gamma)\epsilon  \pi n^2 z},
\end{align}
for $z\in \mathcal{H}^{\pm}$, $\gamma\in \Gamma$ if $\epsilon=1$ or $\gamma \in \Gamma\setminus \{ h(-1), h(-1)h_{-1}\}$ if $\epsilon=-1$, and some constant $\lambda^{\pm}(\gamma, \epsilon)$.  The main difference between  theta series associated to  $\SL_2^{\pm}(\R)$ and $\SL_2(\R)$ is the following: these Metaplectic coverings over $\breve{\Gamma}(2)^{\pm}$ and $\SO_2^{\pm}(\R)$  are not splitting.(cf. Lemmas \ref{thetaf3}, \ref{nonsp})   The above two theta series only generalize the theorem in \cite[p.204, 2.4.13]{LiVe}.  In \cite[p.205, Section 2.4.16]{LiVe}, the two authors also deal with much  complicated theta series and  connect them  with Serre and Stark's important result in \cite{SeSt}. We expect to work out some similar results  in a separate  paper in future. Finally let us propose a question to think:
\begin{question}
If we replace the above discrete group $\Gamma$ by $\SL_2(\Z)$ or $\SL_2^{\pm}(\Z)$, how to obtain  modular forms or theta series from Weil representations in a natural way?
\end{question}
\subsection{} In the fourth part(cf. Section \ref{app}),  we  study the works \cite{LiVe}, \cite{Pe}, and \cite{Ra} in a simple case---$\dim W=2$. Let  $\widehat{\SL_2}(\Z)$ denote  the central extension group of $\SL_2(\Z)$ by $T$.    By \cite{As}, \cite{KiMe}, etc.,  it is known that $\widehat{\SL_2}(\Z)$ can be split by some trivialization map from  Dedekind sums and modular forms of weight $12$. On the other hand, in  \cite{LiVe}, the authors also gave another trivialization map for $\Gamma(2)$. In that section, we will examine their results in an elementary way, and then formulate the explicit character between these two trivialization maps for $\Gamma(2)$. Finally, by comparing two different models of Weil representations, we will demonstrate the splitting for $\widehat{\Gamma(2)}$ in a different way.  This method has already been used in \cite{LiVe} only
not pointing out.
\section{Preliminaries}
 \subsection{Notations and conventions}
 In the whole text, we will  use the following notion and conventions. Let $\R$ and $\C$ denote  the usual real numbers and complex numbers. Let $F=\R$. Let $(W, \langle, \rangle)$ be a symplectic vector space  of  dimension $2m$   over  $F$.  Let $\Ha(W)=W\oplus F$ denote the usual Heisenbeg group, with the multiplication law given by
$$(w, t)(w',t')=(w+w', t+t'+\tfrac{\langle w,w'\rangle}{2}),
\quad\quad w, w'\in W, t,t'\in F.$$
 Let   $\Sp(W)$(resp.$\GSp(W)$) denote the corresponding symplectic group(resp.similitude symplectic group). Let $\lambda$ be the similitude factor map from $\GSp(W)$ to $F^{\times}$, and for $g\in \GSp(W)$, let $\lambda_g$ denote its similitude factor.  Let $\{e_1, \cdots, e_m; e_1^{\ast}, \cdots, e_m^{\ast}\}$ be a symplectic basis of $W$ so that $\langle e_i, e_j\rangle=0=\langle e_i^{\ast}, e_j^{\ast}\rangle$, $\langle e_i, e_j^{\ast}\rangle =\delta_{ij}$. Let $X=\Span\{ e_1, \cdots, e_m\}$, $X^{\ast}=\Span\{e_1^{\ast}, \cdots, e_m^{\ast}\}$.

 In $\Sp(W)$, let  $P=P(X^{\ast})=\{ g\in \Sp(W) \mid X^{\ast}g=X^{\ast}\}=\{ \begin{pmatrix} a& b\\ 0& (a^{\ast})^{-1}\end{pmatrix}\in \Sp(W)\}$, $N=\{ \begin{pmatrix} 1& b\\ 0& 1\end{pmatrix} \in \Sp(W)\}$, $M=\{ \begin{pmatrix} a& 0\\ 0& (a^{\ast})^{-1}\end{pmatrix}\in \Sp(W)\}$, $N^{-}=\{ \begin{pmatrix} 1& 0\\ c& 1\end{pmatrix} \in \Sp(W)\}$.  For a subset $S\subseteq \{1, \cdots, m\}$, let us define an element   $\omega_S$ of  $ \Sp(W)$ as follows: $ (e_i)\omega_S=\left\{\begin{array}{lr}
-e_i^{\ast}& i\in S\\
 e_i & i\notin S
 \end{array}\right.$ and $ (e_i^{\ast})\omega_S=\left\{\begin{array}{lr}
e_i^{\ast}& i\notin S\\
 e_i & i\in S
 \end{array}\right.$.  In particular,  we will  write simply  $\omega$,  for  $S=\{1, \cdots, m\}$.

  Let $(-, -)_F$ denote the Hilbert symbol defined from $F^{\times} \times F^{\times}$ to $\{ \pm 1\}$.  More precisely,
   $$ (a, b)_F=\left\{ \begin{array}{lr} -1& a<0, b<0\\ 1 & \textrm{ otherwise} \end{array} \right..$$
   Let $((a))$ to denote the number  $0$ or $a-[a]-\tfrac{1}{2}$, depending on whether $a$ is an integer or not.   For two coprime integers $c,d$ with $c>0$, let $s(d, c)$ denote  the Dedekind sum, given by $\sum_{k=1}^c ((\tfrac{k}{c})) ((\tfrac{kd}{c}))$.

   If  $V$ is a vector space over $F$ of dimension $l$, and  $Q$ is    a non-degenerate quadratic form on $V$, defined as
     $Q(x)=\sum_{i,j} b_{ij} x_ix_j$,
     for $b_{ij}=b_{ji}$, then the determinant $\det Q$ of $Q$ is defined by $\det Q=\det(b_{ij})$. If $Q \sim a_1x_1^2+ \cdots +a_lx_l^2$,  the Hasse invariant $\epsilon(Q)$ of $Q$ is defined by  $\epsilon(Q)=\prod_{1 \leq i< j\leq l} (a_i, a_j)_F$.

If $\psi$ is a non-trivial character of $F$ and $a\in F$,   we will  write   $\psi^a$ for the character: $t\longrightarrow \psi(at)$.  Let $i=\sqrt{-1}$. We will let  $\psi_0$ denote the fixed character of $F$ defined as: $ t \longmapsto e^{2\pi it}$, for $t\in F$. Let $\mu_n=\langle e^{\tfrac{2\pi i}{n}}\rangle$,  $ e^{\tfrac{2\pi i}{n}} \in \C^{\times}$. Let $T=\{ e^{i\theta}\mid \theta \in \R\}$.

For an $F$-vector space $V$, we will let $L^p(V)$(resp. $S(V)$) denote the $L^p$-vector space(resp. Schwartz space) over $V$.
\subsection{Weil representation}
Note that the center of $\Ha(W)$ is just $F$. Let $\psi$ be a non-trivial continuous unitary character of $F$. According to the Stone-von Neumann's theorem, only one unitary irreducible complex representation of $\Ha(W)$ with central character $\psi$, up to unitary equivalence, exists.  Let us call it the Heisenberg representation, and denote it  by $\pi_{\psi}$. The symplectic group $\Sp(W)$ can then act on $\Ha(W)$ and leave the center $F$ pointwise unchanged. By Weil's famous work,  $\pi_{\psi}$ can give rise  to a projective representation of $\Sp(W)$ and then  to an actual representation of a $\C^{\times}$-covering group over $\Sp(W)$. As is customary, we refer to such group as a Metaplectic group and the actual representation as a Weil representation. As $F=\R$,  this special  central cover can descend  to an  $8$-degree or a $2$-degree cover over $\Sp(W)$.  Let us recall some explicit $2$-cocycles associated with   Metaplectic groups by following \cite{LiVe, Pe, Ra}.
\subsection{2-cocycle I}
 Let $\psi=\psi_0^{e}$ be a non-trivial character of $F$, for some $e\in F^{\times}$.   Let  $(Q,V)$ be  a non-degenerate quadratic vector space  on $F$. Then:
 $$Q(x+y)-Q(x)-Q(y)=(x, y \rho),$$
 for an isomorphism $\rho: V \longrightarrow V^{\ast}=\Hom_F(V, F)$.  Applying the character $\psi$, we obtain
 $$\psi(Q(x+y))\psi(Q(x))^{-1}\psi(Q(y))^{-1}=\psi(x, y \rho).$$
 Note that  $\psi(x, y\rho)$ is a bicharacter in $x$ and $y$. Following  Weil, we call $  \psi(Q(x))$ a character of   second degree. By \cite{We}, there exists a complex constant  associated with  $\psi(Q(x))$, called the Weil index, denoted by $\gamma(\psi(Q))$.
 \begin{example}[{\cite[Prop. A.10]{Ra}}]
 If $V=F$ and $Q(x)=x^2$, then $\gamma(\psi(Q))=\psi_0( \tfrac{\sign e}{8})$.
 \end{example}
If $(Q, V)$ is the above case,  we  will write   $\gamma(\psi)$ for   $\gamma(\psi(Q))$, and   define  $\gamma(a, \psi)= \tfrac{\gamma(\psi^a)}{\gamma(\psi)}$ by following \cite{Ku}. For the general non-degenerate  quadratic space $(Q,V)$,  we have:
\begin{lemma}[{\cite[Def.A.6]{Ra}}]
$\gamma(\psi(Q))=\epsilon(Q)\gamma(\psi)^{\dim V}\gamma(\det Q, \psi)$.
\end{lemma}
 For $g_1, g_2 \in \Sp(W)$, let $q(g_1,g_2)=q(X^{\ast}, X^{\ast} g_2^{-1}, X^{\ast} g_1)$ be the corresponding Leray invariant as given in  \cite[Sections 2.3-2.6]{Ra}. Let us define:
\begin{align}
\widetilde{c}_{X^{\ast}}(g_1, g_2)=\gamma(\psi(q(g_1, g_2)/2)).
\end{align}
Then  $\widetilde{c}_{X^{\ast}}(-,-)$ defines a non-trivial class of order $2$ in $\Ha^2(\Sp(W), \mu_8)$. Let  $\widetilde{\Sp}(W)$  denote the associated   Metaplectic group.
\subsection{2-cocycle II}
As is known that  $\Sp(W)= \sqcup_{j=0}^m C_j$,
where $C_j= P\omega_SP$ for the above $\omega_S$ with $|S|=j$.   In \cite{Ra}, Rao defined the following functions:
\begin{align}
x: \Sp(W) \longrightarrow F^{\times}/{(F^{\times})}^2; p_1\omega_Sp_2 \longmapsto \det(p_1p_2|_{X^{\ast}}) (F^{\times })^2,
\end{align}
\begin{align}
t: \Sp(W) \times \Sp(W) \longrightarrow \Z; (g_1, g_2)\longmapsto \tfrac{1}{2}(|S_1| + |S_2| - |S_3| -l),
\end{align}
 where $ g_1=p_1\omega_{S_1} p_1', g_2= p_2 \omega_{S_2} p_2'$  and  $ g_1g_2= p_3 \omega_{S_3} p_3', l=\dim q(g_1, g_2)$.  By  \cite[p.360, Def.5.2]{Ra},  we can  define the normalizing constant by
\begin{align}\label{mx}
m_{X^{\ast}}: \Sp(W) \longrightarrow \mu_8; g \longmapsto \gamma(x(g), \psi^{\tfrac{1}{2}})^{-1} \gamma(\psi^{\tfrac{1}{2}})^{-j(g)}
\end{align}
 for $g=p_1\omega_Sp_2$, $j(g)=|S|$. Let  us define
 \begin{equation}\label{28inter}
\overline{c}_{X^{\ast}}(g_1, g_2)=m_{X^{\ast}}(g_1g_2)^{-1} m_{X^{\ast}}(g_1) m_{X^{\ast}}(g_2) \widetilde{c}_{ {X^{\ast}}}(g_1,g_2), \quad\quad g_i\in \Sp(W).
\end{equation}
 By \cite[p.360]{Ra},
 $\overline{c}_{X^{\ast}}$ defines a $2$-cocycle on $\Sp(W)$ with values in $\mu_2$.
 Let  $\overline{\Sp}(W)$  denote the associated   Metaplectic group. More precisely, by \cite[Thm.5.3]{Ra} and \cite[p.21, Remark 4.6]{Ku},we have:
\begin{align}
\overline{c}_{X^{\ast}}(g_1, g_2)= (x(g_1), x(g_2))_F(-x(g_1)x(g_2), x(g_1g_2))_F ((-1)^t, \det(2q))_F (-1, -1)_F^{\tfrac{t(t-1)}{2}}\epsilon(2q),
\end{align}
where $t=t(g_1,g_2)$, $q=q(g_1,g_2)$, for  $g_1, g_2 \in \Sp(W)$.
\begin{example}
If $\dim W=2$, $\Sp(W)\simeq \SL_2(F)$. For $g_1, g_2, g_3=g_1g_2\in \SL_2(F)$ with $g_i=\begin{pmatrix}
a_i & b_i\\
c_i& d_i\end{pmatrix}$, we have:
\begin{itemize}
\item[(1)] $x(g_1)=\left\{\begin{array}{lr} d_1F^{\times 2} & \textrm{ if } c_1=0\\
c_1F^{\times 2} & \textrm{ if } c_1\neq 0\end{array}\right.$;
\item[(2)] $\overline{c}_{X^{\ast}}(g_1, g_2)=(x(g_1), x(g_2))_F(-x(g_1)x(g_2), x(g_3))_F$.
\end{itemize}
\end{example}
\begin{proof}
See \cite[p.364]{Ra}.
\end{proof}
In particular, if we take $g_1=\begin{pmatrix}
-1 & 0\\
0& -1\end{pmatrix}$, then $\overline{c}_{X^{\ast}}(g_1, g_2)=(x(g_1), x(g_2))_F=(-1,x(g_2))_F$, which is not continuous on the variable $g_2$. For the reason, we can take a series of elements $g_2^{(n)}=\begin{pmatrix}
-1 & 0\\
\tfrac{1}{n}& -1\end{pmatrix}$. Then: $g_2^{(n)} \longrightarrow g_1$ as $n\longrightarrow +\infty$, and  $\overline{c}_{X^{\ast}}(g_1, g_2^{(n)})=1$, but $\overline{c}_{X^{\ast}}(g_1, g_1)=-1$. Note that $\overline{c}_{X^{\ast}}(g_1, g_1^{-1})\neq 1$.
\subsection{Schr\"odinger model}\label{Sch} Note that $X^{\ast}\times F$ is a closed subgroup of $\Ha(W)$. Let $\psi_{X^{\ast}}$ be a one-dimensional  unitary representation of $X^{\ast}\times F$ extended from $\psi$ by the trivial action of $X^{\ast}$.  Let
$$ \pi_{\psi}=\Ind_{X^{\ast}\times F}^{\Ha(W)}\psi_{X^{\ast}}, V_{\psi}=\Ind_{X^{\ast} \times F}^{\Ha(W)} \C.$$
According to \cite[pp.110-111]{BeHa}, $\pi_{\psi}$ defines a Heisenberg representation $\Ha(W)$ with central character $\psi$. By Weil's result, $\pi_{\psi}$ can extend to be a unitary representation of $\widetilde{\Sp}(W) \ltimes \Ha(W)$. Here, we require  the action of the center subgroup $\mu_8$ of $\widetilde{\Sp}(W)$ to be  the identity map. Since $\Sp(W)$ is a perfect group, this extension is unique.

Let $dx$ denote the self-dual Haar measure on $F$ with respect to $\psi$.  Let $dx$,  $dx^{\ast}$ be the Haar measures on $X$, $X^{\ast}$, given by the product  measures from $dx$ of  $F$. The representation $\pi_{\psi}$ can be realized on $L^2(X)$ by the following formulas:
\begin{equation}\label{representationsp11}
\pi_{\psi}[(x,0)+(x^{\ast},0)+(0,k)]f(y)=\psi(k+\langle x+ y,x^{\ast}\rangle) f(x+y),
\end{equation}
\begin{equation}\label{representationsp2}
\pi_{\psi}[ \begin{pmatrix}
  1&b\\
  0 & 1
\end{pmatrix},t]f(y)=t\psi(\tfrac{1}{2}\langle y,yb\rangle) f(y),
\end{equation}
\begin{equation}\label{representationsp3}
\pi_{\psi}[ \begin{pmatrix}
  a& 0\\
  0 &a^{\ast -1 }
\end{pmatrix},t]f(y)=t|\det(a)|^{1/2} f(ya),
\end{equation}
\begin{equation}\label{representationsp4}
\pi_{\psi}([\omega, t])f(y)=t\int_{X^{\ast}} \psi(\langle y, y^{\ast}\rangle) f(y^{\ast}\omega^{-1}) dy^{\ast},
\end{equation}
where $f(y)\in S(X)$, $x,y \in X,  x^{\ast}, y^{\ast}\in   X^{\ast}$, $k\in F$, $t\in \mu_8$, and $\begin{pmatrix}
  a& 0\\
  0 &a^{\ast -1 } \end{pmatrix}$, $\begin{pmatrix}
  1&b\\
  0 & 1
\end{pmatrix}$, $\omega\in \Sp( W)$, $e_i \omega=-e_i^{\ast}$, $e_i^{\ast}\omega=e_i$.
Moreover, under such actions, for $g_1, g_2\in \Sp( W)$,
$$\Pi_{\psi}(g_1)\Pi_{\psi}(g_2)=\widetilde{c}_{X^{\ast}}(g_1,g_2)\Pi_{\psi}(g_1g_2),$$ for the Perrin-Rao's 2-cocyle $\widetilde{c}_{X^{\ast}}(-,-)$ associated with  $\psi$ and $X^{\ast}$. This precise result can be seen from \cite{Ra} and \cite{Pe}.

For the other elements  $\omega_S $, let $X_S=\Span\{ e_i\mid i\in S\}$, $X_{S'}=\Span\{ e_i\mid i\notin S\}$ and $X_S^{\ast}= \Span\{ e^{\ast}_i\mid i\in S\}$, $X_{S'}^{\ast}= \Span\{ e^{\ast}_i\mid i\notin S\}$. By \cite[p.388, Prop.2.1.4]{Pe} or \cite[p.351, Lmm.3.2,(3.9)]{Ra}, for $x=x_s+x_{s'}\in X_S\oplus X_{S'}$,
 \begin{equation}
  \begin{split}
 [\Pi_{\psi}(\omega_S)f](x_s+x_{s'})&=\int_{X_S^{\ast}} \psi(\tfrac{1}{2}\langle x_{s'} \omega_S,x_s  \omega_S\rangle+ \langle z_s^{\ast} \omega_S,  x_s\omega_S\rangle) f(x_{s'}\omega_S+z_s^{\ast}\omega_S)dz_s^{\ast}\\
  &=\int_{X_S^{\ast}} \psi(\langle x_s,  z_s^{\ast}\rangle ) f(x_{s'}\omega_S+z_s^{\ast}\omega^{-1}_S)dz_s^{\ast}.
    \end{split}
  \end{equation}
\subsection{Lattice model}\label{latticemodel}
To state the lattice model, let us define another type of Heisenberg group $\Ha_{\psi}(W)$, which consists of the set $W\oplus T$, with the multiplication law given by
$$(w, t)(w',t')=(w+w', tt'\psi(\langle w,w'\rangle/2)),
\quad\quad w, w'\in W, t,t'\in T.$$
Then there exists a group homomorphism:
$$h_{\psi}: \Ha(W) \longrightarrow \Ha_{\psi}(W); (w, t)\longmapsto (w, \psi(t)).$$
Let $L$ be a lattice in $W$(cf. \cite[p.138]{LiVe}). Let us define
$$L^{\bot}=\{ w\in W\mid \psi(\langle w, l\rangle)=1, \textrm{ for all } l\in L\}.$$
If $L=L^{\bot}$, $L$ is called a self-dual lattice with respect to $\psi$. Let $\Ha(L)$ be the subgroup of $\Ha(W)$ with the underlying set $L\times F$. Let $\Ha_{\psi}(L)$ be the image of $\Ha(L)$ in $\Ha_{\psi}(W)$ via the map $h_{\psi}$. Then $\Ha_{\psi}(L)$ is a  normal  maximal abelian subgroup of $\Ha_{\psi}(W)$. Let us define
$$\rho=\Ind_{\Ha_{\psi}(L)}^{\Ha_{\psi}(W)} \psi_T,$$
where $\psi_T$ is a  unitary character of $\Ha_{\psi}(L)$ which extends  the identity action on $T$. By \cite[Lemma 3.B.3]{BeHa} and \cite[Theorem 4.3.10]{KaTa}, $\rho$ is an irreducible representation of $\Ha_{\psi}(W)$.

Let us go back to $\Ha(W)$. Let $\psi_{L}=\psi_T\circ h_{\psi}$, which is a one-dimensional unitary character of $\Ha(L)$ extended from $\psi$ of $F$. Let us define:
$$\pi'_{\psi}=\Ind_{\Ha(L)}^{\Ha(W)} \psi_L,$$
It can be seen that $\pi'_{\psi}\simeq \rho \circ h_{\psi}$. Hence  $\pi_{\psi}'$ defines a Heisenberg representation of $\Ha(W)$ with central character $\psi$.
\begin{example}\label{laex}
If  $\psi=\psi_0$, $L=\Za e_1 \oplus  \cdots \oplus \Za e_m \oplus \Za e_1^{\ast} \oplus \cdots \oplus \Za e_m^{\ast}$, we let $\mathcal{H}_{\psi}(L)$ be the set of measurable functions $f: W \longrightarrow \C$ such that
\begin{itemize}
\item[(i)] $f(l+w)=\psi(-\tfrac{\langle x_1, x_l^{\ast}\rangle }{2}-\tfrac{\langle l, w\rangle}{2}) f(w)$, for  $l=x_l+x_{l}^{\ast}\in L=(L\cap X)\oplus (L\cap X^{\ast})$, $w\in W$;
\item[(ii)] $\int_{L\setminus W} ||f(w)||^2 dw<+\infty$.
\end{itemize}
Here, we choose a $W$-right invariant measure on $L\setminus W$. Then $\pi_{\psi}'$ can be realized on $\mathcal{H}_{\psi}(L)$ by the following formulas:
$$\pi_{\psi}'([w',t])f(w)=\psi(t+\tfrac{\langle w, w'\rangle}{2})f(w+w')$$
for $w,w'\in W$, $t\in F$.
\end{example}
\section{Extension to $\GSp(W)$}\label{ext}
Our main purpose is to extend  Barthel's results in \cite{Ba}   from the cases of the $p$-adic field  to the real field.
 \subsection{$2$-cocycle I} Let us first consider the $8$-degree covering case. Note that $[\widetilde{c}_{X^{\ast}}]$ is the unique non-trivial class of order $2$ in $\Ha^2(\Sp(W),\mu_8)$, and $\Sp(W)$ is a perfect group.  Let $\alpha$  now be a continuous automorphism of  $\Sp(W)$.  By Moore's cohomology theory, there exists a unique automorphism $\alpha^{\ast}$ of   $\widetilde{\Sp}(W)$ such that the following diagram
\[
\begin{CD}
1 @>>> \mu_8    @>>>  \widetilde{\Sp}(W)@>>>\Sp(W) @>>> 1\\
@.     @|    @VV{\alpha^{\ast}}V  @VV{\alpha}V \\
1 @>>> \mu_8    @>>>  \widetilde{\Sp}(W)@>>>\Sp(W) @>>> 1
\end{CD}
\]
is commutative. Moreover, there exists a unique function $\nu(\alpha,-): \Sp(W) \longrightarrow \mu_8$, such that the following equality holds:
\begin{align}\label{equ}
\widetilde{c}_{X^{\ast}}(g_1^{\alpha}, g_2^{\alpha})=\widetilde{c}_{X^{\ast}}(g_1, g_2)\nu(\alpha,g_1)^{-1}\nu(\alpha,g_2)^{-1}\nu(\alpha,g_1g_2).
\end{align}
The map  $\alpha^{\ast}$ is given as:
$$\alpha^{\ast}:  \widetilde{\Sp}(W) \longrightarrow \widetilde{\Sp}(W); [g, \epsilon] \longmapsto [g, \nu(\alpha, g) \epsilon].$$
\begin{lemma}\label{twoau}
For two automorphisms $\alpha_1, \alpha_2$,  $\nu(\alpha_1\alpha_2, g)=\nu(\alpha_1,g)\nu(\alpha_2, g^{\alpha_1})$.
\end{lemma}
\begin{proof}
Because  $(g, \epsilon)^{\alpha_1\alpha_2}=(  g^{\alpha_1}, \nu( \alpha_1,g) \epsilon)^{\alpha_2}=( g^{\alpha_1\alpha_2}, \nu(\alpha_1,g) \nu(\alpha_2,g^{\alpha_1})\epsilon)$.
\end{proof}
\begin{example}\label{example1}
If $\alpha=h\in  \Sp(W)$, and  $h$ acts on $\Sp(W)$ by  conjugation, i.e., $g^h=h^{-1}gh$, for $g\in \Sp(W)$,  then $\nu(\alpha,g)=\widetilde{c}_{X^{\ast}}(h^{-1}, g h)\widetilde{c}_{X^{\ast}}(g,h)$.
\end{example}
\begin{proof}
By the properties of the cocycles, we obtain:
\begin{equation}
\tfrac{\widetilde{c}_{X^{\ast}}(g_1^h, g_2^h)}{\widetilde{c}_{X^{\ast}}(g_1,g_2)}
=[\widetilde{c}_{X^{\ast}}(h^{-1}, g_1g_2h)\widetilde{c}_{X^{\ast}}(g_1g_2,h)][\widetilde{c}_{X^{\ast}}(h^{-1}, g_1h)\widetilde{c}_{X^{\ast}}(g_1,h)]^{-1} [\widetilde{c}_{X^{\ast}}(h^{-1}, g_2h)\widetilde{c}_{X^{\ast}}(g_2,h)]^{-1}.
\end{equation}
So $\nu(\alpha,g)=\widetilde{c}_{X^{\ast}}(h^{-1}, g h)\widetilde{c}_{X^{\ast}}(g,h)$ by the uniqueness.
\end{proof}
\begin{example}[Barthel]\label{example2}
If $\alpha=\begin{pmatrix}
1& 0\\
0& y
\end{pmatrix}$, for some $y\in F^{\times}$,  then
\begin{align}\label{alpp}
\nu(\alpha,g)=(\det a_1a_2, y)_F \gamma(y, \psi^{\tfrac{1}{2}})^{-|S|},
\end{align} for $g=p_1\omega_S p_2$, $p_i=\begin{pmatrix}
a_i& b_i\\
0& d_i
\end{pmatrix}$, $i=1,2$.
\end{example}
\begin{proof}
We shall follow the proof of  \cite[p.212, Prop.1.2.A]{Ba} in the $p$-adic case.
By (\ref{equ}), there exists a character $\chi_y$ of $F^{\times}$ and a non-zero complex number $\beta_y$ such that
\begin{itemize}
\item $\nu(\alpha,g)= \nu(\alpha,p_1 ) \nu(\alpha, \omega_S ) \nu(\alpha, p_2)$,
\item $   \nu(y, \begin{pmatrix}
a& b\\
0& (a^{\ast})^{-1}
\end{pmatrix})= \chi_y(\det a)$,  and
\item   $\nu(y, \omega_S)=\beta_y^{|S|}$.
\end{itemize}
By the  embedding  $\SL_2\longrightarrow \Sp$, it reduces to discussing the simple case that  $\dim W=2$.  Following \cite[p.212]{Ba}, let $u_-(x)= \begin{pmatrix}
1& 0\\
x& 1
\end{pmatrix}\in \SL_2$. Then:
$$ \begin{pmatrix}
1& 0\\
x& 1
\end{pmatrix}= \begin{pmatrix}
1& x^{-1}\\
0& 1
\end{pmatrix} \begin{pmatrix}
0& -1\\
1& 0
\end{pmatrix} \begin{pmatrix}
x& 1\\
0& x^{-1}
\end{pmatrix} \quad\textrm{ implies }\quad  \nu(y, u_-(x))=\beta_y \chi_y(x).$$  For $s, t, s+t \in F^{\times}$,
\begin{equation}
\begin{split}
\beta_y \chi_y(s+t)=\nu(y, u_-(s+t))&=\nu(y, u_-(s))\nu(y, u_-(t))\tfrac{\widetilde{c}_{X^{\ast}}(u_-(s)^{y},u_-(t)^{y})}{\widetilde{c}_{X^{\ast}}(u_-(s),u_-(t))}\\
& =\nu(y, u_-(s))\nu(y, u_-(t))\tfrac{\widetilde{c}_{X^{\ast}}(u_-(sy^{-1}),u_-(ty^{-1}))}{\widetilde{c}_{X^{\ast}}(u_-(s),u_-(t))}\\
&=\beta_y \chi_y(s)\beta_y \chi_y(t)\tfrac{\gamma(\psi^{\tfrac{1}{2} st(s+t)y^{-3}})}{\gamma(\psi^{\tfrac{1}{2} st(s+t)})}\\
&=\beta_y^2 \chi_y(st)(st(s+t), y)_F\gamma(y, \psi^{\tfrac{1}{2}}).
\end{split}
\end{equation}
Hence:
\begin{align}
\beta_y^{-1} \chi_y(\tfrac{1}{s}+\tfrac{1}{t})=(\tfrac{1}{s}+\tfrac{1}{t}, y)_F  \gamma(y, \psi^{\tfrac{1}{2}}).
\end{align}
Therefore, $\chi_y(s)=(s,y)_F$,  $\beta_y^{-1}=\gamma(y, \psi^{\tfrac{1}{2}})$.
\end{proof}
\subsection{Subgroups and quotient groups of $\GSp(W)$}\label{subquo}
Let $F^{\times }_+$ be the multiplicative subgroup of positive real numbers. Let $\GSp^+(W)=F^{\times}_+\Sp(W)$, $\PGSp(W)=\GSp(W)/F^{\times}$, $\PGSp^{\pm}(W)=\GSp(W)/F_+^{\times}$.
Then:
\begin{itemize}
\item $F_+^{\times}=F^{\times 2}$ and $F^{\times}/F^{\times 2} \simeq \mu_2$.
\item  There exists an exact sequence: $1 \longrightarrow \Sp(W) \longrightarrow \GSp(W) \stackrel{\lambda}{\longrightarrow} F^{\times}\longrightarrow 1.$
\item  There exists an exact sequence: $1\longrightarrow \Sp(W) \longrightarrow \PGSp^{\pm}(W)\stackrel{\dot{\lambda}}{\longrightarrow}   F^{\times}/F^{\times 2} \longrightarrow 1.$
\end{itemize}
Let us choose a section map
\begin{align}\label{ss}
s: F^{\times} \longrightarrow  \GSp(W); y \longmapsto \begin{pmatrix}
1& 0\\
0& y
\end{pmatrix}.
\end{align}
Let $F_2$ denote the image of $\mu_2$. Then: $$\GSp(W) \simeq  F^{\times}\ltimes \Sp(W), \quad \quad \PGSp^{\pm}(W)
\simeq  F_2\ltimes \Sp(W).$$  Let us define $$\Sp^{\pm}(W)=\{ h\in \GSp(W) \mid \lambda_h=\pm 1\}.$$
Then $\Sp^{\pm}(W)\simeq F_2\ltimes \Sp(W)$, which is a  subgroup of $F^{\times} \ltimes \Sp(W)$. By lifting the action of $F^{\times}$ from $\Sp(W)$ to $\widetilde{\Sp}(W)$,  we   obtain a group $F^{\times} \ltimes \widetilde{\Sp}(W)$; let us denote it by $\widetilde{\GSp}(W)$. Then there exists  an exact sequence:
$$ 1\longrightarrow \mu_8 \longrightarrow \widetilde{\GSp}(W) \stackrel{ }{\longrightarrow} \GSp(W) \longrightarrow 1.$$
 Let $\widetilde{C}_{X^{\ast}}$ denote  the $2$-cocycle  associated to this exact sequence from $\widetilde{c}_{X^{\ast}}$. For  $(y_1, g_1, \epsilon_1), (y_2, g_2, \epsilon_2)\in F^{\times} \ltimes \widetilde{\Sp}(W)$,
\begin{equation}
\begin{split}
(y_1, g_1, \epsilon_1)\cdot (y_2, g_2, \epsilon_2)&=(y_1y_2, [g_1, \epsilon_1]^{y_2}[g_2, \epsilon_2])\\
& =(y_1y_2, [g_1^{y_2}, \nu(y_2, g_1)\epsilon_1][g_2, \epsilon_2])\\
&=(y_1y_2, [g_1^{y_2}g_2, \nu( y_2,g_1)\widetilde{c}_{X^{\ast}}(g_1^{y_2}, g_2)\epsilon_1\epsilon_2]).
\end{split}
\end{equation}
Hence:
\begin{align}\label{yg}
\widetilde{C}_{X^{\ast}}([y_1,g_1], [y_2, g_2]) =\nu( y_2,g_1)\widetilde{c}_{X^{\ast}}(g_1^{y_2}, g_2),   \quad\quad [y_i, g_i]\in F^{\times} \ltimes \Sp(W).
\end{align}
By the restriction on $F_2$, we obtain a subgroup $F_2\ltimes \widetilde{\Sp}(W)$; let us denote it $\widetilde{\Sp}^{\pm}(W)$. Then there exists an exact sequence:
\begin{align}
1  \longrightarrow \mu_8 \longrightarrow \widetilde{\Sp}^{\pm}(W) \longrightarrow \Sp^{\pm}(W) \longrightarrow 1.
\end{align}
\begin{lemma}
$\widetilde{C}_{X^{\ast}}(h_1,h_2)=1$, for $ h_i\in F_2 \subseteq \Sp^{\pm}(W)$.
\end{lemma}
\begin{proof}
It follows from the above formula (\ref{yg}).
\end{proof}
\begin{lemma}
$\widetilde{C}_{X^{\ast}}(yI_{2m}, h)=1=\widetilde{C}_{X^{\ast}}(h, yI_{2m})$, for $ y\in F^{\times}_+$, $h\in \GSp(W)$.
\end{lemma}
\begin{proof}
Let us write $yI_{2m}$ and $h$ in the forms of $s(y^2)g_y$ and $s(h) g_h$ by using the section map $s$ as given  in (\ref{ss}), for $g_y=\begin{pmatrix}
y& 0\\
0& y^{-1}
\end{pmatrix}$, $s(h)=\begin{pmatrix}
1& 0\\
0& \lambda_h
\end{pmatrix}$, $g_h=p_1\omega_Sp_2\in \Sp(W)$, $p_i=\begin{pmatrix}
a_i& b_i\\
0& d_i
\end{pmatrix}$, $i=1,2$.  Then:
$$\widetilde{C}_{X^{\ast}}(yI_{2m}, h)=\widetilde{C}_{X^{\ast}}([y^2, g_y], [\lambda_h, g_h])=\nu( \lambda_h, g_y)\widetilde{c}_{X^{\ast}}(g_y^{\lambda_h}, g_h)=(y^{m}, \lambda_h)_F=1,$$
$$\widetilde{C}_{X^{\ast}}(h,yI_{2m})=\widetilde{C}_{X^{\ast}}([\lambda_h, g_h],[y^2, g_y] )=\nu(y^2, g_h)\widetilde{c}_{X^{\ast}}(g_h^{y^2}, g_y)=\nu(y^2, g_h)=(\det a_1a_2, y^2)_F \gamma(y^2, \psi^{\tfrac{1}{2}})^{-|S|}=1.$$
\end{proof}
As a consequence, we can see that $F_+^{\times}$ lies in the center of $\widetilde{\GSp}(W)$ and $\widetilde{\GSp}(W)/F^{\times}_+\simeq \widetilde{\Sp}^{\pm}(W)$.
\begin{lemma}
If $2\nmid m$, $\widetilde{C}_{X^{\ast}}(\begin{pmatrix}
y_1& 0\\
0& 1
\end{pmatrix}, \begin{pmatrix}
y_2& 0\\
0& 1
\end{pmatrix})=(y_1, y_2)_F$, for $y_i\in F^{\times}$.
\end{lemma}
\begin{proof}
Write $\begin{pmatrix}
y_i& 0\\
0& 1
\end{pmatrix}=s(y_i) g_i$,
for $g_i=\begin{pmatrix}
y_i& 0\\
0& y_i^{-1}
\end{pmatrix}.$  Then
\begin{align}
\widetilde{C}_{X^{\ast}}(\begin{pmatrix}
y_1& 0\\
0& 1
\end{pmatrix}, \begin{pmatrix}
y_2& 0\\
0& 1
\end{pmatrix})&=\widetilde{C}_{X^{\ast}}([y_1,g_1], [y_2,g_2])\\
&=\nu( y_2, g_1)\widetilde{c}_{X^{\ast}}(g_1^{y_2}, g_2)\\
&=(y_1^{m}, y_2)_F=(y_1,y_2)_F.
\end{align}
\end{proof}
Notice that $\widetilde{C}_{X^{\ast}}(s(y_1), s(y_2))=1$, for $y_i\in F^{\times}$.
 \begin{remark}\label{splittingc}
 If define $f: F^{\times} \longrightarrow \mu_8; y \longmapsto e^{\tfrac{[1-\sgn(y)] \pi i}{4}}$, then $(y_1, y_2)_F=f(y_1)^{-1}f(y_2)^{-1}f(y_1y_2)$.
 \end{remark}
\subsection{$2$-cocycle II}
Instead of the $8$-degree cover,  let us consider the $2$-degree covering case. For an automorphism $\alpha$ of $\Sp(W)$, let us define $\nu_2: \alpha \times \Sp(W)  \longrightarrow \mu_2$ such that
$$(g, \epsilon)^{\alpha}=(g^{\alpha}, \nu_2( \alpha,g) \epsilon), \qquad  \qquad (g, \epsilon) \in \overline{\Sp}(W).$$
Similarly, $\nu_2$ determines an automorphism of $\overline{\Sp}(W)$ iff the following equality holds:
\begin{equation}\label{crao2}
\overline{c}_{X^{\ast}}(g, g')=\overline{c}_{X^{\ast}}(g^{\alpha},g^{'\alpha})\nu_2(\alpha,g) \nu_2(\alpha,g') \nu_2(\alpha, gg')^{-1}.
\end{equation}
\begin{lemma}\label{nu23}
$\nu_2( \alpha,g )=\nu(\alpha,g) \tfrac{m_{X^{\ast}}(g)}{m_{X^{\ast}}(g^{\alpha})}$.
\end{lemma}
\begin{proof}
By (\ref{28inter}),(\ref{equ}),(\ref{crao2}),
\begin{equation*}
\begin{split}
&\overline{c}_{X^{\ast}}(g, g')m_{X^{\ast}}(gg') m_{X^{\ast}}(g)^{-1}m_{X^{\ast}}(g')^{-1}\\
&=  \widetilde{c}_{X^{\ast}}(g,g')\\
&=\widetilde{c}_{X^{\ast}}(g^{\alpha},g^{'\alpha})\nu(\alpha,g) \nu(\alpha,g') \nu(\alpha,gg')^{-1}\\
&=\overline{c}_{X^{\ast}}(g^{\alpha},g^{'\alpha}) m_{X^{\ast}}(g^{\alpha}g^{'\alpha}) m_{X^{\ast}}(g^{\alpha})^{-1}m_{X^{\ast}}(g^{'\alpha})^{-1}\nu(\alpha,g) \nu(\alpha,g') \nu(\alpha,gg')^{-1}.
\end{split}
\end{equation*}
Hence:
$$\tfrac{\overline{c}_{X^{\ast}}(g, g')}{\overline{c}_{X^{\ast}}(g^{\alpha},g^{'\alpha})}= [\tfrac{m_{X^{\ast}}(gg')}{m_{X^{\ast}}(g^{\alpha}g^{'\alpha})}]^{-1} \tfrac{m_{X^{\ast}}(g)}{ m_{X^{\ast}}(g^{\alpha})}
\tfrac{m_{X^{\ast}}(g')}{ m_{X^{\ast}}(g^{'\alpha})}\nu(\alpha,g) \nu(\alpha,g') \nu(\alpha,gg')^{-1}.$$
So $\nu_2( \alpha,g)=\nu(\alpha,g) \tfrac{m_{X^{\ast}}(g)}{m_{X^{\ast}}(g^{\alpha})}$.
\end{proof}
By lifting the action of $F^{\times}$ from $\Sp(W)$ to $\overline{\Sp}(W)$, we   can obtain a group $F^{\times} \ltimes \overline{\Sp}(W)$ and an exact sequence:
$$ 1\longrightarrow \mu_2 \longrightarrow F^{\times} \ltimes \overline{\Sp}(W) \stackrel{ }{\longrightarrow}F^{\times} \ltimes \Sp(W) \longrightarrow 1.$$
Let us write  $\overline{\GSp}(W)$ for $F^{\times} \ltimes \overline{\Sp}(W)$ and  denote   the corresponding  $2$-cocycle by $\overline{C}_{X^{\ast}}$. Then:
$$ \overline{C}_{X^{\ast}}([y_1,g_1], [y_2, g_2]) =\nu_2(y_2,g_1)\overline{c}_{X^{\ast}}(g_1^{y_2}, g_2),   \quad\quad [y_i, g_i]\in F^{\times} \ltimes \Sp(W).$$
Consequently, substituting the equality in Lemma  \ref{nu23}, we obtain:
\begin{equation}\label{292inter}
\overline{C}_{X^{\ast}}([y_1,g_1], [y_2, g_2])=m_{X^{\ast}}(g_1^{y_2}g_2)^{-1}m_{X^{\ast}}(g_1)m_{X^{\ast}}(g_2)\widetilde{C}_{X^{\ast}}([y_1,g_1], [y_2, g_2]).
\end{equation}
By abuse of notations, let us also write $m_{X^{\ast}}[y_1,g_1]=m_{X^{\ast}}(g_1)$. By the restriction on $F_2$, we obtain a subgroup $F_2\ltimes \overline{\Sp}(W)$; let us denote it $\overline{\Sp}^{\pm}(W)$. Then there exists an exact sequence:
\begin{align}
1  \longrightarrow \mu_2 \longrightarrow \overline{\Sp}^{\pm}(W) \longrightarrow \Sp^{\pm}(W) \longrightarrow 1.
\end{align}
\begin{lemma}
$\overline{C}_{X^{\ast}}(yI_{2m}, h)=1=\overline{C}_{X^{\ast}}(h, yI_{2m})$, for $ y\in F^{\times}_+$, $h\in \GSp(W)$.
\end{lemma}
\begin{proof}
Let  $yI_{2m}=s(y)g_y$ and $h=s(h) g_h$ for $g_y=\begin{pmatrix}
y& 0\\
0& y^{-1}
\end{pmatrix}$, $s(h)=\begin{pmatrix}
1& 0\\
0& \lambda_h
\end{pmatrix}$, $g_h=p_1\omega_Sp_2\in \Sp(W)$, $p_i=\begin{pmatrix}
a_i& b_i\\
0& d_i
\end{pmatrix}$, $i=1,2$.  Then:
\begin{align}
\overline{C}_{X^{\ast}}(yI_{2m}, h)&=\overline{C}_{X^{\ast}}([y, g_y], [\lambda_h, g_h])\\
&=m_{X^{\ast}}(g_y^{\lambda_h}g_h)^{-1}m_{X^{\ast}}(g_y)m_{X^{\ast}}(g_h)\widetilde{C}_{X^{\ast}}([y, g_y], [\lambda_h, g_h])\\
&=m_{X^{\ast}}(g_y^{\lambda_h}g_h)^{-1}m_{X^{\ast}}(g_y)m_{X^{\ast}}(g_h)\\
&=m_{X^{\ast}}(g_yg_h)^{-1}m_{X^{\ast}}(g_y)m_{X^{\ast}}(g_h)\\
&=(x(g_y), x(g_h))_F=(y^{m}, x(g_h))_F=1;
\end{align}
\begin{align}
\overline{C}_{X^{\ast}}(h,yI_{2m})&=\overline{C}_{X^{\ast}}([\lambda_h, g_h],[y, g_y])\\
&=m_{X^{\ast}}(g_h^{y}g_y)^{-1}m_{X^{\ast}}(g_h)m_{X^{\ast}}(g_y)\widetilde{C}_{X^{\ast}}([\lambda_h, g_h], [y, g_y])\\
&=m_{X^{\ast}}(g_h^{y}g_y)^{-1}m_{X^{\ast}}(g_h)m_{X^{\ast}}(g_y)\\
&=(x(g_h^y), x(g_y))_F=(x(g_h^y),y^{m})_F=1.
\end{align}
\end{proof}
As a consequence, $F_+^{\times}$ lies in the center of $\overline{\GSp}(W)$ and $\overline{\GSp}(W)/F^{\times}_+\simeq \overline{\Sp}^{\pm}(W)$.
\subsubsection{Example} Let us consider the case that  $\dim W=2$, $\GSp(W)\simeq \GL_2(F)$.    Let  $h_i=\begin{pmatrix}
a_i& b_i\\
c_i& d_i
\end{pmatrix}\in \GL_2(F)$,  $i=1, 2, 3$, with $h_3=h_1h_2$.    Let us  write $h_i=\begin{pmatrix}
1& 0\\
0& y_i
\end{pmatrix} g_i$, for $y_i=\det h_i$, $g_i=\begin{pmatrix}
a_i& b_i\\
y_i^{-1}c_i& y_i^{-1}d_i
\end{pmatrix}$.   Then:
  $$\overline{C}_{X^{\ast}}(h_1, h_2) =\nu_2(y_2,g_1)\overline{c}_{X^{\ast}}(g_1^{y_2}, g_2)=\nu_2(\det h_2,g_1)\overline{c}_{X^{\ast}}(g_1^{\det h_2}, g_2).$$
\begin{itemize}
\item[(1)] If $c_1=0$, $m_{X^{\ast}}(g_1)=m_{X^{\ast}}(g_1^{y_2})$, so $\nu_2(y_2, g_1)=\nu(y_2, g_1)=(y_2,  a_1)_F$.
\item[(2)]  If $c_1\neq 0$, then:
\begin{itemize}
\item[] $m_{X^{\ast}}(g_1)=\gamma(c_1 y_1^{-1}, \psi^{\tfrac{1}{2}})^{-1} \gamma(\psi^{\tfrac{1}{2}})^{-1}=\gamma({\psi}^{\tfrac{1}{2} c_1 y_1^{-1}})^{-1}$,
 \item[] $m_{X^{\ast}}(g_1^{y_2})=\gamma({\psi}^{\tfrac{1}{2} c_1 y_1^{-1}y_2^{-1}})^{-1}$,
\item[] $\tfrac{m_{X^{\ast}}(g_1)}{  m_{X^{\ast}}(g_1^{y_2})}=\gamma(y^{-1}_2, \psi^{\tfrac{1}{2}}) (y_1c_1,y_2)_F$,
 \item[] $\nu_2(y_2, g_1)=\nu(y_2, g_1)\tfrac{m_{X^{\ast}}(g_1)}{  m_{X^{\ast}}(g_1^{y_2})}=(y_1 c_1, y_2)_F\gamma(y_2, \psi^{\tfrac{1}{2}})^{-1}\gamma(y^{-1}_2, \psi^{\tfrac{1}{2}}) (y_1c_1,y_2)_F=1$.
      \end{itemize}
  \end{itemize}
\section{Extended Weil representations}\label{exWe}
\subsection{}\label{EWR}
Recall that $\pi_{\psi}$ is a Weil representation of $\widetilde{\Sp}(W)\ltimes \Ha(W)$ associated to $\psi$, and $\omega_{\psi}=\pi_{\psi}|_{\widetilde{\Sp}(W)}$. Let $\widetilde{\GSp}(W)$ act on $\Ha(W)$ via $\GSp(W)$.
\begin{lemma}
$\omega_{\psi} \simeq \omega_{\psi^{h^2}}$, for any $h\in F^{\times}$, and $\omega_{\psi} \ncong \omega_{\psi^{-1}}$.
\end{lemma}
\begin{proof}
We will follow the proof in \cite[p.36]{MoViWa} in the $p$-adic case.  Let $\widetilde{h}=[h,1]\in \widetilde{\GSp}(W)$, and $\widetilde{g}=[g, \epsilon]\in \widetilde{\Sp}(W)$,  $w_t=(w, t)\in \Ha(W)$. Then: $$\pi_{\psi}(w_t^{\widetilde{h}})=\pi_{\psi} (wh, t\lambda_h) \simeq \pi_{\psi^{ \lambda_h}}( w_t);$$
it implies that
$$\pi_{\psi}([\widetilde{g},w_t]^{\widetilde{h}})=\pi_{\psi} ([\widetilde{g}^{\widetilde{h}}, (wh, t\lambda_h)]) \simeq \pi_{\psi^{ \lambda_h}}([\widetilde{g}, w_t]).$$
In particular, by choosing  $h\in F_+^{\times}$, we obtain:
$$\widetilde{g}^{\widetilde{h}}=\widetilde{g},$$
$$\omega_{\psi} (\widetilde{g}) =\pi_{\psi} (\widetilde{g}) \simeq \pi_{\psi^{ h^2}}(\widetilde{g})=\omega_{\psi^{h^2}} (\widetilde{g}).$$
The second statement can derive  from the character calculations of $\omega_{\psi}$ and $\omega_{\psi^{-1}}$ in \cite[Thm.1C]{Th}.
\end{proof}
Let us define:
$$\Pi_{\psi}=\Ind_{\widetilde{\Sp}(W)\ltimes \Ha(W)}^{\widetilde{\Sp}^{\pm}(W)\ltimes \Ha(W)} \pi_{\psi}, \qquad \Omega_{\psi}=\Res_{\widetilde{\Sp}^{\pm}(W)}^{\widetilde{\Sp}^{\pm}(W)\ltimes \Ha(W)}\Pi_{\psi}.$$
By Clifford-Mackey theory, $\Omega_{\psi}\simeq \Ind_{\widetilde{\Sp}(W)}^{\widetilde{\Sp}^{\pm}(W)} \omega_{\psi}$; its restriction on $\widetilde{\Sp}(W)$ consists of the two Weil representations $\omega_{\psi}$ and $\omega_{\psi^{-1}}$. Therefore, $\Omega_{\psi}$ is independent of the choice of $\psi$. Let us  denote it by $\Omega$, and call it the Weil representation of $\widetilde{\Sp}^{\pm}(W)$.

Note that as a representation of $\Ha(W)$, $\Pi_{\psi}\simeq  \pi_{\psi}\oplus  \pi_{\psi^{-1}}$, which is not irreducible.  Let
$\Ha^{\pm}(W)=F_2\ltimes \Ha(W)$. Then $\Pi_{\psi}$ is an irreducible representation of $\Ha^{\pm}(W)$.
\subsection{Schr\"odinger model}\label{sch}
Retain notations from Section \ref{Sch}. Then:
$$ \Pi_{\psi}\simeq \Ind_{X^{\ast}\times F}^{\Ha^{\pm}(W)}\psi_{X^{\ast}}.$$
So $\Pi_{\psi}$ can be  realized on $L^2( \mu_2 \times X )$ by the following formulas:
\begin{equation}\label{representationsp21}
\Pi_{\psi}[(x,0)+(x^{\ast},0)+(0,k)]f([\epsilon,y])=\psi(\epsilon k+\epsilon \langle x+ y,x^{\ast}\rangle) f([\epsilon, x+y]),
\end{equation}
\begin{equation}\label{representationsp22}
\Pi_{\psi}(s(-1))f([\epsilon,y])=f([-\epsilon, y]),
\end{equation}
for $x, y\in X$, $x^{\ast}\in X^{\ast}$, $k\in F$, $\epsilon\in \mu_2$.\\
Let us prove the results:
$$(x,0)+(x^{\ast},0)+(0,k)=(x, x^{\ast}; k+\tfrac{\langle x, x^{\ast}\rangle}{2})\in \Ha(W);$$
\begin{align*}
\Pi_{\psi}[(x,0)+(x^{\ast},0)+(0,k)]f(\epsilon, y)&=f([\epsilon, y]\cdot[1, (x, x^{\ast}; k+\tfrac{\langle x, x^{\ast}\rangle}{2})])\\
     &=f([\epsilon, (x+y, x^{\ast}; k+\tfrac{\langle x+y, x^{\ast}\rangle}{2})]\\
     &=f([1,(\epsilon x^{\ast}, \epsilon k+\epsilon \langle x+y, x^{\ast}\rangle)][\epsilon, x+y])\\
     &=\psi(\epsilon k+\epsilon \langle x+ y,x^{\ast}\rangle) f([\epsilon, x+y]).
     \end{align*}
Following \cite[p.214]{Ba}, let us extend the above formulas (\ref{representationsp21}),(\ref{representationsp22}) to $\widetilde{\Sp}^{\pm}(W)$. Let us identity $L^2(\mu_2 \times X)$ with $\Pi_{\psi}(F_2)L^2(X)$. For any $f\in L^2(\mu_2\times X)$, let us write $f([\epsilon, x])=\left\{ \begin{array}{lr} f_1(x) & \textrm{ if }\epsilon=1,\\
\Pi_{\psi}(s(-1))f_2(x) &  \textrm{ if } \epsilon=-1,\end{array}\right.$ where   $s: \mu_2 \longrightarrow F_2$, given  from (\ref{ss}). For $(1, g)\in \Sp^{\pm}(W)$, we have:
\begin{equation}\label{representationsp22221}
\begin{aligned}
([1, g],1) ([\epsilon, 1], 1)&=([\epsilon, g^{s(\epsilon)}], \widetilde{C}_{X^{\ast}}([1, g], [\epsilon, 1]))\\
&=([\epsilon, g^{s(\epsilon)}], \nu( \epsilon,g)\widetilde{c}_{X^{\ast}}(g^{s(\epsilon)}, 1))\\
&=([\epsilon, g^{s(\epsilon)}], \nu( \epsilon,g))\\
&=([\epsilon, 1], 1)([1, g^{s(\epsilon)}],\nu( \epsilon,g)).
\end{aligned}
\end{equation}
1) If $g= \begin{pmatrix}
  1&b\\
  0 & 1
\end{pmatrix}$, then $\nu( 1,g)=1=\nu( -1,g)$.
\begin{align*}\label{representationsp2221}
\Pi_{\psi}(g)f([1, y])&=\pi_{\psi}( g)f_1(y)=\psi(\tfrac{1}{2}\langle y,yb\rangle) f_1(y)\\
&=\psi(\tfrac{1}{2}\langle y,yb\rangle) f([1,y]),
\end{align*}
\begin{equation*}\label{representationsp2222}
\begin{aligned}
\Pi_{\psi}(g)f([-1, y])& =\Pi_{\psi}([-1, 1])\pi_{\psi}( g^{s(-1)})f_2(y)\\
 &=\psi(\tfrac{1}{2}\langle y,-yb\rangle)\Pi_{\psi}([-1, 1], 1)f_2(y)\\&=\psi(\tfrac{1}{2}\langle y,-yb\rangle) f([-1,y]).
\end{aligned}
\end{equation*}
Hence:
\begin{equation}\label{representationsp23}
\Pi_{\psi}[\begin{pmatrix}
  1&b\\
  0 & 1
\end{pmatrix}]f([\epsilon, y])=\psi(\tfrac{1}{2}\langle y,\epsilon yb\rangle) f([\epsilon,y]).
\end{equation}
2) If $g= \begin{pmatrix}
  a& 0\\
  0 &a^{\ast -1 }
\end{pmatrix}$,then $\nu( 1,g)=1$, $\nu( -1,g)=(\det a, -1)_F$.
\begin{align*}\label{representationsp2231}
\Pi_{\psi}( g)f([1, y])=\pi_{\psi}(g)f_1(y)=|\det(a)|^{1/2} f_1(ya)=|\det(a)|^{1/2} f([1,ya]).
\end{align*}
\begin{equation*}\label{representationsp2232}
\begin{aligned}
\Pi_{\psi}(g)f([-1, y])&=\Pi_{\psi}([-1, 1])\pi_{\psi}[ g^{s(-1)}, \nu( -1,g)]f_2(y)\\
&=|\det(a)|^{1/2}(\det a, -1)_F\Pi_{\psi}([-1, 1], 1) f_2(ya)\\&=|\det(a)|^{1/2}(\det a, -1)_F f([-1,ya]).
\end{aligned}
\end{equation*}
Hence:
\begin{equation}\label{representationsp24}
\Pi_{\psi}[ \begin{pmatrix}
  a& 0\\
  0 &a^{\ast -1 }
\end{pmatrix}]f([\epsilon, y])=|\det(a)|^{1/2} (\det a,\epsilon)_Ff([\epsilon,ya]).
\end{equation}
3) If $g= \omega$,  then $\nu( 1,g)=1$, $\nu( -1,g)=\gamma(-1, \psi^{\tfrac{1}{2}})^{-m}$. Let $f_1, f_2\in S(X)$.
\begin{align*}\label{representationsp2231}
\Pi_{\psi}( g)f([1, y])=\pi_{\psi}(g)f_1(y)=\int_{X^{\ast}} \psi(\langle y, y^{\ast}\rangle) f_1(y^{\ast}\omega^{-1}) dy^{\ast}.
\end{align*}
\begin{equation*}\label{representationsp22321}
\begin{aligned}
\Pi_{\psi}(g)f([-1, y])&=\Pi_{\psi}([-1, 1])\pi_{\psi}[ -g, \nu( -1,g)]f_2(y)\\
&=\Pi_{\psi}([-1, 1]) \nu( -1,g)\int_{X^{\ast}} \psi(\langle -y, y^{\ast}\rangle) f_2(y^{\ast}\omega^{-1}) dy^{\ast}\\
&=\nu( -1,g)\int_{X^{\ast}} \psi(\langle -y, y^{\ast}\rangle) f([-1,y^{\ast}\omega^{-1}]) dy^{\ast}.
\end{aligned}
\end{equation*}
Hence:
\begin{equation}\label{representationsp25}
\Pi_{\psi}(\omega)f([\epsilon, y])=\nu(\epsilon,g)\int_{X^{\ast}} \psi(\langle \epsilon y, y^{\ast}\rangle) f([\epsilon,y^{\ast}\omega^{-1}]) dy^{\ast}.
\end{equation}
By \cite{We}, the above actions determine this representation.
\subsection{The cocycle}  Let  $h_i=[\epsilon_i, g_i]\in \Sp^{\pm}( W)$, $f\in L^2(\mu_2 \times X)$ with $f|_{X}=f_1$, $f|_{[-1\times X]}=\Pi_{\psi}([-1, 1])f_2$, for  $f_1, f_2\in L^2(X)$. Then:
\begin{equation}\label{h1}
\begin{aligned}
\Pi_{\psi}(h_1)\Pi_{\psi}(h_2)f_1&=\Pi_{\psi}([\epsilon_1, 1])\Pi_{\psi}([1,g_1])\Pi_{\psi}([\epsilon_2, 1])\Pi_{\psi}([1,g_2])f_1\\
&=\Pi_{\psi}([\epsilon_1, 1])\Pi_{\psi}([1,g_1])\Pi_{\psi}([\epsilon_2, 1])\pi_{\psi}(g_2)f_1\\
&=\Pi_{\psi}([\epsilon_1, 1])\Pi_{\psi}([\epsilon_2,g_1^{s(\epsilon_2)}]) \nu(\epsilon_2, g_1)\pi_{\psi}(g_2)f_1\\
&=\Pi_{\psi}([\epsilon_1, 1])\Pi_{\psi}([\epsilon_2, 1])\Pi_{\psi}([1, g_1^{s(\epsilon_2)}]) \nu(\epsilon_2, g_1)\pi_{\psi}(g_2)f_1\\
&=\Pi_{\psi}([\epsilon_1, 1])\Pi_{\psi}([\epsilon_2, 1])\pi_{\psi}( g_1^{s(\epsilon_2)}) \nu(\epsilon_2, g_1)\pi_{\psi}(g_2)f_1\\
&=\Pi_{\psi}([\epsilon_1\epsilon_2, 1])\nu(\epsilon_2, g_1) \widetilde{c}_{X^{\ast}}( g_1^{s(\epsilon_2)}, g_2) \pi_{\psi}( g_1^{s(\epsilon_2)}g_2)f_1\\
&=\widetilde{C}_{X^{\ast}}( [\epsilon_1,g_1], [\epsilon_2,g_2]) \Pi_{\psi}([\epsilon_1\epsilon_2,  g_1^{s(\epsilon_2)}g_2]) f_1\\
&=\widetilde{C}_{X^{\ast}}( h_1,h_2)\Pi_{\psi}(h_1h_2) f_1.
\end{aligned}
\end{equation}
\begin{equation}\label{h2}
\begin{aligned}
&\quad \Pi_{\psi}(h_1)\Pi_{\psi}(h_2)\Pi_{\psi}([-1, 1])f_2\\
&=\Pi_{\psi}([\epsilon_1, 1])\Pi_{\psi}([1,g_1])\Pi_{\psi}([\epsilon_2, 1])\Pi_{\psi}([1,g_2])\Pi_{\psi}([-1, 1])f_2\\
&=\Pi_{\psi}([\epsilon_1, 1])\Pi_{\psi}([1,g_1])\Pi_{\psi}([\epsilon_2, 1])\Pi_{\psi}([-1,g_2^{s(-1)}])\nu(-1, g_2)f_2\\
&=\Pi_{\psi}([\epsilon_1, 1])\Pi_{\psi}([1,g_1])\Pi_{\psi}([\epsilon_2, 1])\Pi_{\psi}([-1,1])\pi_{\psi}(g_2^{s(-1)})\nu(-1, g_2)f_2\\
&=\Pi_{\psi}([\epsilon_1, 1])\Pi_{\psi}([\epsilon_2,g_1^{s(\epsilon_2)}]) \nu(\epsilon_2, g_1) \Pi_{\psi}([-1,1]) \nu(-1, g_2)\pi_{\psi}(g_2^{s(-1)})f_2\\
&=\Pi_{\psi}([\epsilon_1, 1])\Pi_{\psi}([\epsilon_2, 1]) \nu(\epsilon_2, g_1) \Pi_{\psi}([1, g_1^{s(\epsilon_2)}]) \Pi_{\psi}([-1,1]) \nu(-1, g_2)\pi_{\psi}(g_2^{s(-1)})f_2\\
&=\Pi_{\psi}([\epsilon_1\epsilon_2, 1])  \nu(\epsilon_2, g_1) \Pi_{\psi}([-1, g_1^{s(-\epsilon_2)}]) \nu(-1, g_1^{s(\epsilon_2)}) \nu(-1, g_2)\pi_{\psi}(g_2^{s(-1)})f_2\\
&=\Pi_{\psi}([\epsilon_1\epsilon_2, 1])  \nu(\epsilon_2, g_1) \Pi_{\psi}([-1, 1]) \pi_{\psi}(g_1^{s(-\epsilon_2)}) \nu(-1, g_1^{s(\epsilon_2)}) \nu(-1, g_2)\pi_{\psi}(g_2^{s(-1)})f_2\\
&=\Pi_{\psi}([\epsilon_1\epsilon_2, 1])\nu(\epsilon_2, g_1)\nu(-1, g_1^{s(\epsilon_2)}) \nu(-1, g_2) \widetilde{c}_{X^{\ast}}( g_1^{s(-\epsilon_2)}, g_2^{s(-1)})\Pi_{\psi}([-1, 1])  \pi_{\psi}( g_1^{s(-\epsilon_2)}g_2^{s(-1)})f_2\\
&=\Pi_{\psi}([\epsilon_1\epsilon_2, 1])\nu(\epsilon_2, g_1)\nu(-1, g_1^{s(\epsilon_2)}) \nu(-1, g_2) \widetilde{c}_{X^{\ast}}( g_1^{s(-\epsilon_2)}, g_2^{s(-1)})\nu(-1,g_1^{s(\epsilon_2)}g_2)^{-1}  \pi_{\psi}( g_1^{s(\epsilon_2)}g_2)\Pi_{\psi}([-1, 1]) f_2\\
&=\Pi_{\psi}([\epsilon_1\epsilon_2, 1])\nu(\epsilon_2, g_1)\widetilde{c}_{X^{\ast}}( g_1^{s(\epsilon_2)}, g_2) \pi_{\psi}( g_1^{s(\epsilon_2)}g_2)\Pi_{\psi}([-1, 1]) f_2\\
&=\widetilde{C}_{X^{\ast}}( [\epsilon_1,g_1], [\epsilon_2,g_2]) \Pi_{\psi}([\epsilon_1\epsilon_2,  g_1^{s(\epsilon_2)}g_2])\Pi_{\psi}([-1, 1]) f_2 \\
&=\widetilde{C}_{X^{\ast}}( h_1,h_2)\Pi_{\psi}(h_1h_2) \Pi_{\psi}([-1, 1]) f_2.
\end{aligned}
\end{equation}
By (\ref{h1}), (\ref{h2}), for $f\in L^2(\mu_2\times X)$,  we have:
\begin{equation}\label{Cdou}
 \Pi_{\psi}(h_1)\Pi_{\psi}(h_2) f=\widetilde{C}_{X^{\ast}}( h_1,h_2)\Pi_{\psi}(h_1h_2) f.
 \end{equation}
\subsection{Lattice model}\label{la}
Retain notations from Section \ref{latticemodel}. Then $\Pi_{\psi}=\Ind_{\Ha(L)}^{\Ha^{\pm}(W)} \psi_L$.
For example, if  $\psi=\psi_0$, we take $L=\Za e_1 \oplus  \cdots \oplus \Za e_m \oplus \Za e_1^{\ast} \oplus \cdots \oplus \Za e_m^{\ast}$.  Let $\mathcal{H}_{\psi}^{\pm}(L)$ be the set of  measurable functions $f: \mu_2  \times W \longrightarrow \C$ such that
\begin{itemize}
\item[(i)] $f(\epsilon, l+w)=\psi(-\tfrac{\langle x_{l}, x^{\ast}_l\rangle}{2}-\epsilon\tfrac{\langle l, w\rangle}{2}) f(\epsilon, w)$, for  $l=x_l+x_{l}^{\ast}\in L=(X\cap L) \oplus (X^{\ast}\cap L)$, $w\in W$,
\item[(ii)] $\int_{L\setminus W} ||f(\epsilon, w)||^2 dw<+\infty$,
\end{itemize}
where $w\in W$, $\epsilon \in \mu_2$, and $dw$ is a $W$-right invariant measure on $L\setminus W$.  Then $\Pi_{\psi}$ can be realized on $\mathcal{H}_{\psi}^{\pm}(L)$ by the following formulas:
\begin{equation}\label{www1}
\Pi_{\psi}([w',t])f( \epsilon, w)=\psi(t+\epsilon\tfrac{\langle w, w'\rangle}{2})f( \epsilon, w+w'),
 \end{equation}
\begin{equation}\label{www2}
\Pi_{\psi}[-1,0]f( \epsilon, x+x^{\ast})=f(- \epsilon, x-x^{\ast}),
 \end{equation}
for $w=x+x^{\ast},w'=x'+x^{'\ast}\in W$, $\epsilon \in \mu_2$, $t\in F$.
\section{Splitting of dual pairs  in $\Sp^{\pm}(W)$}\label{splitting}
\subsection{Notations} Following \cite{MoViWa} and \cite{Ku1}, let $F=\R$, $E=\C$. Let $\mathbb{H}$ denote the unique(up to isometry) quaternion algebra over $F$, with a basis $\{1, i,j,k\}$ such that $ij=-ji=k$, $i^2=j^2=k^2=-1$. Assume $E=F(i)$. Let $D$ be a division algebra over $F$ with an involution $\tau$ such that  $F$ consists of all $\tau$-fixed points of $D$. According to \cite{MoViWa} and \cite{Ku1}, it has  one of the following forms,  up to isometry:
\begin{itemize}
\item  $D=F, \tau=\id$;
\item $D=E$,  $\tau$=the complex conjugation;
\item $D= \mathbb{H}$, $\tau=$ the quaternion conjugation.
\end{itemize}
Assume that  $(W, \langle, \rangle)$ has the following  tensor product decomposition:
$$W=\mathcal{W}_1 \otimes_D \mathcal{W}_2, \qquad \langle, \rangle=\kappa_D\Trd_{D/F}\big( \langle, \rangle_1 \otimes \tau(\langle, \rangle_2)\big)$$
for a right $\epsilon_1$-hermitian space $(\mathcal{W}_1, \langle, \rangle_1)$ over $D$, a left $\epsilon_2$-hermitian space $(\mathcal{W}_2, \langle, \rangle_2)$ over $D$ with  $\epsilon_1 \epsilon_2=-1$, and  $\kappa_D=\left\{\begin{array}{cc} 1& \textrm{ if } D=F,\\
\tfrac{1}{2} & \textrm{ if } D=E,\mathbb{H}.
 \end{array}\right.$ It has the following different cases that correspond to  the above forms:
\begin{itemize}
\item[(1)][Symplectic-orthogonal type]  $D=F$, $\mathcal{W}_1$ (resp. $\mathcal{W}_2$) is a  symplectic(resp. a symmetric)  vector space over $F$ and vice versa.
\item[(2)][Unitary type] $D=E$,  $\mathcal{W}_1$ (resp. $\mathcal{W}_2$) is a  hermitian (resp.  skew hermitian )  vector space over $E$ and  vice versa.
\item[(3)][Quaternion unitary type] $D= \mathbb{H}$, $\mathcal{W}_1$ (resp. $\mathcal{W}_2$) is a  hermitian (resp. skew hermitian )  vector space over $D$ and  vice versa.
\end{itemize}
 Let $\U(\mathcal{W}_1)$(resp.$\U(\mathcal{W}_2)$) denote  the group of isometries of $(\mathcal{W}_1, \langle, \rangle)$(resp.$(\mathcal{W}_2, \langle, \rangle)$), and $\GU(\mathcal{W}_1)$ (resp.$\GU(\mathcal{W}_2)$) the group of   similitude isometries  of $(\mathcal{W}_1, \langle, \rangle)$(resp.$(\mathcal{W}_2, \langle, \rangle)$). By \cite[p.15]{MoViWa},   except one case in the above   (3) where the skew hermitain vector space has dimension $1$,   the pair $\big(\U(\mathcal{W}_1), \U(\mathcal{W}_2)\big)$ is  an irreducible \emph{dual reductive pair}  of type I in the sense of Howe.
   Let $$1 \longrightarrow T\longrightarrow \widehat{\Sp}(W) \longrightarrow \Sp(W) \longrightarrow 1$$
     be a Metaplectic central topological extension of $\Sp(W)$ by $T$ associated to $\widetilde{c}_{X^{\ast}} $.  We now let  $\widehat{\U}(\mathcal{W}_1)$, $\widehat{\U}(\mathcal{W}_2)$ be the  preimages of $\U(\mathcal{W}_1)$,  $\U(\mathcal{W}_2)$ in $\widehat{\Sp}(W)$ respectively. To state a result  about the irreducible dual reductive pair from \cite{MoViWa}, let us list  separately a  case:
\begin{itemize}
\item[(1*)] [Exception case] In the above symplectic-orthogonal type, the  orthogonal vector space has odd dimension.
\end{itemize}
\begin{theorem}\label{scindagedugroupeR0}
Let $(\U(\mathcal{W}_1), \U(\mathcal{W}_2))$ be an  irreducible reductive dual pair in $\Sp(W)$ as given above.
\begin{itemize}
\item[(1)] If $(\mathcal{W}_1, \mathcal{W}_2)$  is not the exception case (1*), then the exact sequence $1 \longrightarrow T \longrightarrow \widehat{\U}(\mathcal{W}_i) \longrightarrow \U(\mathcal{W}_i)\longrightarrow 1$ splits, for $i=1, 2$.
\item[(2)]  If $(\mathcal{W}_1, \mathcal{W}_2)$  is  the exception case and $\epsilon_1=-1$, $\epsilon_2=1$, then
\begin{itemize}
 \item the exact sequence $1 \longrightarrow T \longrightarrow \widehat{\Oa}(\mathcal{W}_2) \longrightarrow \Oa(\mathcal{W}_2)\longrightarrow 1$ splits, but
  \item the exact sequence $1 \longrightarrow T \longrightarrow \widehat{\Sp}(\mathcal{W}_1) \longrightarrow \Sp(\mathcal{W}_1)\longrightarrow 1$  does not split.
  \end{itemize}
    \end{itemize}
\end{theorem}
The $p$-adic version of this theorem  has been proved in  \cite[Chapter 1 ]{MoViWa}   and  \cite{Ku1}. For the   real field case,  it was   proved in \cite[Prop.4.1]{Ku1}. Moreover,  explicit trivialization is also given in \cite{Ku1}. 
\begin{theorem}\label{scindagedugroupeR4}
Let $(\U(\mathcal{W}_1), \U(\mathcal{W}_2))$ be an  irreducible reductive dual pair in $\Sp(W)$ as given above.  Then $\widehat{\U}(\mathcal{W}_1)$  commutates with  $\widehat{\U}(\mathcal{W}_2)$ in $\widehat{\Sp}(W)$.
\end{theorem}
The $p$-adic version of this theorem  has been proved in  \cite[Chapter 1 ]{MoViWa}. For the   real field case,  one can see \cite{Ho1,Ho2}.  Recall
$$\iota: \U(\mathcal{W}_1) \times \U(\mathcal{W}_2) \longrightarrow \Sp(W).$$
The kernel of this map is trivial or $\langle \pm I\rangle$.  By Theorem \ref{scindagedugroupeR4}, there exists a group homomorphism:
\begin{align}
\widehat{\iota}:& \widehat{\U}(\mathcal{W}_1) \times \widehat{\U}(\mathcal{W}_2) \longrightarrow \widehat{\Sp}(W)\\
& (\widehat{u}_1, \widehat{u}_2) \longmapsto \widehat{u}_1\widehat{u}_2.
\end{align}
The image of this map  belongs to $\widehat{\iota(\U(\mathcal{W}_1)\times \U(\mathcal{W}_2))}$. If $(\mathcal{W}_1, \mathcal{W}_2)$  is a pair in the above theorem \ref{scindagedugroupeR0}(1),  then  there exists a composite  group homomorphism:
$$ \U(\mathcal{W}_1) \times U(\mathcal{W}_2) \longrightarrow \widehat{\U}(\mathcal{W}_1) \times \widehat{\U}(\mathcal{W}_2) \longrightarrow \widehat{\Sp}(W).$$
Since the restriction of $\widetilde{c}_{X^{\ast}}  $ on $\ker(\iota)\times \Sp(W)$ or $\Sp(W)\times \ker(\iota)$ is trivial,  the composite map  is trivial on $\ker(\iota)$. Hence $\widehat{\iota(\U(\mathcal{W}_1)\times \U(\mathcal{W}_2))}$ is a splitting group. For simplicity, we can say that the exact sequence
$$1 \longrightarrow T \longrightarrow \widehat{\U(\mathcal{W}_1)\times \U(\mathcal{W}_2)} \longrightarrow \U(\mathcal{W}_1)\times \U(\mathcal{W}_2) \longrightarrow 1$$
is splitting.
In the following, we will  extend these results to the dual pairs in $\Sp^{\pm}(W)$.
\subsubsection{The group $\U^{\pm}(\mathcal{W}_i)$}
Let $\U^{\pm}(\mathcal{W}_i)=\{ g\in \GU(\mathcal{W}_i)\mid \lambda_g=\pm 1\}$.
\begin{lemma}
If $\mathcal{W}_i$ is a hyperbolic space over $D$ of even dimension, then there exists an exact sequence:
$1\longrightarrow \U(\mathcal{W}_i) \longrightarrow \U^{\pm}(\mathcal{W}_i) \longrightarrow \mu_2 \stackrel{\lambda}{\longrightarrow} 1$.
\end{lemma}
\begin{proof}
Assume $i=2$. Assume that  $\mathcal{W}_2 \simeq n_2H$, for  a hyperbolic plane $H$ over $D$. Let  $H= Y\oplus Y^{\ast}$ be a complete polarisation. Then  we can  define a section map
 $$s: \mu_2 \longrightarrow \U(\mathcal{W}_2); a \longmapsto   \underbrace{h_a \times \cdots \times h_a}_{n_2 }$$
 where $h_a=\left(\begin{array}{ccccccc}
1& 0 \\
0 & a \end{array}\right)\in  \GU(H)$.
\end{proof}
Let $\Lambda_{\U^{\pm}(\mathcal{W}_i) }$ be the image of $\U^{\pm}(\mathcal{W}_i)$ in $\mu_2$ under the map $\lambda$.  Assume that $\mathcal{W}_i=\mathcal{W}_i^0\oplus \mathcal{W}_i^1$ and $\mathcal{W}_i^0 \bot \mathcal{W}_i^1$, for an anisotropic vector space $\mathcal{W}_i^0$ and a hyperbolic subspace  $\mathcal{W}_i^1 \simeq n_i H$.
\begin{lemma}\label{sp}
$\Lambda_{\U^{\pm}(\mathcal{W}_i) }=\Lambda_{\U^{\pm}(\mathcal{W}^0_i) }$.
\end{lemma}
\begin{proof}
Assume $\mathcal{W}_i$ is a right vector space over $D$. For  $g\in \U^{\pm}(\mathcal{W}_i)$, the action of  $g$ on  $\mathcal{W}_i$  will yield another Witt decomposition:  $\mathcal{W}_i=g \mathcal{W}_i^0\oplus g\mathcal{W}_i^1$. By Witt's Theorem,  $g \cdot \mathcal{W}_i^0   = g_0  \cdot \mathcal{W}_i^0$ for some suitable  $g_0 \in \U(\mathcal{W}_i)$. Moreover,  $  g_0^{-1}g\cdot (\mathcal{W}_i^0) =\mathcal{W}_i^0$. So  $ g_0^{-1}g \in \GU(\mathcal{W}_i^0)$, and  $\lambda(g_0^{-1}g)=\lambda(g)$. This shows that $\Lambda_{\U^{\pm}(\mathcal{W}_i)} \subseteq \Lambda_{\U^{\pm}(\mathcal{W}^0_i)}$. On the other hand,  by the above lemma, $\Lambda_{\U^{\pm}(\mathcal{W}_i^1)} =\{\pm1\}$.
  Hence,   for $h_0 \in \U^{\pm}(\mathcal{W}^0_i)$ with $\lambda=\lambda(h_0) \in F^{\times}$,  we can find an element   $g_H \in \U^{\pm}(\mathcal{W}_i^1)$ satisfying  $\lambda(g_H)=\lambda$. Then   $g:=h_0 \times g_H $,  viewed as an element of $\U^{\pm}(\mathcal{W}_i)$, satisfies  $\lambda(g)=\lambda(h_0)$. Hence $\Lambda_{\U^{\pm}(\mathcal{W}_i^0)} \subseteq \Lambda_{\U^{\pm}(\mathcal{W}_i)}$.
\end{proof}
By this lemma,  we can determine the image of $\lambda$ in $\mu_2$ by means of  the characteristic of  the anisotropic subspace of $\mathcal{W}_i$.  The following result is from \cite[p.7]{MoViWa}.
\begin{lemma}\label{vigneraslemma1}
  Up to isometry:
\begin{itemize}
\item[-]  an anisotropic quadratic  vector space over $F$ has the following form: $nF(1)$, and $-nF(1)$,   for $ 1\leq n\leq 4$, where $F(1)$ is a quadratic vector space over $F$ of dimension $1$ with the form $x\longrightarrow x^2$;
\item[-]  an anisotropic hermitian  vector space over $E$ has the following form:  $nE$, and $-n E$, for $1\leq n\leq 2$, where  the hermitian form on $E$ is given by $(x ,y)\longmapsto \tau(x)y$,  for $x,y\in E$;
\item[-] an anisotropic right hermitian  vector space over $\mathbb{H}$ has the following form: $\mathbb{H}$, with the form $(x,y) \longmapsto \tau(x)y$.
\end{itemize}
\end{lemma}
Suppose that   $(V, \langle, \rangle)$ is  an anisotropic right hermitian space over $E$. Then we can  take an element $i$ of $E^{\times}$ such that $\overline{i}/{i}=-1$. As it is known that  multiplication of $\langle, \rangle$ by $i$
  will give a skew right hermitian form $i \langle, \rangle$ on $V$. In analogy with  Lemma \ref{vigneraslemma1}, we have:
\begin{lemma}\label{vigneraslemma2}
  Up to isometry:
\begin{itemize}
\item[-]  an anisotropic right skew hermitian  vector space over $E$ has the following form:  $nE$, and $-n E$, for $1\leq n\leq 2$, where  the  skew hermitian form on $E$ is given by $(x ,y)\longmapsto \tau(x)iy$,  for $x,y\in E$;
\item[-] an anisotropic right skew hermitian  vector space over $\mathbb{H}$ has the following form: $\mathbb{H}$, with the form $(x,y) \longmapsto \tau(x)iy$.
\end{itemize}
\end{lemma}
\begin{lemma}\label{mu2ps}
$\Lambda_{\U^{\pm}(\mathcal{W}_i)}=\left\{\begin{array}{cc}
\mu_2 & \textrm{ if } \mathcal{W}_i \textrm{ is a hyperbolic space},\\
1 & \textrm{ otherwise}.
\end{array}\right.$
\end{lemma}
\begin{proof}
By Lemma \ref{sp}, it reduces to an  anisotropic subspace $\mathcal{W}_i^0$ of $\mathcal{W}_i$.  If $\mathcal{W}_i^0 \neq 0$, by Lemmas \ref{vigneraslemma1}, \ref{vigneraslemma2}, case by case, we can obtain  that $\Lambda_{\U^{\pm}(\mathcal{W}_i^0)}=1$.
\end{proof}
Let $\mathcal{W}_i \simeq nH$ be  a hyperbolic  group. Let $\{f_1, \cdots, f_n; f_1^{\ast}, \cdots, f_n^{\ast}\}$ be a split hyperbolic basis of $\mathcal{W}_i$ so that $\langle f_k, f_l\rangle=0=\langle f_k^{\ast}, f_l^{\ast}\rangle$ and $\langle f_k, f_l^{\ast}\rangle=\delta_{kl}$.   Then we can embed $H \hookrightarrow \mathcal{W}_i$ by the subbase $\{e_n, e_n^{\ast}\}$. Let  $\mathfrak{A}=\{ \begin{pmatrix}
a&0 \\
 0& \overline{a}^{-1} \end{pmatrix}\mid  a\in D^{\times}\}$.

\begin{lemma}\label{comse}
There exists a   group homomorphism:
$$\tfrac{\U(H)}{[U(H), U(H)]} \longrightarrow \tfrac{\U(nH)}{[U(nH), U(nH)]}.$$
\end{lemma}
\begin{proof}
Let $E_{nH}(D)$ denote  the elementary linear subgroup of $\U(nH)$. By \cite[p.28, p.230 ]{HaMeTi}, we have:
\begin{itemize}
\item $E_{nH}(D)=[E_{nH}(D),E_{nH}(D)]$;
\item $\U(nH)=\U(H) E_{nH}(D)$.
\end{itemize}
Hence $[\U(nH), \U(nH)]=[\U(H),\U(H)][\U(nH), E_{nH}(D)] \supseteq [\U(H), \U(H)] E_{nH}(D)$. So the result holds.
\end{proof}
\begin{lemma}
 If   $D=E$ or $\mathbb{H}$, or $D=F$ and $\epsilon=-1$,  there exists a surjective homomorphism:
$$\tfrac{D^{\times}}{[D^{\times}, D^{\times}]} \longrightarrow \tfrac{\U(H)}{[U(H), U(H)]}.$$
\end{lemma}
\begin{proof}
Straightforward.
\end{proof}
\begin{lemma}\label{oth}
If    $D=F$ and $\epsilon=1$, $\U(nH)=\Oa(nH)$. Then we have:
\begin{itemize}
\item  There exists a   surjective group homomorphism:
$$F^{\times}\simeq \tfrac{\SO(H)}{[\SO(H), \SO(H)]} \longrightarrow \tfrac{\SO(nH)}{[\SO(nH), \SO(nH)]}.$$
\item $[\Oa(H), \Oa(H)]=\{ \begin{pmatrix} a^2 & 0 \\ 0 & a^{-2}\end{pmatrix}\mid a \in F^{\times}\}$.
\item $\tfrac{\Oa(H)}{[\Oa(H), \Oa(H)]} \simeq (F^{\times}/F^{\times 2}) \times \langle \begin{pmatrix} 0& 1\\ 1 &0 \end{pmatrix}\rangle$.
\end{itemize}
\end{lemma}
\begin{proof}
By Lemma \ref{comse}, $\SO(nH)=[\SO(nH)\cap \Oa(H)]E_{nH}(F)$. Hence there exists a group homomorphism:
$$\tfrac{\SO(H)}{[\SO(H), \SO(H)]} \longrightarrow \tfrac{\SO(nH)}{[\SO(nH), \SO(nH)]}.$$
Note:
 $$\Oa(H)=\{\begin{pmatrix} a & 0 \\ 0 & a^{-1}\end{pmatrix}, \begin{pmatrix} 0 &a \\ a^{-1} & 0\end{pmatrix}\mid a\in F^{\times}\}, \quad\quad\SO(H)=\{\begin{pmatrix} a & 0 \\ 0 & a^{-1}\end{pmatrix}\mid a\in F^{\times}\} .$$ By direct calculation, we can see that the results are all correct.
\end{proof}
\subsection{The splitting of $\widehat{\U^{\pm}}(\mathcal{W}_i)$}
\begin{proposition}\label{splitting}
Let  $(\mathcal{W}_1, \mathcal{W}_2)$ be a pair  in Theorem \ref{scindagedugroupeR0}.
\begin{itemize}
\item[(1)] If $(\mathcal{W}_1, \mathcal{W}_2)$  is not the exception case (1*), then the exact sequence $1 \longrightarrow T \longrightarrow \widehat{\U^{\pm}}(\mathcal{W}_i) \longrightarrow \U^{\pm}(\mathcal{W}_i)\longrightarrow 1$ splits, for $i=1, 2$.
\item[(2)]  If $(\mathcal{W}_1, \mathcal{W}_2)$  is  the exception case and $\epsilon_1=-1$, $\epsilon_2=1$, then
\begin{itemize}
 \item[(i)] the exact sequence $1 \longrightarrow T \longrightarrow \widehat{\Oa^{\pm}}(\mathcal{W}_2) \longrightarrow \Oa^{\pm}(\mathcal{W}_2)\longrightarrow 1$ splits, but
  \item[(ii)] the exact sequence $1 \longrightarrow T \longrightarrow \widehat{\Sp^{\pm}}(\mathcal{W}_1) \longrightarrow \Sp(\mathcal{W}_1)\longrightarrow 1$  does not split.
  \end{itemize}
    \end{itemize}
\end{proposition}
\subsubsection{The proof of Part (1)} According to Lemma \ref{vigneraslemma2},  if $\mathcal{W}_i$ is not a hyperbolic space,  $\U^{\pm}(\mathcal{W}_i)=\U(\mathcal{W}_i)$, so the result holds.   Assume now that  $\mathcal{W}_i \simeq n_iH$, for  a hyperbolic plane $H$ over $D$. Let  $\mathcal{W}_i= Y\oplus Y^{\ast}$ be a complete polarisation.

By almost symmetry, we  assume  $i=1$. Assume $\dim \mathcal{W}_2=n_2$. Then  $W=\mathcal{W}_1\otimes_DW_2\simeq (Y\otimes \mathcal{W}_2) \oplus (Y^{\ast}\otimes \mathcal{W}_2)$ is a complete polarisation. Assume that   $X=Y\otimes \mathcal{W}_2$, $X^{\ast}=Y^{\ast}\otimes \mathcal{W}_2$.  There exists an exact sequence:
$$1\to \U(\mathcal{W}_1) \to \U^{\pm}(\mathcal{W}_1) \to \mu_2\to  1.$$
Now let  $\Ha^{2}( \U^{\pm}(\mathcal{W}_1),T)_1 $  denote the kernel of the  restriction from $\Ha^{2}( \U^{\pm}(\mathcal{W}_1),T)$ to $\Ha^2(\U(\mathcal{W}_1), T)$.
By  Hochschild-Serre spectral sequence,  there exists  the following long exact sequence  of six terms:
$$0 \longrightarrow \Hom(\mu_2,  T) \to \Hom(\U^{\pm}(\mathcal{W}_1),  T)  \longrightarrow  \Hom(\U(\mathcal{W}_1), T)^{\U^{\pm}(\mathcal{W}_1)} \longrightarrow
\Ha^2( \mu_2,T) $$
$$\longrightarrow \Ha^{2}( \U^{\pm}(\mathcal{W}_1),T)_1 \stackrel{p}{\longrightarrow} \Ha^1\big(  \mu_2, \Ha^1( \U(\mathcal{W}_1), T)\big) $$
Since the restriction of $[\widetilde{C}_{X^{\ast}}]$ to $\U(\mathcal{W}_1)$ is trivial, there exists a function $f: \U(\mathcal{W}_1) \longrightarrow T$ such that
$$\widetilde{C}_{X^{\ast}}(g_1, g_2)= f(g_1g_2)f(g_1)^{-1}f(g_2)^{-1}, \qquad \quad g_i\in \U(\mathcal{W}_1).$$
Let $f$   extend to be a Borel  function of $\U^{\pm}(\mathcal{W}_1)$ by  taking the trivial value outside $\U(\mathcal{W}_1) $. We replace $\widetilde{C}_{X^{\ast}}$ with $\widetilde{C}_{X^{\ast}}'= \widetilde{C}_{X^{\ast}} \circ \delta_1 f$. Then:
$$\widetilde{C}_{X^{\ast}}'(h_1, h_2)= \widetilde{C}_{X^{\ast}}(h_1, h_2)f(h_1) f(h_2) f(h_1h_2)^{-1}, \quad\quad  h_1, h_2 \in \U^{\pm}(\mathcal{W}_1).$$
Recall the section map $s_1: \mu_2 \longrightarrow \U^{\pm}(\mathcal{W}_1); \epsilon \longmapsto \begin{pmatrix} I& 0\\ 0 &\epsilon I\end{pmatrix}$, and $s_1(\epsilon) \otimes I_{ 2}=s(\epsilon)$. By abuse of notations, we will omit $I_2$ and don't distinguish $s_1$ from $s$. Then for $g\in \U(\mathcal{W}_1)$, we have:
$$\widetilde{C}_{X^{\ast}}'(s_1(1), g)= \widetilde{C}_{X^{\ast}}(s_1(1),g)= \widetilde{C}_{X^{\ast}}([1,1],[1,g])=1;$$
$$ \widetilde{C}_{X^{\ast}}'(s_1(-1), g)= \widetilde{C}_{X^{\ast}}(s_1(-1),g)f(g)= \widetilde{C}_{X^{\ast}}([-1,1],[1,g])f(g)=\widetilde{c}_{X^{\ast}}(1, g)f(g)=f(g).$$
For any $h\in \U^{\pm}(\mathcal{W}_1)$, let us write  $h=s_1(\epsilon_h)g_h$, for $\epsilon_h \in \mu_2$, $g_h\in \U(\mathcal{W}_1)$. Define a Borel function $\varkappa$ of $\U^{\pm}(\mathcal{W}_1)$ as
$ \varkappa(h)= \widetilde{C}_{X^{\ast}}'(s_1(\epsilon_h), g_h)$. Let us define
$$\widetilde{C}_{X^{\ast}}''= \widetilde{C}_{X^{\ast}}' \circ \delta_1(\varkappa).$$ So the map $p$ is just given by
$$p([\widetilde{C}_{X^{\ast}}]): \epsilon  \longmapsto ( g  \longrightarrow \widetilde{C}_{X^{\ast}}''(g, s(\epsilon))), \quad g\in \U(W_{\nu}), \epsilon \in \mu_2.$$
Note that $p([\widetilde{C}_{X^{\ast}}])(\epsilon )$ is a character of $\U(\mathcal{W}_1)$, which factors through  $\U(\mathcal{W}_1) \longrightarrow \U(\mathcal{W}_1)/[\U(\mathcal{W}_1),\U(\mathcal{W}_1)]$. According to Lemma \ref{comse}, if $(D, \epsilon_1) \neq (F, 1)$, $p([\widetilde{C}_{X^{\ast}}])(\epsilon )$ depends on the values on $D^{\times}$ . Let $g=\diag(1, \cdots, 1, a; 1,\cdots, 1, \overline{a}^{-1})$, for $a\in D^{\times}$. Then:
\begin{align}
\widetilde{C}_{X^{\ast}}''(g, s(-1))&=\widetilde{C}_{X^{\ast}}'(g, s(-1))\varkappa(g)\varkappa(s(-1))\varkappa(gs(-1))^{-1}\\
&=\widetilde{C}_{X^{\ast}}'(g, s(-1))\widetilde{C}_{X^{\ast}}'(s(1), g)\widetilde{C}_{X^{\ast}}'(s(-1),1)\widetilde{C}_{X^{\ast}}'(s(-1), s(-1)^{-1}gs(-1))^{-1}\\
&=\widetilde{C}_{X^{\ast}}(g, s(-1)) f(g)f(g^{s(-1)})^{-1}\\
&=\widetilde{c}_{X^{\ast}}(g, s(-1))\nu(-1,g)f(g)f(g^{s(-1)})^{-1}\\
&=(-1,\det a)_F f(g)f(g^{s(-1)})^{-1}.
\end{align}
\subsection{}If $D=E$, or $D=\mathbb{H}$, we know that $\det a=\Nrd(a)\in F_{+}^{\times}$, so $(-1,\det a)_{F}=1$.    Since $\Ha^2(\mu_2,T)=0$ and $p([\widetilde{C}_{X^{\ast}}])=0$, the restriction of $[\widetilde{C}_{X^{\ast}}]$ on $\U^{\pm}(\mathcal{W}_1)$ is trivial.
\subsection{} If  $D=F$,  $\mathcal{W}_1$ is symplectic and  $\mathcal{W}_2$ is orthogonal. Since $\dim \mathcal{W}_2 $ is even,  $\det a=a^{2k}\in F_{+}^{\times}$, so $(-1,\det a)_{F}=1$.
\subsection{} If $D=F$, $\epsilon_1=1$, $\mathcal{W}_1$ is orthogonal and  $\mathcal{W}_2$ is symplectic.  Let $\mathcal{A}_i=(f_{i1}, \cdots, f_{i n_i};  f_{i1}^{\ast}, \cdots, f_{in_i}^{\ast})$ be  a split hyperbolic basis of $\mathcal{W}_i$.    Let us write $X_1= (f_{11}, \cdots, f_{1 n_1})$, $X_1^{\ast}=(f_{11}^{\ast}, \cdots, f_{1 n_1}^{\ast})$. Then
$$\{X_1\otimes f_{21}, X_1^{\ast}\otimes f_{21}, \cdots,  X_1 \otimes f_{2n_2},  X^{\ast}_1\otimes f_{2n_2}; X^{\ast}_1\otimes f^{\ast}_{21},  X_1\otimes f^{\ast}_{21}, \cdots, X^{\ast}_1 \otimes f^{\ast}_{2n_2}, X_1 \otimes f^{\ast}_{2n_2}\}$$ forms a symplectic basis of $W$.   Recall:
$$s: \Oa^{\pm}(\mathcal{W}_1) \times \Sp^{\pm}(\mathcal{W}_2) \stackrel{s_1\otimes s_2}{ \longrightarrow} \Sp^{\pm}(\mathcal{W}_1 \otimes_F \mathcal{W}_2 ) = \Sp^{\pm}(W).$$
Hence $$s_1(g)\otimes I_{\U(\mathcal{W}_2)}=\begin{pmatrix} g\otimes I_2 & 0\\ 0 & (g^{\ast})^{-1}\otimes I_2\end{pmatrix}, \quad\quad g\in \Oa(\mathcal{W}_1);$$
 $$s_1(\epsilon)\otimes I_{\U(\mathcal{W}_2)}=\diag\Bigg(\begin{pmatrix} I_1 & 0\\ 0 & \epsilon I_1\end{pmatrix}, \cdots, \begin{pmatrix} I_1 & 0\\ 0 & \epsilon I_1\end{pmatrix}; \begin{pmatrix}\epsilon  I_2 & 0\\ 0 & I_2\end{pmatrix},\cdots, \begin{pmatrix} \epsilon I_2 & 0\\ 0 & I_2\end{pmatrix}\Bigg), \quad \epsilon \in \mu_2.$$
Let us write $s_1(\epsilon)\otimes I_{\U(\mathcal{W}_2)}=s(\epsilon) g({\epsilon})$, for some  $g({\epsilon})=\begin{pmatrix} a & 0\\ 0& a^{-1}\end{pmatrix}$, and $  h_i=g(\epsilon_i)[s_1(g_i)\otimes I_{\U(\mathcal{W}_2)}]$.  Then:
\begin{align}
&\widetilde{C}_{X^{\ast}}(s_1(\epsilon_1)s_1(g_1)\otimes I_{\U(\mathcal{W}_2)}, s_1(\epsilon_2)s_1(g_2)\otimes I_{\U(\mathcal{W}_2)})\\
&=\widetilde{C}_{X^{\ast}}([\epsilon_1, h_1],[[\epsilon_2, h_2]])=\nu(\epsilon_2, h_1)\widetilde{c}_{X^{\ast}}(h_1^{s(\epsilon_1)} , h_2)\\
&=\nu(\epsilon_2, h_1).
\end{align}
\subsubsection{}If $4\mid \dim \mathcal{W}_2$, then $\nu(\epsilon_2, h_1)=(\epsilon_2, (\epsilon_1\det g )^{n_1})_{F}=1$.
\subsubsection{}If $4\nmid \dim \mathcal{W}_2$, $\nu(\epsilon_2, h_1)=(\epsilon_2, (\epsilon_1\det g )^{n_1})_{F}=(\epsilon_2, \epsilon_1\det g )_{F}$. Then the restriction of $\widetilde{C}_{X^{\ast}}$ on $\Oa^{\pm}(\mathcal{W}_1)$ factors through $\Oa^{\pm}(\mathcal{W}_1)/\SO(\mathcal{W}_1)\simeq \mu_2\times \Oa(\mathcal{W}_1)/\SO(\mathcal{W}_1)\simeq \mu_2\times \mu_2$. Moreover,
$$\widetilde{C}_{X^{\ast}}((\epsilon_1,[\det g_1]), (\epsilon_2,[\det g_2]))=(\epsilon_2,  \epsilon_1 [\det g_1])_{F}=(\epsilon_2,  \epsilon_1)_F (\epsilon_2, [\det g_1])_F.$$
By Remark \ref{splittingc},$(\epsilon_2,  \epsilon_1)_F $ defines a splitting cocycle on $\mu_2$. As $$\Ha^2(\mu_2\times \mu_2, T) \simeq \Ha^2(\mu_2, T) \oplus \Ha^2(\mu_2, T)\oplus \Ha^1(\mu_2, H^1(\mu_2, T)),$$  $\widetilde{C}_{X^{\ast}}$ is not a splitting cocycle on $\mu_2\times \mu_2$.  The subtle part is to see  whether this $2$-cocycle can be splitting on $\Oa^{\pm}(\mathcal{W}_1)$.

By  Hochschild-Serre spectral sequence,  there exists  the following long exact sequence:
$$0 \longrightarrow \Hom(\mu_2\times \mu_2,  T) \to \Hom(\Oa^{\pm}(\mathcal{W}_1),  T)  \longrightarrow  \Hom(\SO(\mathcal{W}_1), T)^{\Oa^{\pm}(\mathcal{W}_1)} \stackrel{tr}{\longrightarrow}
\Ha^2( \mu_2\times \mu_2,T) $$
$$\stackrel{inf_2}{\longrightarrow}\Ha^{2}(\Oa^{\pm}(\mathcal{W}_1),T)_1 \stackrel{p_1}{\longrightarrow} \Ha^1\big(  \mu_2\times \mu_2, \Ha^1( \SO(\mathcal{W}_1), T)) $$
\paragraph{} Let us first show that $p_1[\widetilde{C}_{X^{\ast}}]$ is trivial. Note that
$$\widetilde{C}_{X^{\ast}}(s_1(g_1)\otimes I_{\U(\mathcal{W}_2)}, s_1(g_2)\otimes I_{\U(\mathcal{W}_2)})=1, \quad\quad g_i\in \Oa(\mathcal{W}_1).$$
For any $h\in \Oa^{\pm}(\mathcal{W}_1)$, let us write $h=[\epsilon_h, g_h]$, for some $\epsilon_h\in \mu_2$,  $g_h\in \Oa(\mathcal{W}_1)$. As
$$\widetilde{C}_{X^{\ast}}(s_1(\epsilon_h)\otimes I_{\U(\mathcal{W}_2)}, s_1(g_h)\otimes I_{\U(\mathcal{W}_2)})=1,$$
 the map $p_1$ is just given by
$$p_1([\widetilde{C}_{X^{\ast}}]): \epsilon  \longmapsto ( g  \longrightarrow \widetilde{C}_{X^{\ast}}(g\otimes I_{\U(\mathcal{W}_2)}, s_1(\epsilon)\otimes I_{\U(\mathcal{W}_2)})=1), \quad g\in \U(\mathcal{W}_1), \epsilon \in \mu_2.$$

\paragraph{} Let us show that the $2$-cocycle $\widetilde{C}_{X^{\ast}}$ does  come from $ \Hom(\SO(\mathcal{W}_1), T)^{\Oa^{\pm}(\mathcal{W}_1)}$ by the transgression.
 Under the basis $\mathcal{A}_1$ and its  subbase $\{ f_{1 n_1}, f_{1n_1}^{\ast}\}$, let us write $$\mathcal{W}_1=n_1H, \quad \Oa^{\pm}(\mathcal{W}_1)=\Oa^{\pm}(n_1H),\quad \SO(\mathcal{W}_1)=\SO(n_1H).$$ By Lemma \ref{oth}, $$\SO(n_1H)=\SO(H) E_{n_1H}(F), \quad \Oa^{\pm}(n_1H)=\Oa^{\pm}(H)E_{n_1H}(F),\quad E_{n_1H}(F)=[E_{n_1H}(F), E_{n_1H}(F)].$$ Any character of $\SO(n_1H)$ is trivial on $E_{n_1H}(F)$. By the explicit construction of  the transgression in \cite[p.65]{NeScWi}, we can assume $n_1=1$.  In this case, let $\chi\in \Hom(\SO(H), T)^{\Oa^{\pm}(H)}$, then $\chi$ is trivial or $\chi(\begin{pmatrix} a & 0\\ 0 & a^{-1}\end{pmatrix})=\sgn(a)$. It is clear that if $\chi$ is the trivial character, $[tr(\chi)]$ can not be $[\widetilde{C}_{X^{\ast}}]$. So let us consider the  other character. Note:
 $$\Oa^{\pm}(H)=\{\begin{pmatrix} \pm a & 0\\ 0 & a^{-1}\end{pmatrix}, \begin{pmatrix} 0 & \pm a\\a^{-1} & 0\end{pmatrix} \mid a\in F^{\times}\},$$
 and
 $$\Oa^{\pm}(H)/\{\begin{pmatrix} a^2 & 0\\ 0 & a^{-2}\end{pmatrix}\}\simeq \Bigg\{ \pm \begin{pmatrix} 1 & 0\\ 0 &1\end{pmatrix}, \pm \begin{pmatrix}0 & 1\\ 1 & 0\end{pmatrix}, \pm \begin{pmatrix} 0 & 1\\ -1&0\end{pmatrix},\pm \begin{pmatrix} 1 & 0\\0&-1\end{pmatrix}\Bigg\}.$$
Following \cite[p.65]{NeScWi}, let us extend $\chi$ to $\Oa^{\pm}(H)$ by defining
$$\chi(g)=1, \quad   \chi(-g)=-1,$$
for $g\in \mathcal{B}=\Bigg\{ \begin{pmatrix} 1 & 0\\ 0 &1\end{pmatrix}, \begin{pmatrix}0 & 1\\ 1 & 0\end{pmatrix}, \begin{pmatrix} 0 & 1\\ -1&0\end{pmatrix},\begin{pmatrix} 1 & 0\\0&-1\end{pmatrix}\Bigg\}$. Then:
$$tr(\chi)(g_1, g_2)=[\partial\chi](g_1,g_2)=\chi(g_1)\chi(g_2)\chi(g_1g_2)^{-1}.$$
Note that $\partial\chi(-g_1,g_2)=\partial\chi(g_1,g_2)=\partial\chi(g_1,-g_2)$. The correspondence  between $\mu_2\times \mu_2$ and $\mathcal{B}$ is given by
\begin{align}
\mu_2\times \mu_2\longrightarrow \mathcal{B};\\
[1,1] \to  \begin{pmatrix} 1 & 0\\ 0 &1\end{pmatrix}\\
[-1,1] \to \begin{pmatrix} 1 & 0\\ 0 &-1\end{pmatrix}\\
[1,-1] \to  \begin{pmatrix}0 & 1\\ 1 & 0\end{pmatrix}\\
[-1,-1] \to \begin{pmatrix} 0 & 1\\ -1&0\end{pmatrix}.
\end{align}
Hence
$$\partial\chi([\epsilon_1, \epsilon_1'], [\epsilon_2, \epsilon_2'])=(\epsilon_2, \epsilon_1')_F,$$
which means that $[\partial\chi]=[\widetilde{C}_{X^{\ast}}]$. Finally we can see that $[\widetilde{C}_{X^{\ast}}]$ is also splitting on $\Oa^{\pm}(\mathcal{W}_1)$.

\subsubsection{The proof of Part (2)}
In this case, $O^{\pm}(\mathcal{W}_2)=O(\mathcal{W}_2)$, and  $\widehat{\Sp}(\mathcal{W}_1) $ is not splitting. By the  long exact sequence  of six terms, we know that the two covering over $\Sp(W)$ only  extends to $\Sp^{\pm}(W)$ in  a unique way. Hence  $\widehat{\Sp^{\pm}}(\mathcal{W}_1)\simeq \widehat{\Sp}^{\pm}(\mathcal{W}_1) $, which  is also not splitting.
\subsection{The bicommutant}
We now let  $\widehat{\U^{\pm}}(\mathcal{W}_1)$, $\widehat{\U^{\pm}}(\mathcal{W}_2)$ be the  preimages of $\U^{\pm}(\mathcal{W}_1)$,  $\U^{\pm}(\mathcal{W}_2)$ in $\widehat{\Sp}^{\pm}(W)$ respectively.

\begin{proposition}\label{comm}
Let  $(\mathcal{W}_1, \mathcal{W}_2)$ be a pair  in Theorem \ref{scindagedugroupeR0}.
\begin{itemize}
\item[(1)] If $D=E$ or $\mathbb{H}$,  $\widehat{\U^{\pm}}(\mathcal{W}_1) $ commutes with  $\widehat{\U^{\pm}}(\mathcal{W}_2)$ in $\widehat{\Sp}^{\pm}(W)$.
\item[(2)] If $D=F$,  assume $\epsilon_1=-1$, $\epsilon_2=1$.
\begin{itemize}
 \item[(a)] If $2\mid \dim \mathcal{W}_2$, and $\mathcal{W}_2$ is not  a hyperbolic space, then:
 \begin{itemize}
\item[(i)] $\Oa^{\pm}(\mathcal{W}_2)=\Oa(\mathcal{W}_2)$.
 \item[(ii)] if $4 \mid \dim \mathcal{W}_1$, then $\widehat{\Sp^{\pm}}(\mathcal{W}_1) $ commutes with  $\Oa(\mathcal{W}_2)$ in $\widehat{\Sp}^{\pm}(W)$.
  \item[(iii)] if $4\nmid \dim \mathcal{W}_1$, then $\widehat{\Sp^{\pm}}(\mathcal{W}_1) $ only commutes with  $\SO(\mathcal{W}_2)$ in $\widehat{\Sp}^{\pm}(W)$, and $\Oa(\mathcal{W}_2)$ only commutes with  $\widehat{\Sp}(\mathcal{W}_1)$ in $\widehat{\Sp}^{\pm}(W)$.
  \end{itemize}
   \item[(b)] If $2\mid \dim \mathcal{W}_2$, and $\mathcal{W}_2$ is  a hyperbolic space, then:
 \begin{itemize}
\item[(i)] $\widehat{\Oa^{\pm}}(\mathcal{W}_2)\simeq \Oa^{\pm}(\mathcal{W}_2) \times T$.
 \item[(ii)] if $4 \mid \dim \mathcal{W}_1$,  then $\widehat{\Sp^{\pm}}(\mathcal{W}_1) $ commutes with  $\widehat{\Oa^{\pm}}(\mathcal{W}_2)$ in $\widehat{\Sp}^{\pm}(W)$.
  \item[(iii)] if $4 \nmid \dim \mathcal{W}_1$ and  $4\mid \dim \mathcal{W}_2$, then $\widehat{\Sp^{\pm}}(\mathcal{W}_1) $ only commutes with  $\widehat{\SO^{\pm}}(\mathcal{W}_2)$ in $\widehat{\Sp}^{\pm}(W)$, and $\widehat{\Oa^{\pm}}(\mathcal{W}_2)$ only commutes with  $\widehat{\Sp}(\mathcal{W}_1)$ in $\widehat{\Sp}^{\pm}(W)$.
 \item[(iv)] if  $4\nmid \dim \mathcal{W}_1$ and   $4\nmid \dim \mathcal{W}_2$,  then $\widehat{\Sp^{\pm}}(\mathcal{W}_1) $ only commutes with  $\widehat{\SO}(\mathcal{W}_2)$ in $\widehat{\Sp}^{\pm}(W)$, and $\widehat{\Oa^{\pm}}(\mathcal{W}_2)$ only commutes with  $\widehat{\Sp}(\mathcal{W}_1)$ in $\widehat{\Sp}^{\pm}(W)$.
  \end{itemize}
  \item[(c)]  If $2\nmid \dim \mathcal{W}_2$, then:
\begin{itemize}
\item[(i)] $\Oa^{\pm}(\mathcal{W}_2)=\Oa(\mathcal{W}_2)$.
 \item[(ii)] if $4 \mid \dim \mathcal{W}_1$, then $\widehat{\Sp^{\pm}}(\mathcal{W}_1) $ commutes with  $\Oa(\mathcal{W}_2)$ in $\widehat{\Sp}^{\pm}(W)$.
  \item[(iii)] if $4\nmid \dim \mathcal{W}_1$, then $\widehat{\Sp^{\pm}}(\mathcal{W}_1) $ only commutes with  $\SO(\mathcal{W}_2)$, and  $\Oa(\mathcal{W}_2)$ only commutes with  $\widehat{\Sp}(\mathcal{W}_1)$ in $\widehat{\Sp}^{\pm}(W)$.
  \end{itemize}
    \end{itemize}
    \end{itemize}
\end{proposition}

\subsection{The proof of Proposition \ref{comm}}
\subsubsection{}
If neither $\mathcal{W}_1$ nor $\mathcal{W}_2$ is a hyperbolic space, then $\U^{\pm}(\mathcal{W}_i)=\U(\mathcal{W}_i)$. So the result follows from Theorem \ref{scindagedugroupeR4}.
\subsubsection{} If  $\mathcal{W}_1\simeq n_1H$, $\mathcal{W}_2$ is not a hyperbolic space, and  $\dim_F \mathcal{W}_2$ is even,  then $\U^{\pm}(\mathcal{W}_2)=\U(\mathcal{W}_2)$, $\U^{\pm}(\mathcal{W}_1)=\U^{\pm}(n_1H)$.   By Proposition \ref{splitting}, $$\widehat{\U^{\pm}}(\mathcal{W}_1)\simeq \U^{\pm}(\mathcal{W}_1) \times T, \quad \widehat{\U}(\mathcal{W}_i) \simeq \U(\mathcal{W}_i)\times T.$$
By Theorem \ref{scindagedugroupeR4}, $\widehat{\U}(\mathcal{W}_1)$ is commuted with $\widehat{\U}(\mathcal{W}_2)$.  Let  $\mathcal{W}_1= Y\oplus Y^{\ast}$ be a complete polarisation. Assume that
$X=Y\otimes_D \mathcal{W}_2$, $X^{\ast}=Y^{\ast}\otimes_D \mathcal{W}_2$. There exists a section map:
$$s: \mu_2\longrightarrow \widehat{\U}(\mathcal{W}_1); \epsilon \longmapsto \begin{pmatrix} I & 0\\ 0& \epsilon I\end{pmatrix}.$$
Let $S_1=s(\mu_2)$, and $\mathcal {S}_1=s(\mu_2)\otimes I_{\U(\mathcal{W}_2)}\subseteq \Sp^{\pm}(W)$. Then $  \widehat{\U^{\pm}}(\mathcal{W}_1)=\widehat{S}_2\widehat{\U}(\mathcal{W}_1)$. So it suffices to show that $\widehat{S}_1$ is commuted with $\widehat{\U}(\mathcal{W}_2)$. For $g\in \U(\mathcal{W}_2)$, $s(-1)\otimes I_{\U(\mathcal{W}_2)} \in \mathcal {S}_2$,
$$\widetilde{C}_{X^{\ast}}(I_{\U(\mathcal{W}_1)}\otimes g, s(-1)\otimes I_{\U(\mathcal{W}_2)})=\nu((\det g)^{n_1}, -1)\widetilde{c}_{X^{\ast}}(I_{\U(\mathcal{W}_1)}\otimes g, I_{\Sp(W)})=((\det g)^{n_1}, -1)_F;$$
$$\widetilde{C}_{X^{\ast}}(s(-1)\otimes I_{\U(\mathcal{W}_2)}, I_{\U(\mathcal{W}_1)}\otimes g)=\widetilde{c}_{X^{\ast}}( I_{\Sp(W)},I_{\U(\mathcal{W}_1)}\otimes g)=1.$$
a) If $D=E$ or $\mathbb{H}$, $\det g$ is a positive number, so  $((\det g)^{n_1}, -1)_F=1$ and $\widehat{S}_1$ commutes with $\widehat{\U}(\mathcal{W}_2)$.\\
b) If $D=F$, and $\mathcal{W}_1$ is symplectic with $2\mid {n_1}$, then   $((\det g)^{n_1}, -1)_F=1$ and $\widehat{S}_1$ commutes with $\widehat{\U}(\mathcal{W}_2)$.\\
c) If $D=F$, and $\mathcal{W}_1$ is symplectic with $2\nmid {n_1}$, then   $((\det g)^{n_1}, -1)_F=1$ implies that $\det g>0$. So $\widehat{\Sp^{\pm}}(\mathcal{W}_1)$ only  commutes with $\widehat{\SO}(\mathcal{W}_2)$. Consequently, $\widehat{\Oa}(\mathcal{W}_2)$ only  commutes with $\widehat{\Sp}(\mathcal{W}_1)$.
 \subsubsection{} If  $\mathcal{W}_1\simeq nH$, $\mathcal{W}_2$ is not a hyperbolic space, and  $\dim_F \mathcal{W}_2$ is odd. In this case, $\mathcal{W}_1$ is symplectic and $\mathcal{W}_2$ is orthogonal.  Then $\U^{\pm}(\mathcal{W}_1) \simeq \Sp^{\pm}(\mathcal{W}_1)$, $\U^{\pm}(\mathcal{W}_2) \simeq \Oa(\mathcal{W}_2)$, and $\widehat{\Sp^{\pm}}(\mathcal{W}_1) \simeq \widehat{\Sp}^{\pm}(\mathcal{W}_1)$, $\widehat{\Oa}(\mathcal{W}_2) \simeq \Oa(\mathcal{W}_2) \times T$. Similarly as above, we have:
 \begin{itemize}
 \item  If $2\mid {n_1}$, $\widehat{\Sp^{\pm}}(\mathcal{W}_1) $  commutes with $\widehat{\Oa}(\mathcal{W}_2)$ or $\Oa(\mathcal{W}_2)$.
 \item  If $2\nmid {n_1}$, $\widehat{\Sp^{\pm}}(\mathcal{W}_1) $ only commutes with $\widehat{\SO}(\mathcal{W}_2)$ or $\SO(\mathcal{W}_2)$, and $\widehat{\Oa}(\mathcal{W}_2)$ only commutes with $ \widehat{\Sp}(\mathcal{W}_1)$.
 \end{itemize}
\subsubsection{} If  $\mathcal{W}_2\simeq n_2H$,  $\mathcal{W}_1$ is not a hyperbolic space, and    $\dim_F \mathcal{W}_1$ is even, the proof is similar as above.
\subsubsection{} If  $\mathcal{W}_1$ and $\mathcal{W}_2$ both are hyperbolic spaces, then $\U^{\pm}(\mathcal{W}_i)\neq \U(\mathcal{W}_i)$ and  $
\widehat{ \U^{\pm}}(\mathcal{W}_i) \simeq \U^{\pm}(\mathcal{W}_i) \times T$. Let $\mathcal{A}_i=(f_{i1}, \cdots, f_{i n_i};  f_{i1}^{\ast}, \cdots, f_{in_i}^{\ast})$ be  a split hyperbolic basis of $\mathcal{W}_i$. Let us write:
$$s_i: \mu_2 \longrightarrow \U^{\pm}(\mathcal{W}_i); \epsilon \longmapsto  \begin{pmatrix} I& 0\\ 0 & \epsilon I\end{pmatrix},$$
where $$s_1(-1)(f_{11}, \cdots, f_{1 n_1};  f_{11}^{\ast}, \cdots, f_{1n_1}^{\ast})=(f_{11}, \cdots, f_{1 n_1};  -f_{11}^{\ast}, \cdots, -f_{1n_1}^{\ast}),$$
$$(f_{21}, \cdots, f_{2 n_2};  f_{21}^{\ast}, \cdots, f_{2n_2}^{\ast}) s_2(-1)=(f_{21}, \cdots, f_{2 n_2};  -f_{21}^{\ast}, \cdots, -f_{2n_2}^{\ast}).$$
Let $S_i=s_i(\mu_2)$, and $\widehat{S_i} \subseteq \widehat{\U^{\pm}}(\mathcal{W}_i)$.  Let us write $X_1= (f_{11}, \cdots, f_{1 n_1})$, $X_1^{\ast}=(f_{11}^{\ast}, \cdots, f_{1 n_1}^{\ast})$, $Y_1= (f_{21}, \cdots, f_{2 n_2})$, $Y_1^{\ast}=(f_{21}^{\ast}, \cdots, f_{2 n_2}^{\ast})$. Then:
$$(f_{11} \otimes Y_1, f_{11} \otimes Y_1^{\ast},  \cdots, f_{1 n_1}\otimes Y_1, f_{1 n_1}\otimes Y_1^{\ast};  f^{\ast}_{11} \otimes Y_1^{\ast},  \epsilon_2 f^{\ast}_{11} \otimes Y_1,  \cdots, f_{1 n_1}^{\ast}\otimes Y_1^{\ast}, \epsilon_2  f_{1 n_1}^{\ast}\otimes Y_1)$$
forms a split hyperbolic basis of $\mathcal{W}_1 \otimes_D \mathcal{W}_2 $. Recall:
$$s: \U^{\pm}(\mathcal{W}_1) \times \U^{\pm}(\mathcal{W}_2) \stackrel{s_1\otimes s_2}{ \longrightarrow} \U^{\pm}(\mathcal{W}_1 \otimes_D \mathcal{W}_2 ) \longrightarrow \Sp(W).$$
Hence $$s_1(\epsilon)\otimes I_{\U(\mathcal{W}_2)}=\begin{pmatrix} I_{1}\otimes I_2 & 0\\ 0 & \epsilon I_{1}\otimes I_2\end{pmatrix}=s(\epsilon),$$
$$I_{\U(\mathcal{W}_1)} \otimes s_2(\epsilon)=\diag\Bigg(\begin{pmatrix} I_2 & 0\\ 0 & \epsilon I_2\end{pmatrix}, \cdots, \begin{pmatrix} I_2 & 0\\ 0 & \epsilon I_2\end{pmatrix}; \begin{pmatrix}\epsilon  I_2 & 0\\ 0 & I_2\end{pmatrix},\cdots, \begin{pmatrix} \epsilon I_2 & 0\\ 0 & I_2\end{pmatrix}\Bigg).$$
Let us write $I_{\U(\mathcal{W}_2)} \otimes s_2(\epsilon)=s(\epsilon) g_0$, for some  $g_0=\begin{pmatrix} a & 0\\ 0& a^{-1}\end{pmatrix}$. Then:
$$\widetilde{C}_{X^{\ast}}(s_1(-1)\otimes I_{\U(\mathcal{W}_2)}, I_{\U(\mathcal{W}_1)} \otimes s_2(-1))=\widetilde{C}_{X^{\ast}}([-1, 1],[-1, g_0])=\nu(-1, 1)\widetilde{c}_{X^{\ast}}(1 , g_0)=1 ,$$
\begin{align*}
\widetilde{C}_{X^{\ast}}( I_{\U(\mathcal{W}_1)} \otimes s_2(-1), s_1(-1)\otimes I_{\U(\mathcal{W}_2)})&=\widetilde{C}_{X^{\ast}}([-1, g_0], [-1, 1])=\nu(-1, g_0)\widetilde{c}_{X^{\ast}}(g_0^{s(-1)} , 1)\\
&=\nu(-1, g_0)= \left\{ \begin{array}{cl} 1 & \textrm{ if } D=E, \mathbb{H},\\ (-1 ,(-1)^{n_1n_2})_{F} &  \textrm{ if } D=F.\end{array}\right.
\end{align*}
Hence:
\begin{itemize}
\item[(1)] If $D=E$ or $\mathbb{H}$, $\widehat{S_1}$ commutes with $\widehat{S_2}$. Consequently, $\widehat{\U^{\pm}}(\mathcal{W}_1)$  is commuted  with $\widehat{\U^{\pm}}(\mathcal{W}_2)$ in $\widetilde{\Sp}^{\pm}(W)$.
\item[(2)] If $D=F$,  assume that $\mathcal{W}_1$ is  symplectic and $\mathcal{W}_2$ is orthogonal. Let $g_1\in \Sp(\mathcal{W}_1)$, $g_2\in \Oa(\mathcal{W}_2)$. Then:
\begin{align*}
\widetilde{C}_{X^{\ast}}( s_1(-1)\otimes I_2, I_{1} \otimes g_2)&=\widetilde{C}_{X^{\ast}}([-1,1], [1, I_{1} \otimes g_2])\\
&=\nu(1, 1)\widetilde{c}_{X^{\ast}}(1 , I_{1} \otimes g_2)=1;
\end{align*}
\begin{align*}
\widetilde{C}_{X^{\ast}}(I_{1} \otimes g_2,  s_1(-1)\otimes I_2)&=\widetilde{C}_{X^{\ast}}([1, I_{1} \otimes g_2], [-1,1])\\
&=\nu(-1,  I_{1} \otimes g_2)\widetilde{c}_{X^{\ast}}( [I_{1} \otimes g_2]^{s(-1)} , 1)\\
&=\nu(-1,  I_{1} \otimes g_2)=((\det g_2)^{n_1}, -1)_{\R}.
\end{align*}
Hence:
\begin{itemize}
\item[i)] If $ 2 \mid n_1$, $\widehat{S_1}$ commutes with $\widehat{\Oa}(\mathcal{W}_2)$.
\item[ii)] If $ 2 \nmid n_1$, $\widehat{S_1}$  only commutes with $\widehat{\SO}(\mathcal{W}_2)$.
\end{itemize}
Similarly, by choosing the similar split hyperbolic basis for $\mathcal{W}_1\otimes \mathcal{W}_2$, we can obtain that $\widehat{S_2}$ commutes with $\widehat{\Sp}(\mathcal{W}_1)$. We conclude:
\begin{itemize}
\item[a)] If $ 2 \mid n_1$ and $2\mid n_2$,  by the above proof of Part (1), we know that  $\widehat{\Sp^{\pm}}(\mathcal{W}_1)$  commutes with $\widehat{\Oa}(\mathcal{W}_2)$  and $\widehat{\Oa^{\pm}}(\mathcal{W}_2)$ commutes   with $\widehat{\Sp}(\mathcal{W}_1)$.  By the above discussion,  $\widehat{S_1}$ commutes with $\widehat{S_2}$ and $\widehat{\Oa}(\mathcal{W}_2)$. Hence $\widehat{\Sp^{\pm}}(\mathcal{W}_1)$  commutes with $\widehat{\Oa^{\pm}}(\mathcal{W}_2)$.
\item[b)] If $ 2 \mid n_1$ and $2\nmid n_2$,   $\widehat{\Sp^{\pm}}(\mathcal{W}_1)$  commutes with $\widehat{\Oa}(\mathcal{W}_2)$, $\widehat{S_1}$ commutes with $\widehat{S_2}$,
 $\widehat{\Oa^{\pm}}(\mathcal{W}_2)$ commutes   with $\widehat{\Sp}(\mathcal{W}_1)$.   Hence $\widehat{\Sp^{\pm}}(\mathcal{W}_1)$  commutes with $\widehat{\Oa^{\pm}}(\mathcal{W}_2)$.
\item[c)] If  $2\nmid n_1$ and  $2\mid n_2$, $\widehat{\Sp^{\pm}}(\mathcal{W}_1)$ only commutes with $\widehat{\SO}(\mathcal{W}_2)$, $\widehat{S_1}$ commutes with $\widehat{S_2}$,
 $\widehat{\Oa^{\pm}}(\mathcal{W}_2)$ commutes   with $\widehat{\Sp}(\mathcal{W}_1)$.   Hence $\widehat{\Sp^{\pm}}(\mathcal{W}_1)$ only commutes with $\widehat{\SO^{\pm}}(\mathcal{W}_2)$ and  $\widehat{\Oa^{\pm}}(\mathcal{W}_2)$ only commutes with $\widehat{\Sp}(\mathcal{W}_2)$.
\item[d)] If $ 2 \nmid n_1$, $2\nmid n_2$, $\widehat{\Sp^{\pm}}(\mathcal{W}_1)$ only commutes with $\widehat{\SO}(\mathcal{W}_2)$, $\widehat{S_1}$ is not commuted with $\widehat{S_2}$,
 $\widehat{\Oa^{\pm}}(\mathcal{W}_2)$ commutes   with $\widehat{\Sp}(\mathcal{W}_1)$.   Hence $\widehat{\Sp^{\pm}}(\mathcal{W}_1)$ only commutes with $\widehat{\SO}(\mathcal{W}_2)$ and  $\widehat{\Oa^{\pm}}(\mathcal{W}_2)$ only commutes with $\widehat{\Sp}(\mathcal{W}_1)$.
\end{itemize}
\end{itemize}

\subsection{Unitary Weil representation}\label{un}
In this last subsection, we will let $V $ be a complex vector space of dimension $m$ with a basis $\{f_1, \cdots, f_m\}$. Let $(,)_{V}$ be a left hermitian form on $V$ given as follows:
$$(v, v')_V=\sum_{i=1}^m z_i\overline{z'}_i, \quad\quad z=\sum_{i=1}^m z_i f_i, z'=\sum_{i=1}^m z_i' f_i.$$
Let $\Ha(V)=V\oplus F$ be the group of elements $(v, t)$,  with the
multiplication law given by
$$(v, t)(v',t')=(v+v', \tfrac{i\Im(v,v')_V}{2}),   \quad\quad v, v'\in V, t,t'\in F.$$
Then there exists a canonical $\Gal(\C/\R)$-action on $V$.  We can extend this action  to $\Ha(V)$ as follows:
$$  \Ha(V)\times \Gal(\C/\R)\longrightarrow V; (\sigma, [v,t]) \longrightarrow  [v^{\sigma}, \chi(\sigma)t],$$
for the non-trivial quadratic character $\chi$ of $\Gal(\C/\R)$. Let $\Gal(\C/\R)=\langle \sigma\rangle $, $F_2=\langle s(-1)\rangle$.
\begin{lemma}\label{iso}
There exists a group isomorphism:
\begin{align}
\phi : &     \Gal(\C/\R)\ltimes \Ha(V) \longrightarrow \Ha^{\pm}(W)\simeq F_2\ltimes \Ha(W)  ;\\
         &\quad \sigma \longmapsto s(-1)\\
          &\quad  (\sum^m_{k=1} (a_k +ib_k ) f_k, t) \longmapsto  (\sum^m_{k=1} a_k e_k +b_k e_k^{\ast},t).
 \end{align}
\end{lemma}
\begin{proof}
It is clear that $\phi$ is a bijective map. \\
1) \begin{align*}
\phi([\sum^m_{k=1} (a_k +ib_k ) f_k, t]^{\sigma})&=\phi([\sum^m_{k=1} (a_k -ib_k ) f_k, -t])=[\sum^m_{k=1} a_k e_k -b_k e_k^{\ast},-t]\\
&=[\sum^m_{k=1} a_k e_k +b_k e_k^{\ast},t]^{s(-1)}=\phi([\sum^m_{k=1} (a_k +ib_k ) f_k, t])^{\phi(\sigma)}.
\end{align*}
2) For two element $h_l=(w_l, t_l)\in  \Ha(W)$, with $w_l=\sum^m_{k=1} a^{l}_k e_k +b^{l}_k e_k^{\ast}$, for $l=1,2$, we have:
\begin{align*}
h_1 h_2=(w_1+w_2, t_1+t_2+\tfrac{\langle w_1,w_2\rangle}{2})=(w_1+w_2, t_1+t_2+\tfrac{1}{2}\sum_{k=1}^m  (a_k^1 b_k^2-a_k^2b_k^1));
\end{align*}
\begin{align*}
\phi^{-1}(h_l)=(\sum^m_{k=1} (a^l_k +ib^l_k ) f_k, t_l);
\end{align*}
\begin{align*}
\phi^{-1}(h_1h_2)=(\sum^m_{k=1} (a^1_k +ib^1_k ) f_k+\sum^m_{k=1} (a^2_k +ib^2_k ) f_k, t_1+t_2+\tfrac{1}{2}\sum_{k=1}^m  (a_k^1 b_k^2-a_k^2b_k^1));
\end{align*}
\begin{align*}
\phi^{-1}(h_1)\phi^{-1}(h_2)&=(\sum^m_{k=1} (a^1_k +ib^1_k ) f_k, t_1)(\sum^m_{k=1} (a^2_k +ib^2_k ) f_k, t_2)\\
&=(\sum^m_{k=1} (a^1_k +ib^1_k ) f_k+\sum^m_{k=1} (a^2_k +ib^2_k ) f_k, t_1+t_2+\tfrac{1}{2}\sum^m_{k=1}(a^1_kb^2_k-b^1_ka^2_k)).
\end{align*}
Hence $\phi^{-1}(h_1h_2)=\phi^{-1}(h_1)\phi^{-1}(h_2)$.
\end{proof}
\begin{lemma}
The above $\phi$ can extend to be a group homomorphism: $$\phi :   \Gal(\C/\R)\U(V)\ltimes \Ha(V)   \longrightarrow \Sp^{\pm}(W)\ltimes \Ha(W).$$
\end{lemma}
\begin{proof}
Let $h=(v, t)\in  \Ha(V)$, with $v=\sum^m_{k=1} (a_k  +ib_k)f_k$. Let $g\in \U(V)$. Then:
 \begin{align*}
h^{(\sigma, g)}=h^{(\sigma, 1)(1, g)}=(\sum^m_{k=1} (a_k  -ib_k)f_k,-t)^g=([\sum^m_{k=1} (a_k  -ib_k)f_k]g,-t);
\end{align*}
 \begin{align*}
\phi(h^{(\sigma, g)})=\phi([\sum^m_{k=1} (a_k  -ib_k)f_k]g,-t)=([\sum^m_{k=1} (a_k e_k -ib_ke^{\ast}_k)]g,-t);
\end{align*}
 \begin{align*}
\phi(h)^{\phi(\sigma, g)}&=(\sum^m_{k=1} (a_k e_k +ib_ke^{\ast}_k),t)^{(s(-1),g)}=(\sum^m_{k=1} (a_k e_k -ib_ke^{\ast}_k),-t)^{g}\\
&=([\sum^m_{k=1} (a_k e_k -ib_ke^{\ast}_k)]g,-t)=\phi(h^{(\sigma, g)}).
\end{align*}
\end{proof}
Let $\widehat{\U(V)}$ be the preimage of $\U(V)$ in $\widehat{\Sp}(W)$.  Then this $\phi$ can extend to  a group homomorphism:
$$\phi:  \Gal(\C/\R)\widehat{\U(V)}\ltimes \Ha(V)   \longrightarrow \widehat{\Sp}^{\pm}(W)\ltimes \Ha(W).$$
Let us next show that $\widehat{\U(V)}$ is splitting from Theorem \ref{scindagedugroupeR0} in details.
\subsubsection{} Let $\mathcal{W}=E$, be a right vector space over $E$,  endowed with the skew  hermitian form: $$(z,z')_{\mathcal{W}}=-i\overline{z}z'.$$
 Let $\mathcal{V}=V$, endowed with the hermitian form: $(-, -)_{\mathcal{V}}=(-, -)_{V}$.
 Then
 $$( \mathcal{W}\otimes \mathcal{V} , \tfrac{1}{2} \tr_{E/F}[(,)_{\mathcal{W}}\otimes \overline{(, )}_{\mathcal{V}}])$$
 is a symplectic space over $F$ of dimension $2m$. Let $\Ha(\mathcal{W} \otimes \mathcal{V})$ be the corresponding Heisenberg group. Then there exists the following isomorphism:
 $$\ell: \Ha(\mathcal{W} \otimes \mathcal{V}) \longrightarrow \Ha(V); (\sum w_l \otimes v_l, t) \longmapsto (\sum w_lv_l, t).$$
 It is clear that $\ell$ is a bijective map. Let us show that it is also a group homomorphism:
 $$ i\Im z=\tfrac{1}{2} \Tr_{E/F}(iz),  \quad \quad z\in E;$$
 \begin{align*}
 (\sum w_l \otimes v_l, t_1) (\sum w_k \otimes v_k, t_2)&=(\sum_l w_l \otimes v_l +\sum_{k}w_k \otimes v_k, t_1+t_2 + \tfrac{1}{4} \sum_l\sum_{k}\tr_{E/F}[-i \overline{w_l} w_k\overline{(v_l ,v_k )}_{V} ]);
\end{align*}
 \begin{align*}
 \ell((\sum w_l \otimes v_l, t_1) (\sum w_k \otimes v_k, t_2))=(\sum_l w_l  v_l +\sum_{k}w_k  v_k, t_1+t_2 + \tfrac{1}{4} \sum_l\sum_{k}\tr_{E/F}[-i \overline{w_l} w_k\overline{(v_l ,v_k )}_{V} ]);
\end{align*}
 \begin{align*}
 \ell([\sum_l w_l \otimes v_l, t_1]) \ell([\sum_k w_k \otimes v_k, t_2])&=(\sum_l w_l  v_l, t_1)(\sum_k w_k v_k, t_2)\\
 &=(\sum_lw_l  v_l+\sum_k w_k v_k, t_1+t_2+ \tfrac{1}{2}i\Im(\sum_l w_l  v_l,\sum_k w_k v_k)_V)\\
 &=(\sum_l w_l  v_l+\sum_k w_k v_k, t_1+t_2+ \tfrac{1}{2}\sum_l\sum_k i\Im(w_l\overline{w_k} (v_l  ,v_k)_V) \\
 &=(\sum_l w_l  v_l+\sum_k w_k v_k, t_1+t_2+ \tfrac{1}{4}\sum_l\sum_k \Tr_{E/F}(iw_l\overline{w_k} (v_l  ,v_k)_V).
\end{align*}
Moreover, $\ell$ can extend to  an isomorphism:
$$\ell: \Gal(\C/\R)\U( \mathcal{V})\ltimes \Ha(\mathcal{W}\otimes \mathcal{V})   \longrightarrow \Gal(\C/\R)\U(V)\ltimes \Ha(V).$$
By Theorem \ref{scindagedugroupeR0},  $\widehat{\U(V)}$ is splitting. In the next section, we will see the explicit trivialization when $\dim_E V=1$.
\section{Parenthesis: Modular form on the two fold covering space}\label{Parenthesis}
In this  section, let us consider the example when $\dim W=2$.  We mainly follow \cite{LiVe} to deal with theta series. On the other hand, approaching theta series in a more gobal way,   see \cite{We} and \cite{Ge}.

\subsection{Notations}\label{notationss} Assume $\psi=\psi_0$. In this case, $F=\R$, $E=\C=F(i)$,  $\Gal(E/F)=\langle \sigma\rangle$, $F_2=\langle s(-1)\rangle$ for $s(-1)=\begin{pmatrix} 1 & 0\\ 0& -1\end{pmatrix}\in \SL^{\pm}_2(\R)$, $\Sp(W)\simeq \SL_2(\R)$, $\Sp^{\pm}(W) \simeq \SL_2^{\pm}(\R)$, $X=Fe_1\simeq \R$, $X^{\ast}=Fe_1^{\ast}\simeq \R$, $V=E$, $\U(V)=\{ \cos \theta + i\sin \theta\mid -\pi<\theta \leq \pi\}$. By Lemma \ref{iso}, we can treat $\U(V)$ as the subgroup $\SO_2(\R)$ of $\SL_2(\R)$, through the map: $\cos \theta + i\sin \theta \longrightarrow \begin{pmatrix} \cos \theta & \sin \theta\\ -\sin \theta & \cos \theta\end{pmatrix}$. Let:
\begin{itemize}
\item $\sgn: \R \longrightarrow \{\pm 1, 0\}; t \longmapsto \left\{\begin{array}{rl} 1 & \textrm{ if } t>0\\ 0& \textrm{ if } t=0\\ -1 & \textrm{ if } t<0\end{array}\right.$;
\item $P(\R)=\{ g=\begin{pmatrix}a & b\\ 0& a^{-1}\end{pmatrix} \in\SL_2(\R) \}$, $N(\R)=\{ g=\begin{pmatrix}1 & b\\ 0& 1\end{pmatrix} \in\SL_2(\R) \}$, $N_-(\R)=\{ g=\begin{pmatrix}1 & 0\\ c& 1\end{pmatrix} \in\SL_2(\R) \}$;
\item $N$:  a positive integer number bigger than $2$;
\item $\SL_2(\Z)=\{ g=\begin{pmatrix}a & b\\ c& d\end{pmatrix} \in \SL_2(\R)\mid a,b,c,d\in \Z\}$;
\item $ \Gamma(N)=\{\begin{pmatrix}
a& b\\
c& d\end{pmatrix}\in \SL_2(\Z)\mid \begin{pmatrix}
a & b\\
c& d\end{pmatrix}\equiv \begin{pmatrix}
1 & 0\\
0& 1\end{pmatrix} (\bmod N)\}$;
\item $P^{\pm}(\R)=\{ g=\begin{pmatrix}a & b\\ 0& a^{-1} \det g  \end{pmatrix} \in\SL^{\pm}_2(\R) \}$;
\item $P_{>0}^{\pm}(\R)=\{ g=\begin{pmatrix} a & b\\ 0& a^{-1} \det g  \end{pmatrix} \in\SL^{\pm}_2(\R)\mid a>0\}$;
\item $\SL_2^{\pm}(\Z)=\{ g=\begin{pmatrix}a & b\\ c& d\end{pmatrix} \in \SL_2^{\pm}(\R)\mid a,b,c,d\in \Z\}$;
\item $\SO_2^{\pm}(\R)=\{ g= \begin{pmatrix} \cos \theta & \sin \theta\\ \mp\sin \theta & \pm \cos \theta\end{pmatrix}\mid -\pi < \theta \leq \pi\}$;
\item $\mathcal{H}^{\pm}=\{ b+ai\in \C\mid a\neq 0\}$;
\item $\Ha(V)=E\oplus \R$, with the multiplication: $(z, t)(z',t')=(z+z', t+t'+\tfrac{1}{2}i\Im(z\overline{z'}))$;
\item $\Ha(W)=\R\oplus \R \oplus \R$, with the multiplication: $(x, y; t)(x',y'; t')=(x+x', y+y'; t+t'+\tfrac{1}{2}(xy'-x'y))$;
\item $\Gal(\C/\R)\ltimes\Ha(V)=\Gal(\C/\R)\ltimes (\C\oplus \R)$, with the action: $(z, t)^{\sigma}=(z^{\sigma}, -t)$;
\item  $\Ha^{\pm}(W)=F_2\ltimes (\R\oplus \R \oplus \R)$, with the action: $(x, y; t)^{s(-1)}=(x,-y; -t)$;
\item  $\Arg(z)\in (-\pi, \pi]$, for the principal argument function $\Arg$ on $\C$;
\item  $\omega=\begin{pmatrix}0 & -1\\ 1& 0\end{pmatrix}$, $u(b)=\begin{pmatrix}
  1&b\\
  0 & 1
\end{pmatrix}$,  $u_-(c)=\begin{pmatrix}
  1&0\\
  c & 1
\end{pmatrix}$, $h(a)=\begin{pmatrix}
  a&0\\
  0 & a^{-1}
\end{pmatrix} $, $h_{\epsilon}=\begin{pmatrix}
  1&0\\
  0& \epsilon
\end{pmatrix}$.
\end{itemize}
Let $g_1, g_2, g_3=g_1g_2\in \SL_2(\R)$ with $g_i=\begin{pmatrix}
a_i & b_i\\
c_i& d_i\end{pmatrix}$.
\begin{itemize}
\item[(1)] $x(g_1)=\left\{\begin{array}{lr} d_1\R^{\times 2} & \textrm{ if } c_1=0,\\
c_1\R^{\times 2} & \textrm{ if } c_1\neq 0.\end{array}\right.$
\item[(2)] $m_{X^{\ast}}(g_1)=\left\{\begin{array}{lr} e^{\tfrac{i \pi}{4}[1-\sgn(d_1)]}& \textrm{ if } c_1=0,\\
e^{\tfrac{i \pi}{4}[- \sgn(c_1)]}& \textrm{ if } c_1\neq 0.\end{array}\right.$
\item[(3)] $\overline{c}_{X^{\ast}}(g_1, g_2)=(x(g_1), x(g_2))_\R(-x(g_1)x(g_2), x(g_3))_\R$.
\item[(4)] $\widetilde{c}_{X^{\ast}}(g_1, g_2)=\gamma(\psi(q(g_1, g_2)/2))=\gamma(\psi(\tfrac{1}{2}c_1c_2c_3))= e^{\tfrac{i \pi}{4} \sgn(c_1c_2c_3)}=\left\{\begin{array}{lr} 1 & \textrm{ if } c_1c_2c_3=0,\\
e^{\tfrac{i\pi}{4}}& \textrm{ if } c_1c_2c_3>0,\\
e^{-\tfrac{i\pi}{4}}& \textrm{ if } c_1c_2c_3<0.\end{array}\right.$
\end{itemize}
Let  $h_1, h_2, h_3=h_1h_2\in \SL^{\pm}_2(\R)$ with $h_i=\begin{pmatrix}
a_i & b_i\\
c_i& d_i\end{pmatrix}$, $y_i=\det h_i$, $g_i=\begin{pmatrix}
a_i & b_i\\
y_i^{-1}c_i& y_i^{-1}d_i\end{pmatrix}$. Then:
\begin{itemize}
\item $\widetilde{C}_{X^{\ast}}(h_1, h_2) =\nu(y_2,g_1)\widetilde{c}_{X^{\ast}}(g_1^{y_2}, g_2)$.
\item $\overline{C}_{X^{\ast}}(h_1, h_2) =\nu_2(y_2,g_1)\overline{c}_{X^{\ast}}(g_1^{y_2}, g_2)$.
\item $\nu(y_2,g_1)=\left\{\begin{array}{lr} (y_2,  a_1)_{\R} & c_1=0,\\ (y_1c_1, y_2)_\R \gamma(y_2, \psi^{\tfrac{1}{2}})^{-1} & c_1\neq 0.\end{array}\right.$
\item$\nu_2(y_2,g_1)=\left\{\begin{array}{lr}  (y_2,  a_1)_\R & c_1=0,\\ 1 & c_1\neq 0. \end{array}\right.$
\end{itemize}
\begin{lemma}
For $h=\begin{pmatrix}
a& b\\
c& d\end{pmatrix}\in  \SL^{\pm}_2(\R)$, $p=\begin{pmatrix}
a_1& b_1\\
0&   a_1^{-1} \det p \end{pmatrix}\in P^{\pm}_{>0}(\R)$, we have:
\begin{itemize}
\item $\widetilde{C}_{X^{\ast}}(p, h)=1$;
\item $\widetilde{C}_{X^{\ast}}(h, p)=\left\{\begin{array}{cl}  (\det p, a)_{\R}  & \textrm{ if } c=0\\(c\det h, \det p)_{\R} \gamma(\det p, \psi^{\tfrac{1}{2}})^{-1}  & \textrm{ if } c\neq 0 \end{array}\right.$.
\end{itemize}
\end{lemma}
\begin{proof}
 Let   $g_h=\begin{pmatrix}
a& b\\
c\det h& d \det h\end{pmatrix}$, and $g_p=\begin{pmatrix}
a_1& b_1\\
0&  a_1^{-1}\end{pmatrix}$. Then:
$$\widetilde{C}_{X^{\ast}}(p, h)=\nu(\det h,g_p)\widetilde{c}_{X^{\ast}}((g_p)^{\det h}, g_h)=(\det h, a_1)_{\R}=1;$$
$$\widetilde{C}_{X^{\ast}}( h,p)=\nu(\det p,g_h)\widetilde{c}_{X^{\ast}}((g_h)^{\det p}, g_p)=\nu(\det p,g_h)=\left\{\begin{array}{cl} (\det p, a)_{\R} & \textrm{ if } c=0,\\(c\det h, \det p)_{\R} \gamma(\det p, \psi^{\tfrac{1}{2}})^{-1}  & \textrm{ if } c\neq 0. \end{array}\right.$$
\end{proof}
\begin{lemma}
For $h=\begin{pmatrix}
a& b\\
c& d\end{pmatrix}\in  \SL^{\pm}_2(\R)$, $p=\begin{pmatrix}
a_1 & b_1\\
0& a_1^{-1} \det p  \end{pmatrix}\in P^{\pm}_{>0}(\R)$, we have:
\begin{itemize}
\item $\overline{C}_{X^{\ast}}(p, h)=1$;
\item $\overline{C}_{X^{\ast}}(h, p)=\left\{\begin{array}{cl}  (\det p, a)_{\R}  & \textrm{ if } c=0\\1  & \textrm{ if } c\neq 0 \end{array}\right.$.
\end{itemize}
\end{lemma}
\begin{proof}
 Let   $g_h=\begin{pmatrix}
a& b\\
c\det h& d \det h\end{pmatrix}$, and $g_p=\begin{pmatrix}
a_1& b_1\\
0&  a_1^{-1}\end{pmatrix}$. Then:
\begin{align*}
\overline{C}_{X^{\ast}}(p, h)&=\nu_2(\det h,g_p)\overline{c}_{X^{\ast}}((g_p)^{\det h}, g_h)\\
&=(\det h, a_1)_{\R}(x(g_p^{\det h}), x(g_h))_\R(-x(g_p^{\det h})x(g_h), x((g_p)^{\det h}g_h))_\R\\
&=1;
\end{align*}
\begin{align*}
\overline{C}_{X^{\ast}}( h,p)&=\nu_2(\det p,g_h)\overline{c}_{X^{\ast}}((g_h)^{\det p}, g_p)=\nu_2(\det p,g_h)=\left\{\begin{array}{cl} (\det p, a)_{\R} & \textrm{ if } c=0,\\1  & \textrm{ if } c\neq 0. \end{array}\right.
\end{align*}
\end{proof}
\subsection{Splitting of $\widehat{\U}(V)$} Following \cite[pp.75-77]{LiVe},  let us define   the  function $u: \R \longrightarrow \Z$ by
$$ u(\theta)=\left\{ \begin{array}{rl}
2k & \textrm{ if } \theta=k\pi, \\
2k+1  & \textrm{ if } k\pi <\theta<(k+1)\pi.
\end{array}\right.$$
\begin{lemma}
$\sgn[\sin(\theta_1)\sin(\theta_2)\sin(\theta_1+\theta_2)]=u(\theta_1)+u(\theta_2)-u(\theta_1+\theta_2)$.
\end{lemma}
\begin{proof}
See \cite[pp.75-77]{LiVe}.
\end{proof}
Let us modify the above $u$ by a character of $\R$ in the following way:
$$u'(\theta)=u(\theta)+ \tfrac{2}{\pi}\theta.$$
\begin{lemma}\label{theq}
\begin{itemize}
\item[(1)] $u'(\theta+k\pi)=u'(\theta)+4k$, for $k\in \Z$.
\item[(2)] $\sgn[\sin(\theta_1)\sin(\theta_2)\sin(\theta_1+\theta_2)]=u'(\theta_1)+u'(\theta_2)-u'(\theta_1+\theta_2)$.
\end{itemize}
\end{lemma}
\begin{proof}
1) Note that $u(\theta+k\pi)=u(\theta)+2k$. So $u'(\theta+k\pi)=u(\theta+k\pi)+\tfrac{2}{\pi}(\theta+k\pi)=u(\theta)+ 2k+ \tfrac{2}{\pi}\theta+2k=u'(\theta)+4k$.\\
2) Since $u'$ is different from $u$ by a character of $\R$, the result follows from the above lemma.
\end{proof}
\begin{lemma}\label{thetaf}
 The restriction of $[\widetilde{c}_{X^{\ast}}]$ on $\U(V)$ is trivial, with an explicit trivialization:
$$
\widetilde{s}:  \U(V) \longrightarrow \C^{\ast};
  \cos \theta + i\sin \theta \longmapsto   e^{\tfrac{i\pi }{4}u'(\theta)},$$
such that $\widetilde{c}_{X^{\ast}}(g_1, g_2) =\widetilde{s}(g_1)^{-1}\widetilde{s}(g_2)^{-1}\widetilde{s}(g_1g_2)$, for $g_i\in \U(V)$.
\end{lemma}
\begin{proof}
It is well defined by the above lemma (1).
\begin{align*}
\widetilde{c}_{X^{\ast}}(e^{i \theta_1}, e^{i \theta_2}) & =e^{\tfrac{i \pi}{4} \sgn[-\sin(\theta_1)\sin(\theta_2)\sin(\theta_1+\theta_2)]}\\
&\stackrel{ \textrm{ Lemma \ref{theq}(2)}}{=} e^{-\tfrac{i \pi}{4} u'(\theta_1)}e^{-\tfrac{i \pi}{4} u'(\theta_2)} e^{\tfrac{i \pi}{4} u'(\theta_1+\theta_2)}.
\end{align*}
\end{proof}

\begin{lemma}\label{trivial}
The restriction of $[\overline{c}_{X^{\ast}}]$ on $\U(V)$ is  trivial,  with an explicit trivialization:
$\overline{s}(e^{i\theta})=e^{\tfrac{i\theta}{2}}$, for $-\pi<\theta\leq \pi$,
such that $\overline{c}_{X^{\ast}}(g_1, g_2) =\overline{s}(g_1)^{-1}\overline{s}(g_2)^{-1}\overline{s}(g_1g_2)$, for $g_i\in \U(V)$.
\end{lemma}
\begin{proof}
By (\ref{28inter}) and the above lemma, we have:
\begin{align*}
\overline{c}_{X^{\ast}}(g_1, g_2)&=m_{X^{\ast}}(g_1g_2)^{-1} m_{X^{\ast}}(g_1) m_{X^{\ast}}(g_2) \widetilde{c}_{ {X^{\ast}}}(g_1,g_2)\\
&=[\widetilde{s}(g_1g_2)m_{X^{\ast}}(g_1g_2)^{-1}] [\widetilde{s}(g_1)m_{X^{\ast}}(g_1)^{-1}]^{-1} [\widetilde{s}(g_2)m_{X^{\ast}}(g_2)^{-1}]^{-1}.
\end{align*}
Let $\overline{s}(e^{i\theta})=\widetilde{s}(e^{i\theta})m_{X^{\ast}}(e^{i\theta})^{-1}=\left\{ \begin{array}{rl}
 1 &\textrm{ if } \theta=0,\\
e^{\tfrac{i\theta}{2}}  & \textrm{ if } 0< \theta< \pi,\\
e^{\tfrac{i\theta}{2}}  &\textrm{ if } -\pi<\theta< 0,\\
e^{\tfrac{i\theta}{2}}& \textrm{ if } \theta= \pi.
\end{array}\right.$ Then $\overline{c}_{X^{\ast}}(g_1, g_2)=\overline{s}(g_1g_2) \overline{s}(g_1)^{-1} \overline{s}(g_2)^{-1}$.
\end{proof}
\begin{remark}\label{triviall}
\begin{itemize}
\item[(1)] The trivialization map for  $\widetilde{c}_{X^{\ast}}$ can not be valued in $\mu_8$.
\item[(2)] The trivialization map for  $\overline{c}_{X^{\ast}}$ can not be valued in $\mu_2$.
\end{itemize}
\end{remark}
\begin{proof}
We only prove the first statement. If there exists another trivialization $ \widetilde{s}': \U(V) \longrightarrow \mu_8$, then $\widetilde{s}$ and $\widetilde{s}'$ will differ by a measurable group homomorphism $\chi:\U(V) \longrightarrow T$. Since $ \U(V)$ is separable, $\chi$ is also continuous. It is known that $\chi(e^{i\theta})=e^{i n \theta}$, for some $n\in \Z$. Then $e^{i\theta n-i\theta \tfrac{1}{2}} \in \mu_8$, for all $\theta\in \R$; this is impossible.
\end{proof}

\subsection{Spin group}
Let us go into the low  spin group $\Spin_2(\R)$. By \cite[pp.18-19]{Se}, there exists an exact sequence
\begin{align}\label{spin}
1\longrightarrow \{\pm 1\} \longrightarrow \Spin_2(\R) \stackrel{\mathcal{A}}{\longrightarrow } \SO_2(\R) \to 1,
\end{align}
where $\Spin_2(\R)=\{ \cos \theta +k \sin \theta\mid k^2=-1\}$, and $\mathcal{A}(\cos \theta +k \sin \theta)=\begin{pmatrix}
\cos 2\theta & \sin 2\theta \\
-\sin 2\theta & \cos 2\theta \end{pmatrix}$.  Note that $\Ha^2(\SO_2(\R), T)=0$ and $\Ha^{2}(\SO_2(\R), \mu_2)\simeq \mu_2$.  By Remark \ref{triviall}, $[\overline{c}_{X^{\ast}}]$ is the unique nontrivial class in $\Ha^{2}(\SO_2(\R), \mu_2)$. Hence $\overline{c}_{X^{\ast}}$ determines the above exact sequence (\ref{spin}). Let us identity $T$ with $\SO_2(\R)$ by $e^{i\theta} \longrightarrow  \begin{pmatrix}
\cos \theta & \sin \theta \\
-\sin \theta & \cos \theta \end{pmatrix}$.  Let
$$1 \longrightarrow \{\pm 1\} \longrightarrow \widetilde{T} \longrightarrow T \longrightarrow 1$$
be the exact sequence associated to $\overline{c}_{X^{\ast}}$. Then $\Spin_2(\R) \simeq \widetilde{T}$. Note that there exists a group homomorphism:  $\sqrt{\,}: \widetilde{T} \stackrel{\sim}{\longrightarrow} T$.
\subsection{Non-splitting of $\widehat{\SO^{\pm}_2}(\R)$}
\begin{lemma}\label{thetaf3}
 The restriction of $[\overline{C}_{X^{\ast}}]$ on $\SO^{\pm}_2(\R)$ is  not trivial.
\end{lemma}
\begin{proof}
Note that there exists an exact sequence:
$$1\longrightarrow\SO_2(\R) \longrightarrow \SO^{\pm}_2(\R) \longrightarrow \mu_2 \longrightarrow 1.$$
By  Hochschild-Serre spectral sequence,  there exists  the following long exact sequence  of six terms:
\begin{align}\label{al22}
0 \longrightarrow \Hom(\mu_2,  T) \to \Hom(\SO^{\pm}_2(\R),  T)  \longrightarrow  \Hom(\SO_2(\R), T)^{\SO^{\pm}_2(\R)} \longrightarrow
\Ha^2( \mu_2,T) \\
\longrightarrow \Ha^{2}( \SO^{\pm}_2(\R),T)_1 \stackrel{p}{\longrightarrow} \Ha^1\big(  \mu_2, \Ha^1( \SO_2(\R), T)\big).
\end{align}
By Lemma \ref{trivial},
$$\overline{C}_{X^{\ast}}(g_1, g_2)= \overline{s}(g_1g_2)\overline{s}(g_1)^{-1}\overline{s}(g_2)^{-1}, \qquad \quad g_i\in \SO_2(\R).$$
Let $\overline{s}$   extend to be a Borel  function of $\SO^{\pm}_2(\R)$ by  taking the trivial value outside $\SO_2(\R)$. We replace $\overline{C}_{X^{\ast}}$ with $\overline{C}_{X^{\ast}}'= \overline{C}_{X^{\ast}} \circ \delta_1 \overline{s}$. Then:
$$\overline{C}_{X^{\ast}}'(h_1, h_2)= \overline{C}_{X^{\ast}}(h_1, h_2)\overline{s}(h_1) \overline{s}(h_2) \overline{s}(h_1h_2)^{-1}, \quad\quad  h_1, h_2 \in \SO^{\pm}_2(\R).$$
Recall the section map $s: \mu_2 \longrightarrow \SO^{\pm}_2(\R); \epsilon \longmapsto \begin{pmatrix} 1& 0\\ 0 &\epsilon \end{pmatrix}$. Then for $g\in \SO_2(\R)$, we have:
$$\overline{C}_{X^{\ast}}'(s(1), g)= \overline{C}_{X^{\ast}}(s(1),g)= \overline{C}_{X^{\ast}}([1,1],[1,g])=1;$$
$$ \overline{C}_{X^{\ast}}'(s(-1), g)= \overline{C}_{X^{\ast}}(s(-1),g) \overline{s}(g)= \overline{C}_{X^{\ast}}([-1,1],[1,g]) \overline{s}(g)=\widetilde{c}_{X^{\ast}}(1, g) \overline{s}(g)= \overline{s}(g).$$
For any $h\in \SO^{\pm}_2(\R)$, let us write  $h=s(\epsilon_h)g_h$, for $\epsilon_h \in \mu_2$, $g_h\in \SO_2(\R)$. Define a Borel function $\varkappa$ of $\SO_2^{\pm}(\R)$ as
$ \varkappa(h)= \overline{C}_{X^{\ast}}'(s_1(\epsilon_h), g_h)$. Let us define
$$\overline{C}_{X^{\ast}}''= \overline{C}_{X^{\ast}}' \circ \delta_1(\varkappa).$$ So the map $p$ is just given by
$$p([\overline{C}_{X^{\ast}}]): \epsilon  \longmapsto ( g  \longrightarrow \overline{C}_{X^{\ast}}''(g, s(\epsilon))), \quad g\in \SO_2(\R), \epsilon \in \mu_2.$$
Note that $p([\overline{C}_{X^{\ast}}])(\epsilon )$ is a character of $\SO_2(\R)$. Let $g=\begin{pmatrix} \cos \theta& \sin \theta\\ -\sin \theta &\cos \theta \end{pmatrix}\in \SO_2(\R)$, $g^{s(-1)} =\begin{pmatrix} \cos( -\theta)& \sin (-\theta)\\ -\sin (-\theta) &\cos (-\theta) \end{pmatrix}$.   Then:
\begin{align}
\overline{C}_{X^{\ast}}''(g, s(-1))&=\overline{C}_{X^{\ast}}'(g, s(-1))\varkappa(g)\varkappa(s(-1))\varkappa(gs(-1))^{-1}\\
&=\overline{C}_{X^{\ast}}'(g, s(-1))\overline{C}_{X^{\ast}}'(s(1), g)\overline{C}_{X^{\ast}}'(s(-1),1)\overline{C}_{X^{\ast}}'(s(-1), s(-1)^{-1}gs(-1))^{-1}\\
&=\overline{C}_{X^{\ast}}(g, s(-1)) \overline{s}(g)\overline{s}(g^{s(-1)})^{-1}\\
&=\left\{ \begin{array}{lr} \overline{s}(g)\overline{s}(g^{s(-1)})^{-1}= e^{i\theta} & \textrm{ if } -\pi<\theta< \pi \\
-1 & \textrm{ if } \theta= \pi\end{array}\right.\\
&=e^{i\theta}.
\end{align}
Note that there does not exist a character $f$ of $\SO_2(\R)$ such that $ e^{i\theta}= f(e^{i\theta}) f^{-1}((e^{i\theta})^{s(-1)})=f(e^{2i\theta})$. So In the above exact sequence of six terms, $p([ \overline{C}_{X^{\ast}}]) \neq 0$, but $\Ha^2( \mu_2,T)=0$, which implies that $[ \overline{C}_{X^{\ast}}]\neq 0$.
\end{proof}

\subsection{Non-splitting of $\overline{\Gamma^{\pm}(2)}$  I}\label{nonspI}

\begin{lemma}\label{nonsp}
 The restriction of $[\overline{C}_{X^{\ast}}]$ on $\Gamma^{\pm}(2)$ is not trivial.
\end{lemma}
\begin{proof}
The proof is similar to that of Lemma \ref{thetaf3} by replacing $\SO_2(\R)$ with $\Gamma^{\pm}(2)$. We use the notations from there. For $g=\begin{pmatrix}a & b\\ c & d \end{pmatrix} \in \Gamma(2)$, $g^{s(-1)} =\begin{pmatrix} a &-b\\ -c &d\end{pmatrix}$,
\begin{align}
\overline{C}_{X^{\ast}}''(g, s(-1))&=\overline{C}_{X^{\ast}}(g, s(-1)) \overline{\beta}(g)\overline{\beta}(g^{s(-1)})^{-1}.
\end{align}
If $c=0$, $\overline{\beta}(g)\overline{\beta}(g^{s(-1)})^{-1}=1$, $\overline{C}_{X^{\ast}}(g, s(-1))= (a,-1)_{F}$, so $\overline{C}_{X^{\ast}}''(g, s(-1))=(a,-1)_F$.\\
If $c\neq 0$, $\overline{C}_{X^{\ast}}(g, s(-1))=1$,
\begin{align*}
\overline{\beta}(g)\overline{\beta}(g^{s(-1)})^{-1}&=\beta(d,c)e^{\tfrac{\pi i \sgn(c)}{4}} \beta(d,c)e^{\tfrac{\pi i \sgn(c)}{4}}\\
&=\beta(d,c)^2 e^{\tfrac{\pi i \sgn(c)}{2}}\\
&=\left\{\begin{array}{lr} 1 &  d\equiv 1(\bmod 4),\\ -1 &  d\equiv 3(\bmod 4).  \end{array}\right.
\end{align*}
Similarly, $[ \overline{C}_{X^{\ast}}]\neq 0$ on $\Gamma^{\pm}(2)$. Let $\overline{\Gamma^{\pm}(2)}$ denote the two-degree  covering group over $\Gamma^{\pm}(2)$ associated to $\overline{C}_{X^{\ast}}$.
\end{proof}

\subsection{Two models} Let $\Pi_{\psi}$ be the Heisenberg representation of $\Ha^{\pm}(W)$ as  defined in Section \ref{sch} and Section \ref{EWR}. It can be realized on the vector spaces $\mathcal{H}^{\pm}(Y) =\Ind_{Y\times F}^{\Ha^{\pm}(W)}\psi_{Y}$ and $\mathcal{H}^{\pm}(L)=\Ind_{\Ha(L)}^{\Ha^{\pm}(W)} \psi_L$, for any $Y\in \Lambda$, $L\in \mathcal{L}_{\psi}$, where $\psi_{Y}$ (resp.  $\psi_L$) is an extending character of $Y\times F$(resp. $\Ha(L)$) from $\psi$ of $F$.

In particular, take $Y^{\ast}=X^{\ast}$, $Y=X$ and $L=\Z e_1\oplus \Z e_1^{\ast}$. Note that the actions of $\Ha^{\pm}(W)$ on $\mathcal{H}^{\pm}(X^{\ast})$ and $\mathcal{H}^{\pm}(L)$ are given by the formulas (\ref{representationsp21}) (\ref{representationsp22}) and (\ref{www1}) (\ref{www2}) respectively. By \cite[pp.164-165]{We}, or \cite[pp.142-145]{LiVe}, there exists  a pair of explicit isomorphisms between   $\mathcal{H}^{\pm}(X^{\ast})\simeq L^2(\mu_2\times X)$ and  $\mathcal{H}^{\pm}(L)$, given as follows:
 \begin{equation}\label{eq7}
\theta_{L, X^{\ast}}(f')(\epsilon, w)=\sum_{l\in L/L\cap X^{\ast}} f'(\epsilon,  w+l)\psi(\epsilon\tfrac{\langle l, w\rangle}{2}+\tfrac{\langle x_{l}, x^{\ast}_l\rangle}{2}),
 \end{equation}
 \begin{equation}\label{eq8}
 \begin{split}
 \theta_{X^{\ast}, L}(f)(\epsilon,y)&=  \int_{X^{\ast}/X^{\ast}\cap L} f([\dot{y}^{\ast},0]+[(\epsilon,y), 0]) d\dot{y}^{\ast}\\
 &= \int_{X^{\ast}/X^{\ast}\cap L} \psi(\epsilon\tfrac{\langle\dot{y}^{\ast}, y\rangle}{2})f(\epsilon,y+\epsilon\dot{y}^{\ast}) d\dot{y}^{\ast},
 \end{split}
      \end{equation}
for $ w=x+ x^{\ast} \in W$, $y\in X$,  $f'\in S(\mu_2\times X)\subseteq L^2(\mu_2 \times X)$,  $f\in L^1(W)\cap \mathcal{H}^{\pm}(L)$. It can be checked that $\theta_{L, X^{\ast}}$ and $\theta_{X^{\ast}, L}$ define a pair of inverse intertwining operators between $ \Ind_{\Ha(L)}^{\Ha^{\pm}(W)} \psi_L$ and  $\Ind_{X^{\ast} \times \R}^{\Ha^{\pm}(W)} \psi_{X^{\ast}}$.
For $g\in  \Sp^{\pm}(W)$, we can define
     \begin{equation}
     \Pi_{\psi}(g)f=\theta_{L, X^{\ast}}[\Pi_{\psi}(g)(f')]=\theta_{L, X^{\ast}}[\Pi_{\psi}(g)\theta_{X^{\ast},L}(f)].
     \end{equation}
Under such action, for $g_1, g_2\in \Sp^{\pm}(W)$,
     \begin{equation}
     \begin{split}
     \Pi_{\psi}(g_1)[\Pi_{\psi}(g_2)f]&=\theta_{L, X^{\ast}}[\Pi_{\psi}(g_1)\theta_{X^{\ast},L}]([\Pi_{\psi}(g_2)f])\\
     &=\theta_{L, X^{\ast}}([\Pi_{\psi}(g_1)\theta_{X^{\ast},L}\theta_{L, X^{\ast}}
     [\Pi_{\psi}(g_2)\theta_{X^{\ast},L}(f)]))\\
     &=\theta_{L, X^{\ast}}[\Pi_{\psi}(g_1)\Pi_{\psi}(g_2)\theta_{X^{\ast},L}(f)]\\
     &=\widetilde{C}_{X^{\ast}}(g_1, g_2)\Pi_{\psi}(g_1g_2)f.
     \end{split}
     \end{equation}
\subsection{Non-splitting of $\overline{\Gamma^{\pm}(2)}$  II}\label{Gamma2}
Let us consider  $g\in \Sp^{\pm}(L)$. Under the basis $\{ e_1,  e_1^{\ast}\}$ of $\mathcal{L}$, $\Sp^{\pm}(L)\simeq \SL_2^{\pm}(\Z)$.  Let $w=x+x^{\ast} \in W$.  \\
Case 1: $g=u(b)\in \Gamma^{\pm}(2)$, and $2\mid b$.
\begin{equation*}
\begin{split}
\Pi_{\psi}(g)f(\epsilon, w)&=\sum_{l\in L/L\cap X^{\ast}}\Pi_{\psi}(g)(f')(\epsilon,  w+l)\psi(\epsilon\tfrac{\langle l, w\rangle}{2}+\tfrac{\langle x_{l}, x^{\ast}_l\rangle}{2})\\
&=\sum_{l\in L\cap X}\Pi_{\psi}(g)(f')(\epsilon,  w+l)\psi(\epsilon\tfrac{\langle l, w\rangle}{2})\\
&=\sum_{l\in L\cap X}\Pi_{\psi}(g)(f')([1,(\epsilon x^{\ast},\epsilon\tfrac{\langle x+ l, x^{\ast}\rangle}{2}) ][ \epsilon,   x+l])\psi(\epsilon\tfrac{\langle l, w\rangle}{2})\\
&=\sum_{l\in L\cap X}\Pi_{\psi}(g)(f')(\epsilon, x+l)\psi(\epsilon\langle l, w\rangle) \psi(\epsilon\tfrac{\langle x, x^{\ast}\rangle}{2}) \\
&=\sum_{l\in L\cap X}\psi(\tfrac{1}{2}\langle (x+l),\epsilon (x+l)b\rangle) f'(\epsilon,x+l)\psi(\epsilon\langle l, w\rangle)\psi(\epsilon\tfrac{\langle x, x^{\ast}\rangle}{2})\\
&= \sum_{l\in L\cap X} \int_{X^{\ast}/X^{\ast}\cap L}  \psi(\tfrac{1}{2}\langle (x+l),\epsilon (x+l)b\rangle) \psi(\epsilon\tfrac{\langle\dot{y}^{\ast}, x+l\rangle}{2})f(\epsilon,x+l+\epsilon\dot{y}^{\ast})  \psi(\epsilon\langle l, w\rangle)\psi(\epsilon\tfrac{\langle x, x^{\ast}\rangle}{2}) d\dot{y}^{\ast}\\
&=  \sum_{l\in L\cap X} \int_{X^{\ast}/X^{\ast}\cap L}   \psi(\epsilon\tfrac{\langle\dot{y}^{\ast}, x+l\rangle}{2})f(\epsilon,x+l+\epsilon\dot{y}^{\ast})  \psi(\epsilon\langle l, x^{\ast}+xb\rangle)\psi(\epsilon\tfrac{\langle x, x^{\ast}+\epsilon xb\rangle}{2}) d\dot{y}^{\ast}\\
&= f([\epsilon, w]g)=\nu(\epsilon,g)f([\epsilon, w]g).
\end{split}
\end{equation*}
Case 2: $g=h(a)\in \Gamma^{\pm}(2)$, $a\in \{\pm 1\}$.
\begin{equation}
\begin{split}
\Pi_{\psi}(g)f(\epsilon, w)&=\sum_{l\in L\cap X}\Pi_{\psi}(g)(f')(\epsilon, x+l)\psi(\epsilon\langle l, w\rangle) \psi(\epsilon\tfrac{\langle x, x^{\ast}\rangle}{2}) \\
&=\sum_{l\in L\cap X}|a|^{1/2}(a,\epsilon)_Ff'(\epsilon, (x+l)a)\psi(\epsilon\langle l, w\rangle) \psi(\epsilon\tfrac{\langle x, x^{\ast}\rangle}{2})\\
&= \sum_{l\in L\cap X}\int_{X^{\ast}/X^{\ast}\cap L} \psi(\epsilon\tfrac{\langle\dot{y}^{\ast}, (x+l)a\rangle}{2})f(\epsilon,xa+la+\epsilon\dot{y}^{\ast}) (a,\epsilon)_F \psi(\epsilon\langle l, w\rangle) \psi(\epsilon\tfrac{\langle x, x^{\ast}\rangle}{2})d\dot{y}^{\ast}\\
&= \sum_{l\in L\cap X}\int_{X^{\ast}/X^{\ast}\cap L} \psi(\epsilon\tfrac{\langle\dot{y}^{\ast}, (x+l)a\rangle}{2})f(\epsilon,xa+la+\epsilon\dot{y}^{\ast}) (a,\epsilon)_F \psi(\epsilon\langle la, wa\rangle) \psi(\epsilon\tfrac{\langle xa, x^{\ast}a\rangle}{2}) d\dot{y}^{\ast}\\
&=(a,\epsilon)_F f(\epsilon, wa)=(a,\epsilon)_Ff([\epsilon,w]g)=\nu(\epsilon,g)f([\epsilon, w]g).
\end{split}
\end{equation}
Case 3: $g=h_{-1} \in  \Gamma^{\pm}(2)$.
\begin{equation}
\begin{split}
\Pi_{\psi}(g)f(\epsilon, w)&=\sum_{l\in L\cap X}\Pi_{\psi}(g)(f')(\epsilon, x+l)\psi(\epsilon\langle l, w\rangle) \psi(\epsilon\tfrac{\langle x, x^{\ast}\rangle}{2}) \\
&=\sum_{l\in L\cap X}f'(-\epsilon, x+l)\psi(\epsilon\langle l, w\rangle) \psi(\epsilon\tfrac{\langle x, x^{\ast}\rangle}{2}) \\
&= \sum_{l\in L\cap X}\int_{X^{\ast}/X^{\ast}\cap L} \psi(-\epsilon\tfrac{\langle\dot{y}^{\ast}, x+l\rangle}{2})f(-\epsilon,x+l-\epsilon\dot{y}^{\ast}) \psi(\epsilon\langle l, w\rangle) \psi(\epsilon\tfrac{\langle x, x^{\ast}\rangle}{2})d\dot{y}^{\ast}\\
&= \sum_{l\in L\cap X}\int_{X^{\ast}/X^{\ast}\cap L} \psi(-\epsilon\tfrac{\langle\dot{y}^{\ast}, x+l\rangle}{2})f(-\epsilon,x+l-\epsilon\dot{y}^{\ast}) \psi(-\epsilon\langle l, -x^{\ast}\rangle) \psi(-\epsilon\tfrac{\langle x, -x^{\ast}\rangle}{2})d\dot{y}^{\ast}\\
&=f(-\epsilon, x-x^{\ast})\\
&= f([\epsilon, w]g).
\end{split}
\end{equation}
Case 4: $g=\omega\in  \SL_2(\Z)$.
\begin{equation}\label{eq3}
\begin{split}
&\Pi_{\psi}(g)f([\epsilon, w])\\
&=\sum_{l\in L\cap X}\pi_{\psi}(g)(f')(\epsilon, x+l)\psi(\epsilon\langle l, w\rangle)\psi(\epsilon\tfrac{\langle x, x^{\ast}\rangle}{2}) \\
&=\sum_{l\in L\cap X} \int_{X^{\ast}} \nu(\epsilon,g) f'(\epsilon, y^{\ast}\omega^{-1}) \psi(\langle \epsilon(x+l), y^{\ast}\rangle)\psi(\epsilon\langle l, w\rangle)\psi(\epsilon\tfrac{\langle x, x^{\ast}\rangle}{2})dy^{\ast}\\
&=\sum_{l\in L\cap X} \int_{X^{\ast}} \nu(\epsilon,g) f'(\epsilon, y^{\ast}\omega^{-1}) \psi(\epsilon\langle x, y^{\ast}\rangle) \psi(\epsilon\langle l, x^{\ast}+y^{\ast}\rangle)\psi(\epsilon\tfrac{\langle x, x^{\ast}\rangle}{2})dy^{\ast}\\
&=\sum_{l\in L\cap X} \int_{X^{\ast}}  \int_{X^{\ast}/X^{\ast}\cap L}\nu(\epsilon,g)  f(\epsilon, y^{\ast}\omega^{-1}+\epsilon\dot{z}^{\ast})\psi(\epsilon\tfrac{\langle  \dot{z}^{\ast}, y^{\ast}\omega^{-1}\rangle}{2}) \psi(\epsilon\langle x, y^{\ast}\rangle) \psi(\epsilon\langle l, x^{\ast}+y^{\ast}\rangle)\psi(\epsilon\tfrac{\langle x, x^{\ast}\rangle}{2})dy^{\ast}dz^{\ast}\\
&= \sum_{l^{\ast}\in L\cap X^{\ast}}\int_{X^{\ast}/X^{\ast}\cap L} \nu(\epsilon,g)f(\epsilon, (-x^{\ast}+l^{\ast})\omega^{-1}+\epsilon\dot{z}^{\ast})\psi(\epsilon\tfrac{\langle \dot{z}^{\ast}, (-x^{\ast}+l^{\ast})\omega^{-1} \rangle}{2}) \psi(\epsilon\langle x, -x^{\ast}+l^{\ast}\rangle) \psi(\epsilon\tfrac{\langle x, x^{\ast}\rangle}{2})dz^{\ast}\\
&= \sum_{l\in L\cap X}\int_{X^{\ast}/X^{\ast}\cap L} \nu(\epsilon,g)f(\epsilon,-x^{\ast}\omega^{-1}+l+\epsilon\dot{z}^{\ast})\psi(\epsilon\tfrac{\langle \dot{z}^{\ast}, -x^{\ast}\omega^{-1}+l \rangle}{2}) \psi(\epsilon\langle x, -x^{\ast}+l\omega\rangle) \psi(\epsilon\tfrac{\langle x, x^{\ast}\rangle}{2})dz^{\ast}\\
&= \sum_{l\in L\cap X}\int_{X^{\ast}/X^{\ast}\cap L} \nu(\epsilon,g)f(\epsilon,x^{\ast}\omega+l+\epsilon\dot{z}^{\ast})\psi(\epsilon\tfrac{\langle\dot{z}^{\ast},  x^{\ast}\omega+l\rangle}{2}) \psi(\epsilon\langle l, x\omega\rangle) \psi(-\epsilon\tfrac{\langle x, x^{\ast}\rangle}{2})dz^{\ast}\\
&= \sum_{l\in L\cap X}\int_{X^{\ast}/X^{\ast}\cap L}\nu(\epsilon,g) f(\epsilon,x^{\ast}\omega+l+\epsilon\dot{z}^{\ast})\psi(\epsilon\tfrac{\langle \dot{z}^{\ast}, x^{\ast}\omega+l\rangle}{2}) \psi(\epsilon\langle l, x\omega\rangle) \psi(\epsilon\tfrac{\langle  x^{\ast}\omega, x\omega\rangle}{2})dz^{\ast}\\
&=\nu(\epsilon,g)f([\epsilon,wg])=\nu(\epsilon,g)f([\epsilon,w]g).
\end{split}
\end{equation}
Case 5: $g=u_-(c) =\omega^{-1} u(-c) \omega\in  \Gamma^{\pm}(2)$ with $2\mid c$. Then:
\begin{equation}\label{eq4}
\begin{split}
\Pi_{\psi}(g)f(\epsilon,w)&=\Pi_{\psi}(\omega^{-1} u(-c) \omega) f(\epsilon,w)\\
&=\widetilde{C}_{X^{\ast}}(\omega^{-1}, u(-c))^{-1}\widetilde{C}_{X^{\ast}}(\omega^{-1}u(-c), \omega)^{-1}\Pi_{\psi}(\omega^{-1})\Pi_{\psi}(u(-c)) \Pi_{\psi}(\omega) f(\epsilon,w)\\
&=\widetilde{C}_{X^{\ast}}(\omega^{-1}u(-c), \omega)^{-1}\Pi_{\psi}(\omega^{-1})\Pi_{\psi}(u(-c)) \Pi_{\psi}(\omega) f(\epsilon,w)\\
&=\widetilde{C}_{X^{\ast}}(\omega^{-1}u(-c), \omega)^{-1} \nu(\epsilon, \omega^{-1})\nu(\epsilon, \omega)f([\epsilon,wg])\\
&=\widetilde{c}_{X^{\ast}}(\omega^{-1}u(-c), \omega)^{-1} \nu(\epsilon, \omega^{-1})\nu(\epsilon, \omega)f([\epsilon,wg])\\
&=e^{\tfrac{\pi i \sgn(c)}{4}}f([\epsilon,wg]).
\end{split}
\end{equation}
Write $f(1, w)=f_1(w)$ and $f(-1, w)=f_2(w)$. Let us denote $\breve{\Gamma}(2)^{\pm}=\langle \omega\rangle \ltimes \Gamma(2)^{\pm}$.  Recall the results from Section \ref{sl2z}. Let $g\in \breve{\Gamma}(2)$. Then:
\subsubsection{}
\begin{itemize}
\item $\Pi_{\psi}(g) f_1( w)= \pi_{\psi}(g)f_1(w)= \widetilde{\beta}(g)^{-1} f_1(wg)$.
\item  $\Pi_{\psi}(h_{-1}) f_1(w)=\Pi_{\psi}(h_{-1}) f(1, w)=f(-1, wh_{-1})=f_2(wh_{-1})$.
\end{itemize}
\subsubsection{}
\begin{itemize}
\item $\Pi_{\psi}(h_{-1})f_2(w)= \Pi_{\psi}(h_{-1}) f(-1, w)=f(1, wh_{-1})=f_1(wh_{-1})$.
\item
\begin{align*}
\Pi_{\psi}(g) f_2(w)&= \Pi_{\psi}(g) \Pi_{\psi}(h_{-1}) f_1(wh_{-1})\\
&=\widetilde{C}_{X^{\ast}}(g, h_{-1})\Pi_{\psi}(gh_{-1})f_1(wh_{-1})\\
&=\widetilde{C}_{X^{\ast}}(g, h_{-1})\Pi_{\psi}(h_{-1}g^{s(-1)})f_1(wh_{-1})\\
&=\widetilde{C}_{X^{\ast}}(g, h_{-1})\widetilde{C}( h_{-1}, g^{s(-1)})^{-1}\Pi_{\psi}(h_{-1})\Pi_{\psi}(g^{s(-1)})f_1(wh_{-1})\\
&=\widetilde{C}_{X^{\ast}}(g, h_{-1}) \widetilde{\beta}(g^{s(-1)})^{-1}f_2(wg^{s(-1)}).
\end{align*}
\end{itemize}
\subsubsection{}
If $\det g=1$, write $h_{\epsilon}=\begin{pmatrix} 1 &0 \\ 0 & \epsilon\end{pmatrix}$,  then
\begin{align*}
\Pi_{\psi}(g) f([\epsilon, w])&=\widetilde{C}_{X^{\ast}}(g, h_{\epsilon}) \widetilde{\beta}(g^{s(\epsilon)})^{-1}f([\epsilon, wg^{s(\epsilon)}]).
\end{align*}
\subsubsection{}
If $\det g=-1$, let us write $g=h_{-1} (h_{-1} g)$. Then:
\begin{align*}
\Pi_{\psi}(g) f([\epsilon, w])&=\widetilde{C}_{X^{\ast}}(h_{-1}, h_{-1} g)^{-1}\Pi_{\psi}(h_{-1} ) [\Pi_{\psi}(h_{-1}g )f]([\epsilon, w])\\
&=[\Pi_{\psi}(h_{-1}g )f]([-\epsilon, wh_{-1}])\\
&=\widetilde{C}_{X^{\ast}}(h_{-1}g, h_{-\epsilon}) \widetilde{\beta}(h_{-1}g^{s(-\epsilon)})^{-1}f([-\epsilon, wg^{s(-\epsilon)}])\\
&=\widetilde{C}_{X^{\ast}}(g, h_{-\epsilon}) \widetilde{\beta}(h_{-1}g^{s(-\epsilon)})^{-1}f([-\epsilon, wg^{s(-\epsilon)}]).
\end{align*}
\subsubsection{}
Let  $g=\begin{pmatrix}
a& b\\ c& d \end{pmatrix} \in \breve{\Gamma}(2)^{\pm}$. \\
If $c=0$, then
\begin{align*}
\Pi_{\psi}(g) f([\epsilon, w])=((\det g)\epsilon, a)_{\R} f([(\det g) \epsilon, wg^{s((\det g)\epsilon)}]).
\end{align*}
If $c\neq 0$ and  $\det g=1$, then
\begin{align*}
\Pi_{\psi}(g) f([\epsilon, w])= e^{\tfrac{\pi i(1-\epsilon)}{4}}(c, \epsilon)_{\R} \beta(d, \epsilon c)^{-1}_{\R} f([\epsilon, wg^{s(\epsilon)}]).
\end{align*}
If $c\neq 0$ and   $\det g=-1$, then $h_{-1}g^{s(-\epsilon)}=\begin{pmatrix} a &-\epsilon b \\ -\epsilon c & d\end{pmatrix}$, and 
\begin{align*}
\Pi_{\psi}(g) f([\epsilon, w])&=\widetilde{C}_{X^{\ast}}(g, h_{-\epsilon}) \widetilde{\beta}(h_{-1}g^{s(-\epsilon)})^{-1}f([-\epsilon, wg^{s(-\epsilon)}])\\
&=e^{\tfrac{\pi i(1-\epsilon)}{4}}(-c, -\epsilon)_{\R}\beta(d,  -\epsilon c)^{-1}f([-\epsilon, wg^{s(-\epsilon)}]).
\end{align*}
We conclude:
\begin{align}
\Pi_{\psi}(g) f([\epsilon, w])=\Upsilon(g) f([(\det g) \epsilon, wg^{s((\det g)\epsilon)}]),
\end{align}
for the constant $$\Upsilon(g) = \left\{ \begin{array}{lr} ((\det g)\epsilon, a)_{\R} & \textrm{ if } c= 0, \\  e^{\tfrac{\pi i(1-\epsilon)}{4}} (c\det g, \epsilon\det g)_{\R} \beta(d, (\det g) \epsilon c)^{-1} & \textrm{ if } c\neq 0. \end{array}\right.$$
\subsection{Lattice}(cf. \cite{De}, \cite{DeSe})
Let $w_1=ae_1+be_1^{\ast}$, $w_2=ce_1+de_1^{\ast}$  be two elements   of $W$. Then $\Za w_1+\Za w_2$ forms a  lattice in $W$ iff $w_1$ and $w_2$ are $\R$-linear independence iff $\begin{pmatrix}a & b\\ c & d \end{pmatrix}\in \GL_2(\R)$. Moreover, it is a $\psi$-self dual lattice iff $\langle w_1, w_2\rangle=1$ or $-1$ iff $\begin{pmatrix}a & b\\ c & d \end{pmatrix} \in \SL_2^{\pm}(\R)$. Let $\mathcal{L}_{\psi}$ denote the set of all $\psi$-self dual lattices in $W$. We can define a right $\SL_2^{\pm}(\R)$-action on $\mathcal{L}_{\psi}$ as follows:
$$\begin{array}{rcl}
  \mathcal{L}_{\psi} \times \SL_2^{\pm}(\R)  &\longrightarrow &\mathcal{L}_{\psi} ; \\
  ( L= \Za w_1  \oplus \Za w_2, g)& \longmapsto & Lg=\Za  (w_1g)  \oplus \Za (w_2g).
\end{array}$$
Let  $L_0=\Za e_1  \oplus \Za e_1^{\ast} \in \mathcal{L}_{\psi}$. Then:
\begin{itemize}
\item  $\mathcal{L}_{\psi}=\{L_0 g \mid g\in \SL_2^{\pm}(\R)\}$.
\item $L_0g=L_0$ iff $g\in \SL_2^{\pm }(\Z)$.
\item There exists a bijection: $  \SL_2^{\pm }(\Z) \setminus \SL_2^{\pm}(\R) \longrightarrow  \mathcal{L}_{\psi}; g \longmapsto L_0 g$.
\end{itemize}
\subsection{Lagrangian Grassmanian}
By Section \ref{un}, we can identity $\C$ with $W$ by $a+ib \longrightarrow ae_1+be_1^{\ast}$, and define  $\langle z,z'\rangle_{\C}=i\Im(z\overline{z}')$. A Lagrangian plane   of $(\C, \langle, \rangle_{\C})$ is a set $\R z$, for some $z \in \C^{\ast}$. Let $\Lambda$ denote the set of all  Lagrangians in $\C$. Then there exists a
left $\C^{\ast}$-action on $\Lambda$ given as follows:
 $$\begin{array}{rcl}
  \C^{\ast}\times \Lambda &\longrightarrow &\Lambda ; \\
  ( t, \R z)& \longmapsto & t(\R z)=\R t(z).
\end{array}$$
Then $\C^{\ast}$ acts transitively on $\Lambda$, and there exists a bijective map:
$$\C^{\ast}/\R^{\ast} \longrightarrow \Lambda; z \longmapsto \R z.$$
Note that $\C^{\ast}/\R^{\ast} \simeq \U(\C)/\mu_2$.
\subsection{Modular functions}
Let us consider the left  action of $\SL_2^{\pm}(\R)$ on $\mathcal{H}^{\pm}$ by $gz=\tfrac{az+b}{cz+d}$, for $g\in \SL_2^{\pm}(\R)$, $z\in \mathcal{H}^{\pm}$. This action is transitive and the stabilizer of $i$ is just the group $\SO_2(\R)$. Then there exists a bijective map:
$$i:\SL_2^{\pm}(\R)/\SO_2(\R) \longrightarrow\mathcal{H}^{\pm}; [g] \longmapsto gi.$$
 Note that  $\SL_2^{\pm}(\R)=P_{>0}^{\pm}(\R)\SO_2(\R)$,  $ P_{>0}^{\pm}(\R) \cap \SO_2(\R)=I$.
So there exists a bijective map:
\begin{align}\label{ph}
\overline{i}: P^{\pm}_{>0}(\R) \longrightarrow \mathcal{H}^{\pm} ;  g=\begin{pmatrix} a& b \\ 0 &(\det g)   a^{-1}\end{pmatrix} \longmapsto ab\det g + ia^2 \det g.
\end{align}
For $g=\begin{pmatrix} a   & b \\ c& d\end{pmatrix}\in \SL^{\pm}_2(\R)$, let us write $g=p_g k_{g}$, for $p_g\in P^{\pm}_{>0}(\R) $ and $k_g\in \SO_2(\R)$. Then:
\begin{lemma}\label{512lmm}
$p_g=\begin{pmatrix} \tfrac{1}{ \sqrt{c^2+d^2}}    & \det g \tfrac{bd+ac}{\sqrt{c^2+d^2}} \\ 0& \det g\sqrt{c^2+d^2}  \end{pmatrix}$, $k_{g}=\det g\begin{pmatrix} \tfrac{d}{\sqrt{c^2+d^2}}   & -\tfrac{c}{\sqrt{c^2+d^2}}\\ \tfrac{c}{\sqrt{c^2+d^2}}& \tfrac{d}{\sqrt{c^2+d^2}}\end{pmatrix}$.
\end{lemma}
\begin{proof}
\begin{align*}
\tfrac{1}{ \sqrt{c^2+d^2}} \tfrac{d\det g}{\sqrt{c^2+d^2}} +(\tfrac{bd+ac}{\sqrt{c^2+d^2}} )\tfrac{c}{\sqrt{c^2+d^2}}&= \tfrac{d\det g}{ \sqrt{c^2+d^2}}+ \tfrac{cbd+ac^2}{c^2+d^2}\\
&= \tfrac{dad-dbc+cbd+ac^2}{c^2+d^2}=a;
\end{align*}
\begin{align*}
(\tfrac{1}{ \sqrt{c^2+d^2}} )\tfrac{-c\det g }{\sqrt{c^2+d^2}} +(\tfrac{bd+ac}{\sqrt{c^2+d^2}} )\tfrac{d}{\sqrt{c^2+d^2}}&= \tfrac{-c\det g}{ \sqrt{c^2+d^2}}- \tfrac{bd^2+acd}{c^2+d^2}\\
&= \tfrac{-cad+bc^2+bd^2+acd}{c^2+d^2}=b.
\end{align*}
\end{proof}
Recall the functions $\widetilde{s}(k_g)$, $ \overline{s}(k_g)$ from Lemmas \ref{thetaf}, \ref{trivial}. Let us extend them  to the group  $\SL_2^{\pm}(\R)$ as follows:
\begin{align}
 \widetilde{s}(g)=\widetilde{s}(k_g) \quad \textrm{ and }  \quad \overline{s}(g)= \overline{s}(k_g).
\end{align}
\begin{lemma}
$ m_{X^{\ast}}(g)= m_{X^{\ast}}(p_g) m_{X^{\ast}}(k_g)$, for  $g=\begin{pmatrix} a & b \\ c& d \end{pmatrix} \in \SL_2^{\pm}(\R)$.
\end{lemma}
\begin{proof}
1) If $c=0$, $\det g=1$, then $m_{X^{\ast}}(g)=e^{\tfrac{i \pi}{4}[1- \sgn(d)]}$, $m_{X^{\ast}}(p_g) =1$, $m_{X^{\ast}}(k_g)=e^{\tfrac{i \pi}{4}[1- \sgn(\tfrac{d}{\sqrt{c^2+d^2}})]}$.\\
2) If $c=0$, $\det g=-1$, then $m_{X^{\ast}}(g)=e^{\tfrac{i \pi}{4}[1+ \sgn(d)]}$, $m_{X^{\ast}}(p_g) =1$, $m_{X^{\ast}}(k_g)=e^{\tfrac{i \pi}{4}[1+ \sgn(\tfrac{d}{\sqrt{c^2+d^2}})]}$.\\
3) If $c\neq 0$, $\det g=1$, then $m_{X^{\ast}}(g)=e^{\tfrac{i \pi}{4}[- \sgn(c)]}$, $m_{X^{\ast}}(p_g) =1$, $m_{X^{\ast}}(k_g)=e^{\tfrac{i \pi}{4}[- \tfrac{c}{\sqrt{c^2+d^2}}]}$.\\
4)  If $c\neq 0$, $\det g=-1$, then $m_{X^{\ast}}(g)=e^{\tfrac{i \pi}{4}[ \sgn(c)]}$, $m_{X^{\ast}}(p_g) =1$, $m_{X^{\ast}}(k_g)=e^{\tfrac{i \pi}{4}[ \tfrac{c}{\sqrt{c^2+d^2}}]}$.
\end{proof}

As a consequence, we have:
$$\overline{s}(g)=\overline{s}(p_g)\overline{s}(k_g), \quad \quad \widetilde{s}(g)=\widetilde{s}(p_g)\widetilde{s}(k_g),  \quad \textrm{ and } \quad \widetilde{s}(g)=\overline{s}(g)m_{X^{\ast}}(g).$$
 Let us modify the above $\widetilde{C}_{X^{\ast}}$ and $\overline{C}_{X^{\ast}}$ as follows:
\begin{align}
\widetilde{\widetilde{C}}_{X^{\ast}}(g_1,g_2)=\widetilde{C}_{X^{\ast}}(g_1,g_2)\widetilde{s}(g_1)\widetilde{s}(g_2) \widetilde{s}(g_1g_2)^{-1},
\end{align}
\begin{align}\label{ovov}
\overline{\overline{C}}_{X^{\ast}}(g_1,g_2)=\overline{C}_{X^{\ast}}(g_1,g_2)\overline{s}(g_1)\overline{s}(g_2) \overline{s}(g_1g_2)^{-1}.
\end{align}
\begin{lemma}\label{cx}
\begin{itemize}
\item[(1)] $\widetilde{\widetilde{C}}_{X^{\ast}}(g_1,g_2)=\overline{\overline{C}}_{X^{\ast}}(g_1,g_2)$, for $g_i\in \SL_2^{\pm}(\R)$. 
\item[(2)] $\overline{\overline{C}}_{X^{\ast}}(p,g)=1$, for $g\in \SL_2^{\pm}(\R)$, $p\in P^{\pm}_{>0}(\R) $.
\item[(3)] $\overline{\overline{C}}_{X^{\ast}}(g,k)=1$, for $g\in \SL_2^{\pm}(\R)$, $k\in \SO_2(\R)$.
\end{itemize}
\end{lemma}
\begin{proof}
1) \begin{align*}
\widetilde{\widetilde{C}}_{X^{\ast}}(g_1,g_2)&=\widetilde{C}_{X^{\ast}}(g_1,g_2)\widetilde{s}(g_1)\widetilde{s}(g_2) \widetilde{s}(g_1g_2)^{-1}\\
&=\overline{C}_{X^{\ast}}(g_1,g_2)m_{X^{\ast}}(g_1g_2)m_{X^{\ast}}(g_1)^{-1}m_{X^{\ast}}(g_2)^{-1}\widetilde{s}(g_1)\widetilde{s}(g_2) \widetilde{s}(g_1g_2)^{-1}\\
&=\overline{C}_{X^{\ast}}(g_1,g_2)\overline{s}(g_1)\overline{s}(g_2) \overline{s}(g_1g_2)^{-1}\\
&=\overline{\overline{C}}_{X^{\ast}}(g_1,g_2).
\end{align*}
2)
$$\overline{s}(pg)=\overline{s}(pp_gk_g)=\overline{s}(pp_g)\overline{s}(k_g)=\overline{s}(p)\overline{s}(p_g)\overline{s}(k_g)=\overline{s}(p)\overline{s}(g);$$
$$\overline{C}_{X^{\ast}}(p,g)=1;$$
\begin{align*}
\overline{\overline{C}}_{X^{\ast}}(p,g)&=\overline{C}_{X^{\ast}}(p,g)\overline{s}(p)\overline{s}(g) \overline{s}(pg)^{-1}=1.
\end{align*}
3) Let $k_1,k_2\in  \SO_2(\R)$.
\begin{align*}
\overline{\overline{C}}_{X^{\ast}}(p_g,k_g)&=\widetilde{C}_{X^{\ast}}(p_g,k_g)\widetilde{s}(p_g)\widetilde{s}(k_g)\widetilde{s}(g)^{-1}\\
&=\widetilde{s}(p_g)\widetilde{s}(k_g)\widetilde{s}(g)^{-1}=1;
\end{align*}
\begin{align*}
\overline{\overline{C}}_{X^{\ast}}(k_1,k_2)&=\widetilde{C}_{X^{\ast}}(k_1, k_2) \widetilde{s}(k_1) \widetilde{s}(k_2) \widetilde{s}(k_1k_2)^{-1}=1;
\end{align*}
\begin{align*}
\overline{\overline{C}}_{X^{\ast}}(g,k)=\overline{\overline{C}}_{X^{\ast}}(p_gk_g,k)&=\overline{\overline{C}}_{X^{\ast}}(p_g, k_g)^{-1}\overline{\overline{C}}_{X^{\ast}}(p_g,k_gk)\overline{\overline{C}}_{X^{\ast}}(k_g,k)=1.
\end{align*}
\end{proof}
For $z\in \mathcal{H}^{\pm}$, let us write $p_z$ for the corresponding element in $P_{>0}^{\pm}(\R)$ via the map $\overline{i}$.
\begin{definition}
For $g\in  \SL_2^{\pm}(\R), z\in \mathcal{H}^{\pm}$, we define:
$ \overline{j}(g, z)\stackrel{\Delta}{=} \overline{\overline{C}}_{X^{\ast}}(g, p_z)$.
\end{definition}
\begin{lemma}
$\overline{\overline{C}}_{X^{\ast}}(g_1, g_2)= \overline{j}(g_1, g_2z) \overline{j}(g_2, z) \overline{j}(g_1g_2, z)^{-1}$, for $g_1, g_2\in \SL_2^{\pm}(\R)$, $z\in \mathcal{H}^{\pm}$.
\end{lemma}
\begin{proof}
$\overline{\overline{C}}_{X^{\ast}}(g_1, g_2)\overline{\overline{C}}_{X^{\ast}}(g_1g_2,  p_z)= \overline{\overline{C}}_{X^{\ast}}(g_1, g_2 p_z) \overline{\overline{C}}_{X^{\ast}}(g_2,  p_z) $. Then:
\begin{align*}
\overline{\overline{C}}_{X^{\ast}}(g_1, g_2)&= \overline{\overline{C}}_{X^{\ast}}(g_1, g_2 p_z) \overline{\overline{C}}_{X^{\ast}}(g_2,  p_z)\overline{\overline{C}}_{X^{\ast}}(g_1g_2,  p_z)^{-1}\\
&=\overline{\overline{C}}_{X^{\ast}}(g_1, p_{g_2 p_z}k_{g_2 p_z}) \overline{\overline{C}}_{X^{\ast}}(g_2,  p_z)\overline{\overline{C}}_{X^{\ast}}(g_1g_2,  p_z)^{-1}\\
&=\overline{\overline{C}}_{X^{\ast}}(g_1, p_{g_2 p_z}k_{g_2 p_z}) \overline{\overline{C}}_{X^{\ast}}( p_{g_2 p_z},k_{g_2 p_z})\overline{\overline{C}}_{X^{\ast}}(g_2,  p_z)\overline{\overline{C}}_{X^{\ast}}(g_1g_2,  p_z)^{-1}\\
&=\overline{\overline{C}}_{X^{\ast}}(g_1, p_{g_2 p_z}) \overline{\overline{C}}_{X^{\ast}}( g_1p_{g_2 p_z},k_{g_2 p_z})\overline{\overline{C}}_{X^{\ast}}(g_2,  p_z)\overline{\overline{C}}_{X^{\ast}}(g_1g_2,  p_z)^{-1}\\
&=\overline{\overline{C}}_{X^{\ast}}(g_1, p_{g_2 p_z}) \overline{\overline{C}}_{X^{\ast}}(g_2,  p_z)\overline{\overline{C}}_{X^{\ast}}(g_1g_2,  p_z)^{-1}\\
&=\widetilde{j}(g_1, g_2z) \widetilde{j}(g_2, z) \widetilde{j}(g_1g_2, z)^{-1}.
\end{align*}
\end{proof}
\subsection{Holomorphic function} Form Lemma \ref{512lmm}, we can see that $p_g$ is a $C^{\infty}$-function from $\SL_2^{\pm}(\R)$ to $P_{>0}^{\pm}(\R)$. Moreover, $\overline{i}$ is a  $C^{\infty}$-homeomorphism.
Let $$\mathfrak{sl}_2(\R)=\{ \begin{bmatrix} a & b\\ c& d\end{bmatrix}\in M_2(\R)\mid a+d=0\}$$ denote the real Lie algebra of $\SL_2(\R)$.  It is known that there exist the left and right regular $ \SL^{\pm}_2(\R)$-actions on $C^{\infty}(\SL_2^{\pm}(\R), \R)$. These actions can extend to the lie algebra in the following way:
$$[l_{Z}f](g)=\tfrac{d}{dt}[f(e^{-tZ}g)]|_{t=0}, \quad\quad  Z\in \mathfrak{sl}_2(\R),$$
$$[r_{Z}f](g)=\tfrac{d}{dt}[f(ge^{tZ})]|_{t=0}, \quad\quad  Z\in \mathfrak{sl}_2(\R).$$
If $f$ is a real-$C^{\infty}$ complex  function, we can extend this action  to the complex Lie algebra by linearity.  Following \cite[pp.181-182]{LiVe}, let
\begin{equation}\label{hefj}
h=\begin{bmatrix} 1 & 0\\ 0& -1\end{bmatrix}, e^+=\begin{bmatrix} 0 & 1\\ 0& 0\end{bmatrix}, e^-=\begin{bmatrix} 0 & 0\\ 1& 0\end{bmatrix}, J^+=\tfrac{1}{2}\begin{bmatrix} -i & 1\\ 1& i\end{bmatrix}, J^-=\tfrac{1}{2}\begin{bmatrix} i & 1\\ 1&-i\end{bmatrix}, J^0=\begin{bmatrix} 0& 1\\-1& 0\end{bmatrix}.
 \end{equation}
 The following two results come from \cite[p.182, Section 2.3.8,Lmm]{LiVe} for $\SL_2(\R)$.  We repeat the proofs for $\SL_2^{\pm}(\R)$ without too much change.
\begin{lemma}
 For a  real-$C^{\infty}$ complex function $f$ on $\SL_2^{\pm}(\R)$, $f$ is right $\SO_2(\R)$-invariant iff $r_{J^0}f=0$.
\end{lemma}
\begin{proof}
Let us write $f=f_1+if_2$.  \\
$(\Rightarrow) $ For $k=\begin{pmatrix} \cos \theta & \sin \theta\\ -\sin \theta & \cos \theta\end{pmatrix}\in \SO_2(\R)$, $f(gk)=f(g)$. Write
$$J^0=\begin{pmatrix} 0& 1\\ -1 & 0\end{pmatrix}=c^{-1}\begin{pmatrix} -i& 0\\ 0 & i\end{pmatrix}c, \qquad c=\tfrac{1}{\sqrt{2}}\begin{bmatrix} 1& i\\ i& 1\end{bmatrix}.$$ Then $$e^{t\begin{pmatrix} 0& 1\\ -1 & 0\end{pmatrix}}=c^{-1}\begin{bmatrix} e^{-it}& 0\\ 0& e^{it}\end{bmatrix}c=\begin{pmatrix} \cos t & \sin t\\ -\sin t & \cos t\end{pmatrix}.$$ Then $[f_i(ge^{t\begin{pmatrix} 0& -1\\ 1 & 0\end{pmatrix}})]=f_i(g)$. Hence $r_{J^0}f=0$.\\
$(\Leftarrow)$ Let  us  write $g_{t}=\begin{pmatrix} \cos t & \sin t\\ -\sin t & \cos t\end{pmatrix}$. If  $r_{J^0}f=0$, then $r_{J^0} l_{g_{t}} f= l_{g_{t}} r_{J^0}f=0$.  So $f_i(gg_t)=f_i(g)$, for any $t$. So $f$ is right $\SO_2(\R)$-invariant.
\end{proof}
For any function $f$ on $\mathcal{H}^{\pm}$, let us define
$$[I(f)](g)=f(gi).$$
\begin{lemma}\label{holo}
 For a complex valued real-$C^{\infty}$-function $f$ on $\mathcal{H}^{\pm}$, $f$ is holomorphic iff $r_{J^-}[I(f)]=0$.
\end{lemma}
\begin{proof}
Let us write:
 $$J^-=\begin{bmatrix} 0 & \tfrac{1}{2}\\\tfrac{1}{2}&0\end{bmatrix}+ i\begin{bmatrix}\tfrac{1}{2} & 0\\ 0&-\tfrac{1}{2}\end{bmatrix}, \quad \quad f(x+iy)=u(x,y)+iv(x,y).$$ Then:
\begin{align*}
e^{\begin{bmatrix} 0 & \tfrac{t}{2}\\ \tfrac{t}{2}&0\end{bmatrix}}&=\begin{bmatrix} 1& 0\\ 0&1\end{bmatrix}+\begin{bmatrix} 0 & \tfrac{t}{2}\\ \tfrac{t}{2}&0\end{bmatrix} +o(t^2), \\
 e^{\begin{bmatrix} t/2 & 0\\ 0&-t/2\end{bmatrix}}&=\begin{bmatrix} 1& 0\\ 0&1\end{bmatrix}+ \begin{bmatrix} t/2 & 0\\ 0&-t/2\end{bmatrix} +o(t^2);
\end{align*}
\begin{align*}
r_{J^-}[I(f)](\begin{bmatrix} 1 & 0\\ 0&1\end{bmatrix})&=\tfrac{d}{dt}f(e^{\begin{bmatrix} 0 & \tfrac{t}{2}\\ \tfrac{t}{2}&0\end{bmatrix}}i)|_{t=0}+ i\tfrac{d}{dt}f(e^{\begin{bmatrix} t/2 & 0\\ 0&-t/2\end{bmatrix}}i)|_{t=0}\\
&=\tfrac{d}{dt}f(\begin{bmatrix} 1 & \tfrac{t}{2}\\ \tfrac{t}{2}&1\end{bmatrix}i)|_{t=0}+ i\tfrac{d}{dt}f(\begin{bmatrix}1+ t/2 & 0\\ 0&1-t/2\end{bmatrix} i)|_{t=0}\\
&=\tfrac{d}{dt}[u(\tfrac{t}{1-t^2/4},1)+iv(\tfrac{t}{1-t^2/4},1)]|_{t=0}+ i\tfrac{d}{dt}[u(0, \tfrac{1+t/2}{1-t/2})+iv(0, \tfrac{1+t/2}{1-t/2})]|_{t=0}\\
&=\tfrac{\partial u}{\partial x }(0,1)+i\tfrac{\partial v}{\partial x }(0,1)+ i\tfrac{\partial u}{\partial y }(0, 1)-i\tfrac{\partial v}{\partial y }(0, 1)\\
&=(\tfrac{\partial }{\partial x }+i\tfrac{\partial }{\partial y })f(i)\\
&=2\tfrac{\partial }{\partial \overline{z} }f(i);
\end{align*}
\begin{align*}
r_{J^-}[I(f)](g)&=r_{J^-}[g^{-1}I(f)](\begin{bmatrix} 1 & 0\\ 0&1\end{bmatrix})\\
&=g^{-1}r_{J^-}[I(f)](\begin{bmatrix} 1 & 0\\ 0&1\end{bmatrix})\\
&=2\tfrac{\partial }{\partial \overline{z} }f(gi).
\end{align*}
\end{proof}

\subsection{  $\widehat{\SL^{\pm}_2}(\R)$}
Let $\widehat{\SL^{\pm}_2}(\R)$ denote the covering group over $\SL_2^{\pm}(\R)$ associated to the above $2$-cocycle $\overline{\overline{C}}_{X^{\ast}}$. Then there exists an exact sequence:
$$1 \longrightarrow T \longrightarrow \widehat{\SL^{\pm}_2}(\R)\longrightarrow \SL_2^{\pm}(\R) \to 1.$$
Let $g_i=\begin{pmatrix} a_i & b_i\\ c_i&d_i\end{pmatrix} \in \SL_2^{\pm}(\R)$. Then:
\begin{equation}\label{overlinecccc}
\begin{split}
\overline{\overline{C}}_{X^{\ast}}(p_{g_1}k_{g_1},p_{g_2}k_{g_2})&=\overline{\overline{C}}_{X^{\ast}}(p_{g_1}k_{g_1},p_{g_2}k_{g_2})
\overline{\overline{C}}_{X^{\ast}}(p_{g_2},k_{g_2})\\
&=\overline{\overline{C}}_{X^{\ast}}(p_{g_1}k_{g_1},p_{g_2})\\
&=\overline{\overline{C}}_{X^{\ast}}(p_{g_1},k_{g_1})\overline{\overline{C}}_{X^{\ast}}(p_{g_1}k_{g_1},p_{g_2})\\
&=\overline{\overline{C}}_{X^{\ast}}(p_{g_1},k_{g_1}p_{g_2})\overline{\overline{C}}_{X^{\ast}}(k_{g_1},p_{g_2})\\
&=\overline{\overline{C}}_{X^{\ast}}(k_{g_1},p_{g_2})\\
&=\overline{C}_{X^{\ast}}(k_{g_1},p_{g_2})\overline{s}(k_{g_1})\overline{s}(p_{g_2}) \overline{s}(k_{g_1}p_{g_2})^{-1}.
\end{split}
\end{equation}
By Lemma \ref{512lmm}, the maps $g\to k_g$,$g\to  p_{g}$  both are  real $C^{\infty}$.\\
If $c\neq 0$,  $$\overline{\overline{C}}_{X^{\ast}}(g_1, g_2)=\overline{s}(k_{g_1})\overline{s}(p_{g_2})\overline{s}(k_{g_1}p_{g_2})^{-1}.$$
In this case, $k_{g_1}p_{g_2}=\begin{pmatrix}\ast & \ast\\ c'&\ast \end{pmatrix} $ with $c'\neq0$. By Lemma \ref{trivial}, we know that $\overline{\overline{C}}_{X^{\ast}}(g_1, g_2)$ is $C^{\infty}$ on  $g_1$ and $g_2$.\\
If $d\neq 0$,
\begin{align*}
\overline{\overline{C}}_{X^{\ast}}(g_1, g_2)&=\overline{\overline{C}}_{X^{\ast}}(\omega,  \omega^{-1}g_1)^{-1}\overline{\overline{C}}_{X^{\ast}}(\omega, \omega^{-1} g_1 g_2)\overline{\overline{C}}_{X^{\ast}}(\omega^{-1} g_1, g_2).
\end{align*}
So it is also   $C^{\infty}$ on $g_1$ and $g_2$. Hence $\widehat{\SL^{\pm}_2}(\R)$ is a smooth manifold, and $\widehat{\SL^{\pm}_2}(\R) \longrightarrow \SL^{\pm}_2(\R)$ is a $C^{\infty}$-map.
\subsection{Extended Weil representation of $\widehat{\SL^{\pm}_2}(\R)$}
The group $\widehat{\SL_2^{\pm}}(\R)$ is generated by
$$[ u(b)= \begin{pmatrix}
  1&b\\
  0 & 1
\end{pmatrix}, 1], [h(a)= \begin{pmatrix}
  a& 0\\
  0 &a^{-1 }
\end{pmatrix},1], [h_{-1}= \begin{pmatrix}
  1& 0\\
  0 &-1
\end{pmatrix},1],  [\omega= \begin{pmatrix}
  0& -1\\
  1 &0
\end{pmatrix},1],  [1,t],$$
for  $t\in T$, $b\in F$, $a\in F^{\times}$.  By (\ref{292inter})(\ref{ovov}), we have:
 \begin{align*}
 \overline{\overline{C}}_{X^{\ast}}(g_1,g_2)&=\widetilde{C}_{X^{\ast}}(g_1, g_2)\widetilde{s}(g_1)\widetilde{s}(g_2) \widetilde{s}(g_1g_2)^{-1}.
  \end{align*}
1) If $g\in \in P^{\pm}_{>0}(\R)$, $\widetilde{s}(g)= 1$.\\
2)  If $g=h(a)$ with $a<0$, $\widetilde{s}(g)=m_{X^{\ast}}(g)\overline{s}(g)= i\overline{s}(g)=-1$.\\
3) If $g=\omega$, $\widetilde{s}(g)=-i$.\\
4) If $g=\omega^{-1}$, $\widetilde{s}(g)=i$.\\
5) If $g=\omega^{-1}u(b)\omega=\begin{pmatrix}
  1& 0\\
  -b &1
\end{pmatrix}  $, then  $\widetilde{s}(g)=m_{X^{\ast}}(g)\overline{s}(g)= e^{\tfrac{\pi i (\sgn b)}{4}} \sqrt{\tfrac{b}{\sqrt{1+b^2}}i+\tfrac{1}{\sqrt{1+b^2}}}$.

Go back to $(\ref{Cdou})$, and modify the action of $\Pi_{\psi} (g)$ on  $L^2(\mu_2 \times \R)$ by $\widetilde{s}(g)$. Let us write
\begin{align}\label{ovov}
\overline{\overline{\Pi}}_{\psi} (g)=\Pi_{\psi} (g)\widetilde{s}(g).
   \end{align}
Then:
 \begin{align*}
 \overline{\overline{\Pi}}_{\psi} (g_1) \overline{\overline{\Pi}}_{\psi} (g_2)&=\Pi_{\psi} (g_1) \widetilde{s}(g_1)\Pi_{\psi} (g_2) \widetilde{s}(g_2)\\
 &=\widetilde{C}_{X^{\ast}}(g_1,g_2) \Pi_{\psi} (g_1g_2)\widetilde{s}(g_1) \widetilde{s}(g_2)\\
 &=\overline{\overline{\Pi}}_{\psi} (g_1g_2)\widetilde{C}_{X^{\ast}}(g_1,g_2) \widetilde{s}(g_1) \widetilde{s}(g_2)\widetilde{s}(g_1g_2)^{-1} \\
 &=\overline{\overline{\Pi}}_{\psi} (g_1g_2)\overline{\overline{C}}_{X^{\ast}}(g_1,g_2).
     \end{align*}
By (\ref{representationsp22})(\ref{representationsp23})(\ref{representationsp24})(\ref{representationsp25}), the representation $\overline{\overline{\Pi}}_{\psi} $ of $\widehat{\SL^{\pm}_2}(\R)$ can be realized on $L^2(\mu_2 \times \R)$ by the following formulas:
\begin{equation}\label{ac1}
 \overline{\overline{\Pi}}_{\psi} [
 u(b)]f([\epsilon, x])=e^{\pi i \epsilon x^2 b} f([\epsilon,x]);
\end{equation}
\begin{equation}\label{ac2}
\overline{\overline{\Pi}}_{\psi}[ h(a)]f([\epsilon, x])=\sgn(a)|a|^{1/2} ( a,\epsilon)_Ff([\epsilon,xa]);
\end{equation}
\begin{equation}\label{ac3}
\overline{\overline{\Pi}}_{\psi}[\omega]f([\epsilon, x])=-i\nu(\epsilon,\omega)\int_{\R} e^{2\pi i \epsilon xy} f([\epsilon,-y]) dy;
\end{equation}
\begin{equation}\label{ac4}
\overline{\overline{\Pi}}_{\psi}[h_{-1}, t]f([\epsilon,x])=tf([-\epsilon, x]).
\end{equation}
\subsection{ Modular forms on $\widehat{\SL^{\pm}_2}(\R)$ }
Recall that a character of $\SO_2(\R)$ has the form:
$$\chi_{n}: \SO_2(\R) \longrightarrow \C^{\times};   \begin{pmatrix} \cos t& \sin t\\ -\sin t & \cos t\end{pmatrix} \longmapsto e^{-i
 nt}=[i\sin (-t)+\cos(- t)]^n.$$
Note that there exists a group embedding map:
$$\SO_2(\R) \longrightarrow  \widehat{\SL^{\pm}_2}(\R); k \longmapsto [k,1].$$
Let $\Gamma$ denote a discrete group of $\SL^{\pm}_2(\R)$, and let $\widehat{\Gamma}$ be inverse image of it in $\widehat{\SL_2^{\pm}}(\R)$.  Let $\delta_{\Gamma}$ denote a character of $\Gamma$. Let $\widehat{\delta_{\Gamma}}$ denote its extending representation on $\widehat{\Gamma}$.
By \cite[pp.184-188]{LiVe}, let $\mathcal{M}( \widehat{\SL^{\pm}_2}(\R),  J^-, \chi_n, \widehat{\delta_{\Gamma}})$ denote the space of  functions $f: \widehat{\SL^{\pm}_2}(\R) \longrightarrow \C$ such that
\begin{itemize}
\item[(1)] $f$ is analytic on $\widehat{\SL^{\pm}_2}(\R)$,
\item[(2)] $f(g k(t))=\chi_n(t) f(g )$,
\item[(3)] $r_{J^-}(f)=0$,
\item[(4)] $f(\gamma g)=\widehat{\delta_{\Gamma}}(\gamma)f(g)$,
\end{itemize}
 for $g \in \widehat{\SL^{\pm}_2}(\R)$, and $k(t)= \begin{pmatrix} \cos t& \sin t\\ -\sin t & \cos t\end{pmatrix}$,  $\gamma\in \widehat{\Gamma}$.
  Let us extend  $\chi_n$ to  a function on $\widehat{\SL^{\pm}_2}(\R)$ from Lemma \ref{512lmm}. For $g=\begin{pmatrix} a & b \\ c& d\end{pmatrix} \in \SL_2^{\pm}(\R)$,  define
$$\chi_{n}([g, \epsilon])=\epsilon(ci+d)^{n}.$$
If write $g=p_gk_g$, then $$\chi_{n}([p_g, \epsilon])=\epsilon(\det g \sqrt{c^2+d^2})^n, \quad \quad \chi_n([k_g, 1])= (\frac{\det g}{\sqrt{c^2+d^2}})^n(ci+d)^{n},$$
$$\chi_{n}([g, \epsilon])=\chi_{n}([p_g,\epsilon][k_g,1 ])=\chi_{n}([p_g, \epsilon]) \chi_n([k_g, 1]).$$
\begin{lemma}
For $[g, \epsilon]\in \widehat{\SL^{\pm}_2}(\R)$, $k(t)= \begin{pmatrix} \cos t& \sin t\\ -\sin t & \cos t\end{pmatrix} \in \SO_2(\R)$, we have:
$$\chi_{n}([g, \epsilon] [k(t),1])= e^{- i nt}\chi_{n}([g, \epsilon]).$$
\end{lemma}
\begin{proof}
\begin{align*}
[g, \epsilon] [k(t),1]&=[gk(t), \overline{\overline{C}}_{X^{\ast}}(g,k(t)) \epsilon]\\
&=[gk(t),  \epsilon]\\
&=[p_g, \epsilon][k_g, 1][k(t), 1].
\end{align*}
 \begin{align*}
\chi_{n}([g, \epsilon] [k(t),1])&=\chi_{n}([gk(t), \epsilon])\\
&=\chi_{n}([p_g, \epsilon][k_gk(t), 1])\\
&=\chi_{n}([p_g, \epsilon])\chi_{n}([k_g, 1])\chi_{n}([k(t), 1])\\
&=\chi_{n}([g, \epsilon])  \chi_{n}([k(t), 1]).
 \end{align*}
\end{proof}
Note that $\chi_n$ is a $C^{\infty}$-function on $\widehat{\SL^{\pm}_2}(\R)$.  Similarly as \cite[p.186, 2.3.14, Lemma]{LiVe}, we also have:
\begin{lemma}
$r_{J^-}(\chi_n)=0$.
\end{lemma}
\begin{proof}
We follow the proof of Lemma 2.3.14 in \cite{LiVe}. Recall the notations $h,e^+,e^-,J^0,J^+,J^-$ from (\ref{hefj}). Let $I=\begin{bmatrix} 1 & 0\\ 0& 1\end{bmatrix}$. Then:
$$r_{e}(\chi_n)(I)= \frac{d}{dt} (\chi_n([e^{\begin{bmatrix} 0 & t\\ 0& 0\end{bmatrix}}, 1]))|_{t=0}=\frac{d}{dt} (\chi_n(\begin{bmatrix} 1 & t\\ 0& 1\end{bmatrix}))|_{t=0}=0;$$
$$r_{h}(\chi_n)(I)= \frac{d}{dt} (\chi_n([e^{\begin{bmatrix} t & 0\\ 0& -t\end{bmatrix}}, 1]))|_{t=0}=\frac{d}{dt} (e^{-nt})|_{t=0}=-n;$$
$$r_{J^0}(\chi_n)(I)=\frac{d}{dt}\chi_n([ \begin{pmatrix} \cos t & \sin t\\ -\sin t & \cos t\end{pmatrix},1])|_{t=0}=-in.$$
Note that $J^-=\frac{1}{2}(ih+2e^+-J^0)$. Hence $r_{J^-}(\chi_n)(I)=-in+in=0$.  For the general $[g, \epsilon]\in \widehat{\SL^{\pm}_2}(\R)$, we have
$$\chi_n([g, \epsilon])=\chi_{n}([p_g, \epsilon]) \chi_n([k_g, 1])=\epsilon(\det g \sqrt{c^2+d^2})^n\chi_n([k_g, 1])$$
\end{proof}
Let $f \in \mathcal{M}( \widehat{\SL^{\pm}_2}(\R), J^-, \chi_n, \widehat{\delta_{\Gamma}})$. We can  define a function  on $\widehat{\SL^{\pm}_2}(\R)$  twisted by $\chi_{n}$ as follows:
$$f_{n}=f\chi_{n}^{-1}.$$
Then $f_n$ is $\SO_2(\R)$-invariant and it satisfies the above conditions (1)(2)(3)(4). Note that  $\widehat{\SL_2^{\pm}}(\R)=\widehat{P_{>0}^{\pm}}(\R)\SO_2(\R)$,  $ \widehat{P_{>0}^{\pm}}(\R) \cap \SO_2(\R)=I$.
So there exists a bijective map: $$\overline{i}: \widehat{P^{\pm}_{>0}}(\R) \longrightarrow \mathcal{H}^{\pm}\times T;  [g, t]=[\begin{pmatrix} a& b \\ 0 &(\det g)   a^{-1}\end{pmatrix}, t] \longmapsto (ab\det g + ia^2 \det g, t).$$
For the above $f_n$, through $\overline{i}$, it defines a function on $\mathcal{H}^{\pm}\times T$.  By Lemma \ref{holo}, it is holomorphic on $\mathcal{H}^{\pm}$.
\subsection{Multiplier factors}
For $g=\begin{pmatrix} a&b\\c&d\end{pmatrix}\in\SL^{\pm}_2(\R)$, and $z\in \mathcal{H}^{\pm}$, let us  consider the analytic function on $\mathcal{H}^{\pm}$:
$$J(g, -): z \longrightarrow \sqrt{\det(g)(cz+d)} .$$
  Note that in our choice,  $\sqrt{|t|}=|t|^{\tfrac{1}{2}}$, and $\sqrt{e^{i\theta}}=e^{\tfrac{i\theta}{2}}$, for $-\pi< \theta<\pi$, and if $c=0$, $d=-1$, $\sqrt{-1}=-i$ in our choice. If $g\in \SO_2(\R)$, then $J(g, i)=\overline{s}(g)^{-1}$. 
  \begin{lemma}
  $J(g,z)=J(p_g, z)J(k_g,z)$.
  \end{lemma}
  \begin{proof}
  By Lemma \ref{512lmm}, $J(p_g, z)=\sqrt[4]{c^2+d^2}$, $J(k_g,z)=\sqrt{(\det g)(\tfrac{c}{\sqrt{c^2+d^2}} z+\tfrac{d}{\sqrt{c^2+d^2}})}$.
  \end{proof}
\begin{lemma}
  $J(gp_z,i)=J(g,z)J(p_z,i)$.
  \end{lemma}
  \begin{proof}
  Write $p_z=\begin{pmatrix} y&x \\ 0 &(\det p_z)   y^{-1}\end{pmatrix}$. Then: 
  $$z=yx\det p_z + iy^2 \det p_z;$$
   $$J(p_z,i)=\sqrt{y^{-1}};$$
   $$ J(g,z)=\sqrt{\det g(cz+d)};$$
  $$gp_z=\begin{pmatrix} ay&ax+by^{-1}\det p_z \\ cy &cx+dy^{-1}\det p_z\end{pmatrix};$$
  \begin{align*}
  J(gp_z,i)&=\sqrt{(cyi+cx+dy^{-1}\det p_z) \det (gp_z)}\\
  &=\sqrt{y^{-1}\det g(cz+d)}\\
  &=J(g,z)J(p_z,i).
  \end{align*}
  \end{proof}
 
\subsection{Modular forms from an extended Weil representation}
 Let $\overline{\SL_2}^{\pm}(\R)$ denote the two-degree covering group over $\SL_2^{\pm}(\R)$ associated to the above $2$-cocycle $\overline{C}_{X^{\ast}}$.  Go back to $(\ref{Cdou})$, and modify the action of $\Pi_{\psi} (g)$ on  $L^2(\mu_2 \times \R)$ by $\overline{s}(g)$. Let us write
\begin{align}\label{ovov}
\overline{\Pi}_{\psi} (g)=\Pi_{\psi} (g)m_{X^{\ast}}(g).
   \end{align}
Then:
 \begin{align*}
 \overline{\Pi}_{\psi} (g_1) \overline{\Pi}_{\psi} (g_2)&=\overline{\Pi}_{\psi} (g_1g_2)\overline{C}_{X^{\ast}}(g_1,g_2).
     \end{align*}
   Note that $\overline{\Pi}_{\psi}=\overline{\overline{\Pi}}_{\psi} \overline{s}(g)^{-1}$. By (\ref{ac1})-(\ref{ac4}),  the representation $\overline{\Pi}_{\psi} $ of $\overline{\SL^{\pm}_2}(\R)$ can be realized on $L^2(\mu_2 \times \R)$ by the following formulas:
\begin{equation}\label{ac11}
 \overline{\Pi}_{\psi} [
 u(b)]f([\epsilon, x])=e^{\pi i \epsilon x^2 b} f([\epsilon,x]);
\end{equation}
\begin{equation}\label{ac21}
\overline{\Pi}_{\psi}[ h(a)]f([\epsilon, x])=e^{\tfrac{\pi i}{4}[1-\sgn(a)]}|a|^{1/2} ( a,\epsilon)_Ff([\epsilon,xa]);
\end{equation}
\begin{equation}\label{ac31}
\overline{\Pi}_{\psi}[\omega]f([\epsilon, x])=e^{-\tfrac{\pi i}{4}}\nu(\epsilon,\omega)\int_{\R} e^{2\pi i \epsilon xy} f([\epsilon,-y]) dy;
\end{equation}
\begin{equation}\label{ac41}
\overline{\Pi}_{\psi}[h_{-1}, t]f([\epsilon,x])=tf([-\epsilon, x]).
\end{equation}
Recall the notations $h,e^+,e^-,J^0,J^+,J^-$ from (\ref{hefj}). Write  $f([1, x])=f_1(x)$ and $f([-1, x])=f_2(x)$. Assume $f_i\in S(\R)$. Then:\\
1) Note that $e^{th}=\begin{pmatrix}
  e^t&0\\
  0&e^{-t}
\end{pmatrix}$ and $\overline{C}_{X^{\ast}}(e^{t_1h}, e^{t_2h})=1$. So let us consider the one parameter group homomorphism:
$$\exp: th \longrightarrow [\begin{pmatrix}
  e^t&0\\
  0&e^{-t}
\end{pmatrix}, 1].$$
\begin{align*}
 d\overline{\Pi}_{\psi} (h)f([\epsilon, x])&=\frac{d}{dt}\Big\{ \overline{\Pi}_{\psi}([e^{th}, 1]) f([\epsilon, x])\Big\}\Big|_{t=0} \\
 &= \frac{d}{dt}\Big\{e^{\tfrac{t}{2}} f([\epsilon, xe^t])\Big\}\Big|_{t=0}\\
&= \{ \tfrac{1}{2}+x \tfrac{d }{dx}\}f([\epsilon, x]).
 \end{align*}
2) Note that $e^{te^+}=\begin{pmatrix}
  1&t\\
  0&1
\end{pmatrix}$ and $\overline{C}_{X^{\ast}}(e^{t_1e^+}, e^{t_2e^+})=1$. So let us consider the one parameter group homomorphism:
$$\exp: te^+ \longrightarrow [\begin{pmatrix}
  1&t\\
  0&1
\end{pmatrix}, 1].$$
 \begin{align*}
 d\overline{\Pi}_{\psi} (e^+)f([\epsilon, x])&=\frac{d}{dt}\Big\{ \Pi_{\psi}([e^{te^+}, 1]) f([\epsilon, x])\Big\}\Big|_{t=0} \\
 &=\frac{d}{dt}\Big\{ e^{\pi i \epsilon x^2 t} f([\epsilon, x]) \Big\}\Big|_{t=0} \\
&=\pi i\epsilon x^2 f([\epsilon, x]).
 \end{align*}
3)  Note that $\exp(te^-)=\begin{pmatrix}
  1& 0\\
  t&1
\end{pmatrix}=\omega^{-1}\begin{pmatrix}
  1&-t\\
  0&1
\end{pmatrix} \omega$.  Moreover, we have:
\begin{align*}
\overline{C}_{X^{\ast}}(\omega^{-1}, u(-t))&=\overline{c}_{X^{\ast}}(\omega^{-1}, u(-t))\\
 &=(x(\omega^{-1}), x(u(-t)))_\R(-x(\omega^{-1})x(u(-t)), x(\omega^{-1}u(-t)))_\R\\
 &=1;
 \end{align*}
 \begin{align*}
\overline{C}_{X^{\ast}}(\omega^{-1}u(-t), \omega)&=\overline{c}_{X^{\ast}}(\omega^{-1}u(-t), \omega)\\
 &=(x(\omega^{-1}u(-t)), x(\omega))_\R(-x(\omega^{-1}u(-t))x(\omega), x(u_-(t))_\R\\
 &=1;
 \end{align*}
 $$ \overline{C}_{X^{\ast}}(\omega^{-1} , \omega)=1;$$
 \begin{align}
\overline{C}_{X^{\ast}}(u_-(t_1), u_-(t_2))&=\overline{c}_{X^{\ast}}(u_-(t_1), u_-(t_2))\\
 &=(x(u_-(t_1)), x(u_-(t_2))_\R(-x(u_-(t_1))x(u_-(t_2)), x(u_-(t_1+t_2))_\R\label{qqq};
 \end{align}
 \begin{itemize}
\item If $t_1t_2=0$, $(\ref{qqq})=1$.
\item  If $t_1t_2>0$, assume $t_1>0$ and $t_2<0$; then  $(x(u_-(t_1)), x(u_-(t_2))_\R=1$, $(-x(u_-(t_1))x(u_-(t_2)), x(u_-(t_1+t_2))_\R=(-t_1t_2,  x(u_-(t_1+t_2))_\R=1$. So $(\ref{qqq})=1$.
\item  If $t_1>0$, $t_2>0$, $(x(u_-(t_1)), x(u_-(t_2))_\R=(t_1, t_2)_\R=1$, $(-x(u_-(t_1))x(u_-(t_2)), x(u_-(t_1+t_2))_\R=(-t_1t_2, t_1+t_2)_\R=1$, so $(\ref{qqq})=1$.
 \item If $t_1<0$, $t_2<0$, $(x(u_-(t_1)), x(u_-(t_2))_\R=(t_1, t_2)_\R=-1$, $(-x(u_-(t_1))x(u_-(t_2)), x(u_-(t_1+t_2))_\R=(-t_1t_2, t_1+t_2)_\R=-1$. Therefore,  $(\ref{qqq})=1$.
 \end{itemize}
Let us consider the one parameter group homomorphism:
$$\exp: te^- \longrightarrow [\omega^{-1}, 1][\begin{pmatrix}
  1&-t\\
  0&1
\end{pmatrix}, 1][\omega, 1]=[e^{te^-}, 1].$$
a)\begin{align*}
\overline{\Pi}_{\psi} (\exp(te^-)) f_1(x)&= \overline{\Pi}_{\psi} ([-\omega, 1])\{\overline{\Pi}_{\psi}([
 u(-t), 1])\overline{\Pi}_{\psi} ([\omega, 1])f_1\}(x)\\
 &=ie^{-\tfrac{\pi i}{4}}\nu(1, \omega)\int_{\R} e^{2\pi i xy} \{\overline{\Pi}_{\psi}([
 u(-t), 1])\overline{\Pi}_{\psi} ([\omega, 1])f_1\}( y) dy\\
 &=ie^{-\tfrac{\pi i}{4}}\int_{\R} e^{2\pi i xy} e^{-\pi i y^2t}\{\overline{\Pi}_{\psi} ([\omega, 1])f_1\}( y) dy\\
 &=ie^{-\tfrac{\pi i}{4}}[e^{-\tfrac{\pi i}{4}}]\int_{\R} e^{2\pi i xy} e^{-\pi i y^2t}\widehat{f_1}(y) dy\\
 &=\int_{\R} e^{-\pi i t y^2} \widehat{f_1}(y) e^{2\pi i xy} dy. 
 \end{align*}
 \begin{align*}
 d\overline{\Pi}_{\psi} (e^-)f([1, x])&= d\overline{\Pi}_{\psi} (e^-)f_1( x)\\
&=\frac{d}{dt}\Big\{  \overline{\Pi}_{\psi} (\exp(te^-)) f_1( x)\Big\}\Big|_{t=0} \\
&=\int_{\R} (-\pi i y^2)\widehat{f_1}(y) e^{2\pi i xy} dy\\
&=-\frac{1}{4\pi i}\frac{d^2}{dx^2}\int_{\R}\widehat{f_1}(y) e^{2\pi i xy} dy\\.
&=-\frac{1}{4\pi i}\frac{d^2}{dx^2} f_1(x).
 \end{align*}

b) \begin{align*}
 \overline{\Pi}_{\psi} ([\omega, 1]) f_2(x)&= \overline{\Pi}_{\psi} ([\omega, 1]) f([-1,x])\\
&=e^{-\tfrac{\pi i}{4}}\nu(-1, \omega) \int_{\R} e^{-2\pi i xy} f([-1, -y]) dy\\
&=e^{-\tfrac{\pi i}{4}}\nu(-1, \omega)\int_{\R} e^{-2\pi i xy} f_2( -y) dy\\
&=e^{-\tfrac{\pi i}{4}}\nu(-1, \omega)\widehat{f_2}(-x).
 \end{align*}
 $$[-\omega, 1]=[\omega, 1][h(-1), 1]$$
 \begin{align*}
 \overline{\Pi}_{\psi} ([-\omega, 1]) f_2(x)&= \overline{\Pi}_{\psi} ([-\omega, 1]) f([-1,x])\\
&=e^{-\tfrac{\pi i}{4}}\nu(-1, \omega) \int_{\R} e^{-2\pi i xy} [\overline{\Pi}_{\psi} ([h(-1), 1])f]([-1, -y]) dy\\
&=-ie^{-\tfrac{\pi i}{4}}\nu(-1, \omega)  \int_{\R} e^{-2\pi i xy}f([-1, y]) dy\\
&=-ie^{-\tfrac{\pi i}{4}}\nu(-1, \omega)\int_{\R} e^{-2\pi i xy} f_2( y) dy\\
&=-ie^{-\tfrac{\pi i}{4}}\nu(-1, \omega)\widehat{f_2}(x).
 \end{align*}
\begin{align*}
\overline{\Pi}_{\psi}([
 u(-t), 1])f([-1, x])=e^{\pi i  x^2 t} f([-1,x])
\end{align*}
\begin{align*}
\overline{\Pi}_{\psi} (\exp(te^-)) f_2( x)&= \overline{\Pi}_{\psi} ([-\omega, 1])\{\overline{\Pi}_{\psi}([
 u(-t), 1])\overline{\Pi}_{\psi} ([\omega, 1])f_2\}(x)\\
 &=-ie^{-\tfrac{\pi i}{4}}\nu(-1, \omega)\int_{\R} e^{-2\pi i xy} \{\overline{\Pi}_{\psi}([
 u(-t), 1])\overline{\Pi}_{\psi} ([\omega, 1])f_2\}( y) dy\\
  &=-ie^{-\tfrac{\pi i}{4}}\nu(-1, \omega)\int_{\R} e^{-2\pi i xy}e^{\pi i y^2 t} \{\overline{\Pi}_{\psi} ([\omega, 1])f_2\}( y) dy\\
  &=-ie^{-\tfrac{\pi i}{4}}\nu(-1, \omega)[ e^{-\tfrac{\pi i}{4}}\nu(-1, \omega)]\int_{\R} e^{-2\pi i xy}e^{\pi i y^2 t} \widehat{ f_2}( -y) dy\\
 &=\int_{\R}e^{-2\pi i xy} \widehat{ f_2}( -y)e^{\pi i y^2t}dy.
\end{align*}
 \begin{align*}
 d\overline{\Pi}_{\psi} (e^-)f([-1, x])&= d\overline{\Pi}_{\psi} (e^-)f_2( x)\\
&=\frac{d}{dt}\Big\{  \overline{\Pi}_{\psi} (\exp(te^-)) f_2( x)\Big\}\Big|_{t=0} \\
&=\frac{d}{dt}\Big\{ \int_{\R}e^{-2\pi i xy} \widehat{ f_2}( -y)e^{\pi i y^2t}dy\Big\}\Big|_{t=0}\\
&=\frac{d}{dt}\Big\{ \int_{\R}e^{2\pi i xy} \widehat{ f_2}( y)e^{\pi i y^2}dy\Big\}\Big|_{t=0}\\
&=\int_{\R} (\pi i y^2)\widehat{f_2}(y) e^{2\pi i xy} dy\\
&=\frac{1}{4\pi i}\frac{d^2}{dx^2}\int_{\R}\widehat{f_2}(y) e^{2\pi i xy} dy\\.
&=\frac{1}{4\pi i}\frac{d^2}{dx^2} f_2(x).
 \end{align*}

Hence:
 \begin{align*}
 d\overline{\Pi}_{\psi} (e^-)f([\epsilon, x])&=\frac{d}{dt}\Big\{ \Pi_{\psi}(\exp(te^-)) f([\epsilon, x])\Big\}\Big|_{t=0} \\
&=\frac{\epsilon i}{4\pi }\frac{d^2}{dx^2}f([\epsilon, x]).
 \end{align*}
 Analogue of the lemma in \cite[p.199]{LiVe}, we have:
\begin{lemma}
Let $A([\epsilon,x])=e^{-\epsilon\pi x^2}$, $B([\epsilon, x])=xe^{-\epsilon\pi x^2}\in L^2(\mu_2 \times \R)$. Then:
\begin{itemize}
\item[(1)] $[ d\overline{\Pi}_{\psi} (J^0)]A([\epsilon, x])=\frac{i}{2}A([\epsilon, x])$ and $[ d\overline{\Pi}_{\psi} (J^-)]A([\epsilon, x])=0$,
\item[(2)] $[ d\overline{\Pi}_{\psi} (J^0)]B([\epsilon, x])=\frac{3i}{2} B([\epsilon, x])$ and $[ d\overline{\Pi}_{\psi} (J^-)]B([\epsilon, x])=0$.
\end{itemize}
\end{lemma}
\begin{proof}
1) \begin{align*}
& d\overline{\Pi}_{\psi} (e^+)A([\epsilon,x])- d\overline{\Pi}_{\psi} (e^-)A([\epsilon,x])\\
&=\pi i\epsilon x^2 A([\epsilon,x])-\frac{\epsilon i}{4\pi }\frac{d^2}{dx^2}A([\epsilon,x])\\
&=\pi i\epsilon x^2 A([\epsilon,x])-\pi i\epsilon x^2 A([\epsilon,x])- \frac{\epsilon i}{4\pi }[-2\epsilon \pi e^{-\epsilon \pi x^2}]\\
&=\tfrac{i}{2}  A([\epsilon,x]);
\end{align*}
\begin{align*}
2 d\overline{\Pi}_{\psi} (J^-)A([\epsilon,x])&=id \Pi_{\psi}(h)A([\epsilon,x])+ d\overline{\Pi}_{\psi} (e^+)A([\epsilon,x])+ d\overline{\Pi}_{\psi} (e^-)A([\epsilon,x])\\
&=\tfrac{i}{2} e^{-\epsilon \pi x^2} -2\epsilon \pi i x^2 e^{-\epsilon \pi x^2}+  \pi i\epsilon x^2e^{-\epsilon \pi x^2} +\pi i\epsilon x^2e^{-\epsilon \pi x^2} -\tfrac{i}{2}  e^{-\epsilon \pi x^2}\\
&=0.
\end{align*}
2) \begin{align*}
& d\overline{\Pi}_{\psi} (e^+)B([\epsilon,x])- d\overline{\Pi}_{\psi} (e^-)B([\epsilon,x])\\
&=\pi i\epsilon x^2 B([\epsilon,x])-\frac{\epsilon i}{4\pi }\frac{d^2}{dx^2}B([\epsilon,x])\\
&=\pi i\epsilon x^3e^{-\epsilon \pi x^2}-\frac{\epsilon i}{4\pi } [-2\pi \epsilon x e^{-\epsilon \pi x^2} -4\pi \epsilon x e^{-\epsilon \pi x^2}+ 4\pi^2 x^3 e^{-\epsilon \pi x^2}]\\
&=\tfrac{3i}{2}  B([\epsilon,x]);
\end{align*}
\begin{align*}
&2 d\overline{\Pi}_{\psi} (J^-)B([\epsilon,x])\\
&=id \overline{\Pi}(h)B([\epsilon,x])+ d\overline{\Pi}_{\psi} (e^+)B([\epsilon,x])+ d\overline{\Pi}_{\psi} (e^-)B([\epsilon,x])\\
&=\tfrac{i}{2} x e^{-\epsilon \pi x^2} + i x  e^{-\epsilon \pi x^2} - 2 \pi i \epsilon x^3 e^{-\epsilon \pi x^2}+ \pi i\epsilon x^3e^{-\epsilon \pi x^2} +\pi i\epsilon x^3e^{-\epsilon \pi x^2} -\tfrac{3i}{2}  xe^{-\epsilon \pi x^2}\\
&=0.
\end{align*}
\end{proof}
As a consequence(cf. \cite[Chapter \S 3]{Kn}), we have:
\begin{lemma}
\begin{itemize}
\item[(1)] $\overline{\Pi}_{\psi}([g_{t}, e^{\tfrac{-it}{2}}])A([\epsilon,x])=A([\epsilon,x]) $,
\item[(2)] $\overline{\Pi}_{\psi}([g_{t}, e^{\tfrac{-3it}{2}}])B([\epsilon,x])=B([\epsilon,x])$,
\end{itemize}
for $g_t=\begin{pmatrix} \cos t & \sin t\\ -\sin t & \cos t\end{pmatrix} $, for $-\pi <t<  \pi $.
\end{lemma}
\begin{proof}
1) $$\lim\limits_{t\to 0} \frac{\overline{\Pi}_{\psi}([g_{t}, e^{\tfrac{-it}{2}}])A([\epsilon,x])-A([\epsilon,x])}{t}=d\overline{\Pi}_{\psi} (J^-) A([\epsilon,x])-\tfrac{i}{2} A([\epsilon,x])=0.$$
Consequently,
$$\lim\limits_{t\to 0} \frac{\overline{\Pi}_{\psi}([g_{t}, e^{\tfrac{-it}{2}}])\overline{\Pi}_{\psi}([g_{t_0}, e^{\tfrac{-it_0}{2}}])A([\epsilon,x])-\overline{\Pi}_{\psi}([g_{t_0}, e^{\tfrac{-it_0}{2}}])A([\epsilon,x])}{t}=0.$$
2) Similarly.
\end{proof}
\begin{lemma}
 $\overline{\Pi}_{\psi}([h(-1), -i])_{\R}A[\epsilon, x]=(-1, \epsilon)_{\R}A[\epsilon,x] $ and $\overline{\Pi}_{\psi}([h(-1), (-i)^3])B[\epsilon,x]=(-1,\epsilon)_\R  B[\epsilon,x]$.
 \end{lemma}
 \begin{proof}
 By (\ref{ac21}), we have:
 \begin{align*}
 \overline{\Pi}_{\psi}([h(-1),-i])A([\epsilon, x])&=(-i)e^{\tfrac{\pi i}{4}[1+1]}(-1,\epsilon)_\R A([\epsilon,-x])\\
 &=(-1,\epsilon)_\R A([\epsilon,-x]).
 \end{align*}
  \begin{align*}
 \overline{\Pi}_{\psi}([h(-1),i])B([\epsilon, x])&=ie^{\tfrac{\pi i}{4}[1+1]}(-1,\epsilon)_\R B([\epsilon,-x])\\
 &=(-1,-\epsilon)_\R  B([\epsilon,-x])\\
  &=(-1,\epsilon)_\R  B([\epsilon,x]).
 \end{align*}
 \end{proof}

\subsubsection{Theta series}
For $p=\begin{pmatrix} a & b\\ 0& a^{-1} \det p  \end{pmatrix} \in P_{>0}^{\pm}(\R)$, let us write
$$p=\begin{pmatrix} a& 0\\ 0& a^{-1} \end{pmatrix}\begin{pmatrix} 1 & \tfrac{b\det p}{a}\\ 0&1 \end{pmatrix}\begin{pmatrix} 1 & 0\\ 0&\det p  \end{pmatrix}.$$
By (\ref{ph}), $g$ corresponds to the element  $z_p=ab\det p + ia^2 \det p$ of  $\mathcal{H}^{\pm}$.  By (\ref{representationsp23})(\ref{representationsp24}), we have:
\begin{equation}
\overline{\Pi}_{\psi} (p)A([\epsilon, x])= a^{\tfrac{1}{2}} e^{-(\det p) \epsilon \pi x^2 a^2} \cdot e^{\pi i \epsilon (\det p) x^2 ab}=  a^{\tfrac{1}{2}}e^{i \epsilon  \pi x^2 z_p}.
\end{equation}
\begin{equation}
\overline{\Pi}_{\psi} (p)B([\epsilon, x])=  a^{\tfrac{3}{2}}x e^{-(\det p) \epsilon \pi x^2 a^2} \cdot e^{\pi i \epsilon (\det p) x^2 ab}= a^{\tfrac{3}{2}}x  e^{i \epsilon  \pi x^2 z_p}.
\end{equation}

Let us consider  $\Gamma=\breve{\Gamma}(2)^{\pm}$, which is  a discrete group of $\SL_2^{\pm}(\R)$.   Recall $\theta_{L, X^{\ast}}$  from (\ref{eq7}).
Let us write $g=p_gk_g$, for $g\in \SL_2^{\pm}(\R)$, $p_g\in P_{>0}^{\pm}(\R)$, $k_g\in \SO_2(\R)$. Then:
 \begin{align*}
 \theta_{L, X^{\ast}}(A)([\epsilon, 0])=\sum_{l\in L\cap X}A(\epsilon,  l)=\sum_{n\in \Z}A(\epsilon,  n).
 \end{align*}
$$z_{p_g}=\frac{ai+b}{ci+d}=gi=z_g.$$
 If $k_g\neq h(-1)$, then  
 $$J(g,i)\overline{\Pi}_{\psi} (g)A=J(p_g,i)J(k_g,i)\overline{\Pi}_{\psi} (p_g)\overline{\Pi}_{\psi} (k_g)A=J(p_g,i)\overline{\Pi}_{\psi} (p_g)A,$$
 $$J(g,i)^3\overline{\Pi}_{\psi} (g)B=J(p_g,i)^3J(k_g,i)^3\overline{\Pi}_{\psi} (p_g)\overline{\Pi}_{\psi} (k_g)^3B=J(p_g,i)^3\overline{\Pi}_{\psi} (p_g)B.$$
Let:
 \begin{align*}
 \theta_{1/2}(p_g, \epsilon)&= \theta_{L, X^{\ast}}\Big(J(p_g,i)\overline{\Pi}_{\psi} (p_g)A\Big)[\epsilon, 0]\\
&=\sum_{n\in \Z}J(p_g,i)\overline{\Pi}_{\psi} (p_g)A(\epsilon,  n)\\
&=\sum_{n\in \Z} e^{i  \epsilon   \pi n^2 z_g}.
 \end{align*}
 \begin{align*}
 \theta_{3/2}(p_g, \epsilon)&=\theta_{L, X^{\ast}}\Big(J(p_g,i)^3\overline{\Pi}_{\psi} (p_g)B\Big)[\epsilon, 0]\\
 &=\sum_{n\in \Z}J(p_g,i)^3\overline{\Pi}_{\psi} (p_g)B(\epsilon,  n)\\
&=\sum_{n\in \Z}n e^{i \epsilon  \pi n^2 z_g}.
  \end{align*}
Let:
\begin{align*}
\theta_{1/2}(z, \epsilon)&=\sum_{n\in \Z} e^{i \epsilon  \pi n^2 z},
\end{align*}
\begin{align*}
\theta_{3/2}(z, \epsilon)&=\sum_{n\in \Z} n e^{i \epsilon  \pi n^2 z},
\end{align*}
for $z\in \mathcal{H}^{\pm}$, $\epsilon \in \mu_2$.  Let us consider the explicit action of $\Gamma$ on $\theta_{1/2}(z,\epsilon)$ and  $\theta_{3/2}(z,\epsilon)$ from Section \ref{Gamma2}. Note that if $k_{\gamma p_g}=h(-1)$, then $\gamma p_g=-p_{\gamma p_g}$, which implies that $-\gamma\in P^{\pm}_{>0}(\R)$, and then $\gamma=h(-1)$ or $h(-1)h_{-1}$. For $\gamma \in \Gamma$, we have:
\begin{align*}
\theta_{1/2}(\gamma p_g, \epsilon)&=\theta_{L, X^{\ast}}\Big(J(\gamma  p_g,i)\overline{\Pi}_{\psi} (\gamma p_g)A\Big)[\epsilon, 0]\\
&= \overline{C}_{X^{\ast}}(\gamma, p_g)^{-1}J(\gamma ,p_gi)  \overline{\Pi}_{\psi} (\gamma)\theta_{L, X^{\ast}}\Big(J(p_g,i)\overline{\Pi}_{\psi} ( p_g)A\Big)[\epsilon, 0] \\
&= \overline{C}_{X^{\ast}}(\gamma,p_g)^{-1}J(\gamma ,p_gi) m_{X^{\ast}} (\gamma) \Pi_{\psi} (\gamma)\theta_{L, X^{\ast}}\Big(J(p_g,i)\overline{\Pi}_{\psi} ( p_g)A\Big)[\epsilon, 0] \\
&\stackrel{p_gi=z}{=}\overline{C}_{X^{\ast}}(\gamma,p_g)^{-1} J(\gamma ,z)\Upsilon(\gamma) m_{X^{\ast}} (\gamma)\sum_{n\in \Z} e^{i (\det \gamma)\epsilon  \pi n^2 z}\\
&=\lambda^{\pm}(\gamma, \epsilon) J(\gamma ,z)\sum_{n\in \Z} e^{i (\det \gamma)\epsilon  \pi n^2 z};
\end{align*}
\begin{align*}
\theta_{3/2}(\gamma p_g,  \epsilon)&=\theta_{L, X^{\ast}}\Big(J(\gamma  p_g,i)^3\overline{\Pi}_{\psi} (\gamma p_g)B\Big)[\epsilon, 0]\\
&=  \overline{C}_{X^{\ast}}(\gamma,p_g)^{-1}J(\gamma ,p_gi)^3  \overline{\Pi}_{\psi} (\gamma)\theta_{L, X^{\ast}}\Big(J(p_g,i)^3\overline{\Pi}_{\psi} ( p_g)B\Big)[\epsilon, 0] \\
&=  \overline{C}_{X^{\ast}}(\gamma,p_g)^{-1}J(\gamma ,p_gi)^3 m_{X^{\ast}} (\gamma) \Pi_{\psi} (\gamma)\theta_{L, X^{\ast}}\Big(J(p_g,i)^3\overline{\Pi}_{\psi} ( p_g)B\Big)[\epsilon, 0] \\
&\stackrel{p_gi=z}{=}  \overline{C}_{X^{\ast}}(\gamma,p_g)^{-1} J(\gamma ,z)^3\Upsilon(\gamma) m_{X^{\ast}} (\gamma) \sum_{n\in \Z}n e^{i (\det \gamma)\epsilon  \pi n^2 z}\\
&=\lambda^{\pm}(\gamma, \epsilon) J(\gamma ,z)^3\sum_{n\in \Z} n e^{i (\det \gamma)\epsilon  \pi n^2 z}.
\end{align*}
for the constant $\lambda^{\pm}(\gamma, \epsilon)= \overline{C}_{X^{\ast}}(\gamma,p_g)^{-1}\Upsilon(\gamma)m_{X^{\ast}} (\gamma) $, where $ \overline{C}_{X^{\ast}}(\gamma,p_g)^{-1}$ only depends on $\gamma$ and $\det g$.
\paragraph{} Let us give   the expression of $\lambda^{\pm}(\gamma)$ in the following. Let $\gamma=\begin{pmatrix}
                                        a & b \\
                                        c & d 
                                      \end{pmatrix}\in \Gamma $. Let us define $\sgn(z)=1$ or $-1$ according to $z\in \mathcal{H}$ or $\mathcal{H}^-$.\\
\begin{itemize}
\item If $c=0$, $\Upsilon(\gamma)=(\epsilon \det \gamma, a)_{\R}$, $\overline{C}_{X^{\ast}}(\gamma, p_g)^{-1}=(\det p_g, a)_{\R}=(\sgn z_g, a)_{\R}$, $m_{X^{\ast}}(\gamma)=e^{\tfrac{\pi i}{4}[1-\sgn(a)]}$. 
\begin{align}
\lambda^{\pm}(\gamma, \epsilon)&=\overline{C}_{X^{\ast}}(\gamma,p_g)^{-1}\Upsilon(\gamma) m_{X^{\ast}}(\gamma) e^{\tfrac{\pi i}{4}[1-\sgn(a)]} \\
&=(\epsilon\det \gamma, a)_{\R}(\sgn z_g, a)_{\R}e^{\tfrac{\pi i}{4}[1-\sgn(a)]}.
\end{align}   
\item     If $c\neq 0$, $\Upsilon(\gamma)=e^{\tfrac{\pi i(1-\epsilon)}{4}} (c \det \gamma, \epsilon\det\gamma)_{\R} \beta(d, (\det \gamma) \epsilon c)^{-1}$,   $\overline{C}_{X^{\ast}}(\gamma,p_g)=1$, $m_{X^{\ast}}(\gamma)=e^{\tfrac{\pi i}{4}[-\sgn(c)]}$.
    \begin{align}
\lambda^{\pm}(\gamma, \epsilon)&=\overline{C}_{X^{\ast}}(\gamma,p_g)^{-1}\Upsilon(\gamma) m_{X^{\ast}}(\gamma)  \\
&=e^{\tfrac{\pi i}{4}[-\sgn(c)]}e^{\tfrac{\pi i(1-\epsilon)}{4}} (c\det \gamma, \epsilon\det \gamma)_{\R} \beta(d, (\det \gamma) \epsilon c)^{-1}.
\end{align}
\end{itemize} 
\begin{remark}
 If $\epsilon=1$, $\det \gamma=1$, the above constant $\lambda^{\pm}$ is compatible with  the constant $\lambda$ given in \cite[p.2.4.13, theorem]{LiVe} by adding the factor $J(\gamma ,z)$  in our choice.
\end{remark}    
 Hence:
\begin{align}\label{theta1}
 \theta_{1/2}(\tfrac{az+b}{cz+d}, \epsilon)=\lambda^{\pm}(\gamma, \epsilon) \sqrt{\det \gamma(cz+d)}\sum_{n\in \Z} e^{i (\det \gamma)\epsilon  \pi n^2 z},
\end{align}
\begin{align}\label{theta2}
 \theta_{3/2}(\tfrac{az+b}{cz+d}, \epsilon)=\lambda^{\pm}(\gamma, \epsilon) (\sqrt{\det \gamma(cz+d)})^3\sum_{n\in \Z} ne^{i (\det \gamma)\epsilon  \pi n^2 z},
\end{align}
for $z\in \mathcal{H}^{\pm}$, $\gamma\in \Gamma $ if $\epsilon=1$ or  $\gamma\in \Gamma\setminus \{ h(-1), h(-1)h_{-1} \}$ if $\epsilon\neq -1$, and some constant $\lambda^{\pm}(\gamma, \epsilon)$.
 \begin{remark}
 One can add  $\{ h(-1), h(-1)h_{-1} \}$ to change the constant in above two statements, but it is indeed a projective representation of $\Gamma$, which does not factor through the group $\Gamma/\{\pm 1\}$.
\end{remark}      
\section{Appendix: Weil representation of  $\Mp_2(\R)$}\label{app}
The  purpose of this section is to study the works \cite{LiVe}, \cite{Pe}, and \cite{Ra} in the simple case---$\dim W=2$. Let $\psi=\psi_0: t\longrightarrow e^{2\pi i t}$.   Let $dx$ denote the Lebesgue measure on $\R$.   Recall the  classical  Fourier transformation:
\begin{align*}
\mathcal{F}: & L^2(\R) \longrightarrow L^2(\R);\\
                                & f \longmapsto \mathcal{F}(f)(x)=\int_{\R} \psi( -xt) f(t)  dt
\end{align*}
for $f\in S(\R)$. By convection, let us write $ \hat{f}=\mathcal{F}(f)$. Then:
\begin{itemize}
\item $\hat{\hat{f}}(x)=f(-x)$.
\item  $\hat{f}=f$, for $f(x)=e^{-\pi x^2}$.
\item  $(\tau_t(f))^{\textasciicircum }(x)=\hat{f}(x) e^{2\pi ixt}$, where $\tau_t(f)(x)=f(x+t)$.
\item If $g(x)=f(x)e^{2\pi i xt}$, then  $\hat{g}(x)=\hat{f}(x-t)$.
\item If $g(x)=\lambda^{-1}f(\lambda^{-1}x)$, then  $\hat{g}(x)=\hat{f}(\lambda x)$.
\item $\int_{x\in \R} \hat{f}(x) e^{\pi i x^2} dx=e^{\tfrac{\pi i}{4}} \int_{x\in \R} f(x) e^{-\pi i x^2} dx$.(cf.\cite[p.50, Lemma]{LiVe})
\item $\int_{x\in \R} \hat{f}(x) e^{-\pi i x^2} dx=e^{-\tfrac{\pi i}{4}} \int_{x\in \R} f(x) e^{\pi i x^2} dx$.(cf.\cite[p.50, Lemma]{LiVe})
\end{itemize}
Let  $X=\R e_1\simeq \R$, $X^{\ast}=\R e_1^{\ast}\simeq \R$. By the fixed basis $\{e_1, e_1^{\ast}\}$, let us identity $X$  with $\R$, and $X^{\ast}$ with $\R$.  Let $dx$, $dx^{\ast}$ denote the Lebesgue measures on $X$ and $X^{\ast}$ respectively.  Let us give  two classical  Fourier transformations(cf. \cite[p.147]{We}):
\begin{align*}
\mathcal{F}'_{XX^{\ast} }: &  L^2(X)\longrightarrow  L^2(X^{\ast});\\
                                & f \longmapsto \mathcal{F}'_{XX^{\ast} }(f)(x^{\ast})=\int_{X} \psi(\langle x, x^{\ast}\rangle) f(x)  dx,
\end{align*}
\begin{align*}
\mathcal{F}'_{X^{\ast} X}: &  L^2(X^{\ast})\longrightarrow  L^2(X);\\
                                & g \longmapsto \mathcal{F}'_{X^{\ast} X}(g)(x)=\int_{X^{\ast}} \psi(\langle x^{\ast}, x\rangle)g(x^{\ast})  dx^{\ast},
\end{align*}
for $f\in L^2(X)$, $g\in L^2(X^{\ast})$. Then  $\mathcal{F}'_{ X^{\ast}X}=\mathcal{F}_{X X^{\ast}}^{'-1}$. Let $\pi_{\psi} $ be the Weil  representation of $\Ha(W)$ associated to $\psi$. Then $\pi_{\psi}\simeq \Ind_{X\times \R}^{\Ha(W)}(\id\otimes \psi)$.  Following \cite{Pe}, let $\mathcal{H}(X)$ be  the set of measurable functions $f: \Ha(W) \longrightarrow \C$ such that
\begin{itemize}
\item[(i)] $f([x,t]h)=\psi(t) f(h)$, $x\in X$, $t\in \R$,  $h\in \Ha(W)$;
\item[(ii)] $\int_{(X\times \R)\setminus \Ha(W)} ||f(\dot{h})||^2 d\dot{h}<+\infty$, where  $d\dot{h}$ is an $\Ha(W)$-right invariant measure on $(X\times \R)\setminus \Ha(W)$.
\end{itemize}
Note that $(X\times \R)\setminus \Ha(W)$ is homeomorphic to $X^{\ast}$, so  we  choose the  $d\dot{w}$  corresponding to  the Lebesgue measure on $\R e_1^{\ast}$ through the base $e_1^{\ast}$. Then $\pi_{\psi}$ can be realized on  $L^2(X^{\ast})$ by the following formulas:
\begin{equation}\label{representationsp211}
\pi_{\psi}[(x^{\ast},0)+(x,0)+(0,k)]g(y^{\ast})=\psi(k+\langle x^{\ast}+y^{\ast},x\rangle) g(x^{\ast}+y^{\ast}),
\end{equation}
for  $g\in L^2(X^{\ast})$, $x\in X$, $x^{\ast}\in X^{\ast}$, $k\in \R$. Similarly, we can also consider $\mathcal{H}(X^{\ast})=\Ind_{X^{\ast}\times F}^{\Ha(W)}(\id\otimes \psi) $. By \cite[p.373]{Pe}, we can define the Fourier transformation:
\begin{align*}
\mathcal{F}_{X^{\ast} X}: &  \mathcal{H}(X)\longrightarrow  \mathcal{H}(X^{\ast});\\
                                & f \longmapsto \mathcal{F}_{X^{\ast} X}(f)(h)=\int_{X^{\ast}}  f([x^{\ast}, 0]h)  dx^{\ast}.
\end{align*}
It can be checked that $\mathcal{F}_{X^{\ast} X}(f)([x^{\ast},t]h)=\psi(t)\mathcal{F}_{X^{\ast} X}(f)(h)$, for $[x^{\ast}, t]\in X^{\ast} \times \R$. Moreover, there exists the following commutative diagram:
\[
\begin{CD}
\mathcal{H}(X) @>\mathcal{F}_{X^{\ast} X}>> \mathcal{H}(X^{\ast}) \\
@V\wr VV @V\wr VV \\
L^2(X^{\ast})@>\mathcal{G}_{X^{\ast} X}>> L^2(X)
\end{CD}
\]
Let us check that $\mathcal{G}_{X^{\ast} X}=\mathcal{F}'_{X^{\ast} X}$:
\begin{align*}
\mathcal{G}_{X^{\ast} X}(f)(x)&=\int_{X^{\ast}} f([y^{\ast},0][x,0])dy^{\ast}\\
&=\int_{X^{\ast}} \psi(\langle y^{\ast}, x\rangle) f(y^{\ast})dy^{\ast}\\
&=\mathcal{F}'_{X^{\ast} X}(f)(x).
\end{align*}
Hence $\mathcal{F}_{X^{\ast} X}=\mathcal{F}_{X X^{\ast}}^{-1}$. For simplicity, we will not  distinguish between $\mathcal{F}_{X^{\ast}X}$ and $\mathcal{F}'_{X^{\ast}X}$. For two general  Lagrangian lines $Y^{\ast}=\R e_{Y^{\ast}}$ and $Z^{\ast}=\R e_{Z^{\ast}}$, we have:
 $$Y^{\ast} \cap Z^{\ast}\neq 0 \textrm{ iff } Y^{\ast}=Z^{\ast}.$$
 1) If  $Y^{\ast} \cap Z^{\ast}=0$, $W=Y^{\ast}\oplus Z^{\ast}$. Let us  define $A_{Y^{\ast} Z^{\ast}}=\langle e_{Y^{\ast}}, e_{Z^{\ast}}\rangle$ by following \cite{Pe}. Let $e'_{Y^{\ast}}=|A_{Y^{\ast}Z^{\ast}}|^{-\tfrac{1}{2}} e_{Y^{\ast}}$ and $e'_{Z^{\ast}}=|A_{Y^{\ast}Z^{\ast}}|^{-\tfrac{1}{2}} e_{Z^{\ast}}$. Then $\langle e'_{Y^{\ast}}, e'_{Z^{\ast}}\rangle=\pm1$.  By the basis $\{ e'_{Y^{\ast}}, e'_{Z^{\ast}}\}$, let  $d'y^{\ast}$, $d'z^{\ast}$ denote the Lebesgue measures on $Y^{\ast}$ and $Z^{\ast}$ respectively. By the basis $\{e_{Y^{\ast}}, e_{Z^{\ast}}\}$, let  $dy^{\ast}$, $dz^{\ast}$ denote the Lebesgue measures on $Y^{\ast}$ and $Z^{\ast}$ respectively.  Similarly, we can define:
\begin{align}\label{ZZYY}
\mathcal{F}_{ Z^{\ast}Y^{\ast}}: &    \mathcal{H}(Y^{\ast})\longrightarrow  \mathcal{H}(Z^{\ast});\\
                                & f \longmapsto \mathcal{F}_{Z^{\ast} Y^{\ast}}(f)(h)=\int_{Z^{\ast}}  f([z^{\ast}, 0]h)  d'z^{\ast}\\
                                &\qquad\quad\qquad\quad\qquad\quad=\int_{Z^{\ast}}  f([z^{\ast}, 0]h)  |A_{Y^{\ast}Z^{\ast}}|^{1/2}dz^{\ast}.
\end{align}
2) If $Y^{\ast} \cap Z^{\ast}\neq 0$, then $Y^{\ast}=Z^{\ast}$. In this case, let us define:
 $$\mathcal{F}_{Y^{\ast} Z^{\ast}}=\id.$$
For $w=xe_1+ye_1^{\ast}\in W$, let us write $w=(x,y)\begin{pmatrix} e_1\\ e_1^{\ast}\end{pmatrix}$. For $g\in \Sp(W)$, let us write $\begin{pmatrix} e_1\\ e_1^{\ast}\end{pmatrix}g=A_g\begin{pmatrix} e_1\\ e_1^{\ast}\end{pmatrix}$, for some $A_g\in \SL_2(\R)$. In this way, we will identity $\Sp(W)$ with $\SL_2(\R)$. Let us consider two elements $g_1=\begin{pmatrix} a &b\\ c& d\end{pmatrix}$ and $g_2=\begin{pmatrix} p &q\\ r& s\end{pmatrix}$ of  $\SL_2(\R)$. Let  $e_{Y}=ae_1+be_1^{\ast}$, $e_{Y^{\ast}}=ce_1+de_1^{\ast}=e_1^{\ast}g_1$, $e_Z=pe_1+qe_1^{\ast}$, $e_{Z^{\ast}}=re_1+se_1^{\ast}=e_1^{\ast}g_2$. Let $Y=\R e_{Y}$, $Y^{\ast}=\R e_{Y^{\ast}}$, $Z=\R e_Z$, $Z^{\ast}=\R e_{Z^{\ast}}$. Following \cite[pp.71-73]{LiVe}, let us first consider the case that  $c>0$, $d=0$ and $r=1,s>0$. In this case,   $Y^{\ast}=X$, and $e_{Y^{\ast}}=ce_1$,  $Z^{\ast}=\R (e_1+se_1^{\ast})$. Let us fix $W=X\oplus X^{\ast}$.\\
a) $\mathcal{H}(X^{\ast}) \simeq L^2(X)$, $\mathcal{H}(X) \simeq L^2(X^{\ast})$, $|A_{Y^{\ast}X^{\ast}}|=c$.  Then $\mathcal{F}_{Y^{\ast}X^{\ast}}=\mathcal{F}_{XX^{\ast}}: L^2(X) \longrightarrow L^2(X^{\ast})$.
\begin{align}
\mathcal{F}_{Y^{\ast}X^{\ast}}(f)(xe_{1}^{\ast})&=\int_{tce_1\in X}  f([tce_{1}, 0][xe^{\ast}_{1},0])  |A_{Y^{\ast}X^{\ast}}|^{1/2}dt\\
 &=\int_{\R} \psi(tcx) f([tce_{1}, 0])c^{\tfrac{1}{2}} dt\\
  &=\int_{\R} \psi(tx) f([te_{1}, 0])c^{-\tfrac{1}{2}} dt.
\end{align}
b) $\mathcal{H}(Z^{\ast}) \simeq L^2(X)$, $\mathcal{H}(X) \simeq L^2(X^{\ast})$, $|A_{Z^{\ast}Y^{\ast}}|= |cs|=cs$, $\mathcal{F}_{Z^{\ast}Y^{\ast}}=\mathcal{F}_{Z^{\ast}X}: L^2(X^{\ast}) \longrightarrow L^2(X)$.
\begin{align}
\mathcal{F}_{Z^{\ast}Y^{\ast}}(f)(xe_{1})&=\int_{te_{Z^{\ast}}\in Z^{\ast}}  f([te_{Z^{\ast}}, 0][xe_1,0])  |A_{Z^{\ast}Y^{\ast}}|^{1/2}dt\\
&=\int_{te_{Z^{\ast}}\in Z^{\ast}}  f([(t+x)e_{1}, -\tfrac{st^2}{2}-tsx][tse^{\ast}_1,0]) (cs)^{\tfrac{1}{2}}dt\\
 &=\int_{\R} \psi(-tsx-\tfrac{st^2}{2}) f([tse_{1}^{\ast}, 0])(cs)^{\tfrac{1}{2}} dt\\
 &=\int_{\R} \psi(-ts^{\tfrac{1}{2}}x-\tfrac{t^2}{2}) f([ts^{\tfrac{1}{2}}e_{1}^{\ast}, 0]) c^{\tfrac{1}{2}} dt.
\end{align}
c) $\mathcal{H}(X^{\ast}) \simeq L^2(X)$, $\mathcal{H}(Z^{\ast}) \simeq L^2(X)$, $|A_{X^{\ast}Z^{\ast}}|= 1$, $\mathcal{F}_{X^{\ast}Z^{\ast}}: L^2(X) \longrightarrow L^2(X)$.
\begin{align}
\mathcal{F}_{X^{\ast}Z^{\ast}}(f)(xe_{1})&=\int_{te_1^{\ast} \in X^{\ast}}  f([te_1^{\ast}, 0][xe_{1},0])  |A_{X^{\ast}Z^{\ast}}|^{1/2}dt\\
&=\int_{te_1^{\ast} \in X^{\ast}}  f([te_1+tse_1^{\ast}, -\tfrac{st^2}{2}][(x-t)e_1, 0]) s dt\\
 &=\int_{\R} \psi(-\tfrac{st^2}{2}) f([(x-t)e_{1},0]) sdt\\
  &=\int_{\R} \psi(-\tfrac{t^2}{2}) f([(x- s^{-\tfrac{1}{2}}t)e_{1},0]) s^{\tfrac{1}{2}}dt.
\end{align}
d) $\mathcal{H}(X^{\ast}) \simeq L^2(X)$, $\mathcal{H}(Z^{\ast}) \simeq L^2(X)$, $\mathcal{F}_{Z^{\ast}X^{\ast}}: L^2(X) \longrightarrow L^2(X)$.
\begin{align}
\mathcal{F}_{Z^{\ast}X^{\ast}}(f)(xe_{1})&=\int_{te_{Z^{\ast}}\in Z^{\ast}}  f([te_{Z^{\ast}}, 0][xe_{1},0])  |A_{Z^{\ast}X^{\ast}}|^{1/2}dt\\
&=\int_{te_{Z^{\ast}}\in Z^{\ast}}   f([(t+x)e_1+tse_1^{\ast}, -\tfrac{tsx}{2}])  dt\\
&=\int_{te_{Z^{\ast}}\in Z^{\ast}}   f([tse_1^{\ast}, \tfrac{st^2}{2}][(t+x)e_1, 0])  dt\\
 &=\int_{\R} \psi(\tfrac{st^2}{2}) f([(x+t)e_{1},0]) dt\\
  &=\int_{\R} \psi(\tfrac{t^2}{2}) f([(x+ts^{-\tfrac{1}{2}})e_{1},0])s^{-\tfrac{1}{2}} dt.
\end{align}
\begin{lemma}[{\cite[p.47]{LiVe}}]\label{pi1}
For the above $X^{\ast}, Y^{\ast}, Z^{\ast}$, $\mathcal{F}_{ X^{\ast}Z^{\ast}} \circ \mathcal{F}_{Z^{\ast}Y^{\ast} } \circ \mathcal{F}_{Y^{\ast}X^{\ast} } =e^{-\tfrac{\pi i}{4}} \id_{\mathcal{H}(X^{\ast})}$.
\end{lemma}
\begin{proof}
By the base $e_1$, we will  identity $X$ with $\R$. For $f\in S(X)$, let $g(t)= f(t) e^{-2\pi i x t s}$, $h(t)=s^{-\tfrac{1}{2}}g(s^{-\tfrac{1}{2}} t)$. Then:
$$ \hat{h}(t)=\hat{g}(s^{1/2}t)=\hat{f}(s^{1/2}t+sx).$$
\begin{align}
&\mathcal{F}_{Z^{\ast}Y^{\ast} } \circ \mathcal{F}_{Y^{\ast}X^{\ast} }(f)(xe_1)\\
&=\int_{\R} \psi(-ts^{\tfrac{1}{2}}x-\tfrac{t^2}{2}) \mathcal{F}_{Y^{\ast}X^{\ast}}(f)([ts^{\tfrac{1}{2}}e_1^{\ast},0])c^{\tfrac{1}{2}}dt\\
&=\int_{\R} \int_{\R} \psi(ts^{\tfrac{1}{2}}u-ts^{\tfrac{1}{2}}x-\tfrac{t^2}{2})f([ue_1,0])dudt \\
&=\int_{\R} \int_{\R} \psi(-ts^{\tfrac{1}{2}}u+ts^{\tfrac{1}{2}}x-\tfrac{t^2}{2})f([ue_1,0])dudt\\
&=\int_{\R}\psi(ts^{\tfrac{1}{2}}x-\tfrac{t^2}{2})\hat{f}(ts^{\tfrac{1}{2}})dt\\
&=\psi(\tfrac{sx^2}{2})\int_{\R}\psi(-\tfrac{(t-s^{\tfrac{1}{2}}x)^2}{2})\hat{f}(ts^{\tfrac{1}{2}})dt\\
&=\psi(\tfrac{sx^2}{2})\int_{\R}\psi(-\tfrac{t^2}{2})\hat{f}(ts^{\tfrac{1}{2}}+sx)dt\\
&=\psi(\tfrac{sx^2}{2})\int_{\R}\psi(-\tfrac{t^2}{2})\hat{h}(t)dt\\
&=e^{-\tfrac{\pi i}{4}}\psi(\tfrac{sx^2}{2})\int_{t\in \R}   h(t) e^{\pi i t^2} dt\\
&=e^{-\tfrac{\pi i}{4}}\int_{t\in \R}   s^{-\tfrac{1}{2}} f( s^{-\tfrac{1}{2}}t) e^{\pi i s x^2} e^{-2\pi i xts^{\tfrac{1}{2}}}e^{\pi i t^2} dt\\
&=e^{-\tfrac{\pi i}{4}}\int_{t\in \R}   s^{-\tfrac{1}{2}} f( s^{-\tfrac{1}{2}}t) e^{\pi i (t-s^{\tfrac{1}{2}}x)^2} dt\\
&=e^{-\tfrac{\pi i}{4}}\int_{t\in \R}   s^{-\tfrac{1}{2}} f( s^{-\tfrac{1}{2}}t+x) \psi(\tfrac{t^2}{2}) dt\\
&=e^{-\tfrac{\pi i}{4}}\mathcal{F}_{Z^{\ast}X^{\ast}}(f)(xe_{1}).
\end{align}
\end{proof}
Let us consider the second case that $c>0$, $d=0$ and $r=1,s=-|s|<0$. In this case,   $Y^{\ast}=X=\R ce_1$, $Z^{\ast}=\R (e_1+se_1^{\ast})$. Let us fix $W=X\oplus X^{\ast}$.\\
b') $\mathcal{H}(Z^{\ast}) \simeq L^2(X)$, $\mathcal{H}(X) \simeq L^2(X^{\ast})$, $\mathcal{F}_{Z^{\ast}Y^{\ast}}=\mathcal{F}_{Z^{\ast}X}: L^2(X^{\ast}) \longrightarrow L^2(X)$.
\begin{align}
\mathcal{F}_{Z^{\ast}Y^{\ast}}(f)(xe_{1})&=\int_{te_{Z^{\ast}}\in Z^{\ast}}  f([te_{Z^{\ast}}, 0][xe_1,0])  |A_{Z^{\ast}Y^{\ast}}|^{1/2}dt\\
&=\int_{te_{Z^{\ast}}\in Z^{\ast}}  f([(t+x)e_{1}, -\tfrac{st^2}{2}-tsx][tse^{\ast}_1,0]) (c|s|)^{\tfrac{1}{2}}dt\\
 &=\int_{\R} \psi(-tsx-\tfrac{st^2}{2}) f([tse_{1}^{\ast}, 0])(c|s|)^{\tfrac{1}{2}} dt\\
 &=\int_{\R} \psi(t|s|^{\tfrac{1}{2}}x+\tfrac{t^2}{2}) f([-t|s|^{\tfrac{1}{2}}e_{1}^{\ast}, 0]) c^{\tfrac{1}{2}}dt.
\end{align}
c') $\mathcal{H}(X^{\ast}) \simeq L^2(X)$, $\mathcal{H}(Z^{\ast}) \simeq L^2(X)$, $\mathcal{F}_{X^{\ast}Z^{\ast}}: L^2(X) \longrightarrow L^2(X)$.
\begin{align}
\mathcal{F}_{X^{\ast}Z^{\ast}}(f)(xe_{1})&=\int_{te_1^{\ast} \in X^{\ast}}  f([te_1^{\ast}, 0][xe_{1},0])  |A_{X^{\ast}Z^{\ast}}|^{1/2}dt\\
&=\int_{te_1^{\ast} \in X^{\ast}}  f([xe_1+te_1^{\ast}, -\tfrac{tx}{2}])  dt\\
&=\int_{te_1^{\ast} \in X^{\ast}}  f([xe_1+tse_1^{\ast}, -\tfrac{tsx}{2}]) s dt\\
&=\int_{te_1^{\ast} \in X^{\ast}}  f([te_1+tse_1^{\ast}, -\tfrac{st^2}{2}][(x-t)e_1, 0]) s dt\\
 &=\int_{\R} \psi(-\tfrac{st^2}{2}) f([(x-t)e_{1},0]) sdt\\
  &=\int_{\R} \psi(\tfrac{t^2}{2}) f([(x+ |s|^{-\tfrac{1}{2}}t)e_{1},0]) |s|^{\tfrac{1}{2}}dt.
\end{align}
d') $\mathcal{H}(X^{\ast}) \simeq L^2(X)$, $\mathcal{H}(Z^{\ast}) \simeq L^2(X)$, $\mathcal{F}_{Z^{\ast}X^{\ast}}: L^2(X) \longrightarrow L^2(X)$.
\begin{align}
\mathcal{F}_{Z^{\ast}X^{\ast}}(f)(xe_{1})&=\int_{te_{Z^{\ast}}\in Z^{\ast}}  f([te_{Z^{\ast}}, 0][xe_{1},0])  |A_{Z^{\ast}X^{\ast}}|^{1/2}dt\\
&=\int_{te_{Z^{\ast}}\in Z^{\ast}}   f([(t+x)e_1+tse_1^{\ast}, -\tfrac{tsx}{2}])  dt\\
&=\int_{te_{Z^{\ast}}\in Z^{\ast}}   f([tse_1^{\ast}, \tfrac{st^2}{2}][(t+x)e_1, 0])  dt\\
 &=\int_{\R} \psi(\tfrac{st^2}{2}) f([(x+t)e_{1},0]) dt\\
  &=\int_{\R} \psi(-\tfrac{t^2}{2}) f([(x+t|s|^{-\tfrac{1}{2}})e_{1},0])|s|^{-\tfrac{1}{2}} dt.
\end{align}
\begin{lemma}[{\cite[pp.47-48]{LiVe}}]\label{pi2}
For the above $X^{\ast}, Y^{\ast}, Z^{\ast}$, $\mathcal{F}_{ X^{\ast}Z^{\ast}} \circ \mathcal{F}_{Z^{\ast}Y^{\ast} } \circ \mathcal{F}_{Y^{\ast}X^{\ast} } =e^{\tfrac{\pi i}{4}} \id_{\mathcal{H}(X^{\ast})}$.
\end{lemma}
\begin{proof}
By the base $e_1$, we will  identity $X$ with $\R$. For $f\in S(X)$, let $g(t)= f(t) e^{-2\pi i x t s}$, $h(t)=|s|^{-\tfrac{1}{2}}g(|s|^{-\tfrac{1}{2}} t)$. Then:
$$ \hat{h}(t)=\hat{g}(|s|^{1/2}t)=\hat{f}(|s|^{1/2}t+sx).$$
\begin{align}
\mathcal{F}_{Z^{\ast}Y^{\ast} } \circ \mathcal{F}_{Y^{\ast}X^{\ast} }(f)(xe_1)&=\int_{\R} \psi(t|s|^{\tfrac{1}{2}}x+\tfrac{t^2}{2}) \mathcal{F}_{Y^{\ast}X^{\ast}}(f)([-t|s|^{\tfrac{1}{2}}e_1^{\ast},0])c^{\tfrac{1}{2}}dt\\
&=\int_{\R}\psi(t|s|^{\tfrac{1}{2}}x+\tfrac{t^2}{2})\hat{f}(t|s|^{\tfrac{1}{2}})dt\\
&=\psi(\tfrac{sx^2}{2})\int_{\R}\psi(\tfrac{(t+|s|^{\tfrac{1}{2}}x)^2}{2})\hat{f}(t|s|^{\tfrac{1}{2}})dt\\
&=\psi(\tfrac{sx^2}{2})\int_{\R}\psi(\tfrac{t^2}{2})\hat{f}(t|s|^{\tfrac{1}{2}}+sx)dt\\
&=\psi(\tfrac{sx^2}{2})\int_{\R}\psi(\tfrac{t^2}{2})\hat{h}(t)dt\\
&=e^{\tfrac{\pi i}{4}}\psi(\tfrac{sx^2}{2})\int_{t\in \R}   h(t) e^{-\pi i t^2} dt\\
&=e^{\tfrac{\pi i}{4}}\int_{t\in \R}   |s|^{-\tfrac{1}{2}} f( |s|^{-\tfrac{1}{2}}t) e^{\pi i s x^2} e^{2\pi i xt|s|^{\tfrac{1}{2}}}e^{-\pi i t^2} dt\\
&=e^{\tfrac{\pi i}{4}}\int_{t\in \R}   |s|^{-\tfrac{1}{2}} f( |s|^{-\tfrac{1}{2}}t) e^{-\pi i (t-|s|^{\tfrac{1}{2}}x)^2} dt\\
&=e^{\tfrac{\pi i}{4}}\int_{t\in \R}   |s|^{-\tfrac{1}{2}} f( |s|^{-\tfrac{1}{2}}t+x) \psi(-\tfrac{t^2}{2}) dt\\
&=e^{\tfrac{\pi i}{4}}\mathcal{F}_{Z^{\ast}X^{\ast}}(f)(xe_{1}).
\end{align}
\end{proof}

\begin{lemma}[{\cite[p.56, Theorem]{LiVe}}]\label{three}
$ \mathcal{F}_{ X^{\ast}Z^{\ast}} \circ \mathcal{F}_{Z^{\ast}Y^{\ast} } \circ \mathcal{F}_{Y^{\ast}X^{\ast} }= \widetilde{c}_{X^{\ast}}(g_2 g_1^{-1}, g_1)\id_{\mathcal{H}(X^{\ast})}$.
\end{lemma}
\begin{proof}
1) If $c=0$, then $ Y^{\ast}=X^{\ast}$, $\mathcal{F}_{X^{\ast} Y^{\ast}}=\id$, and  $\mathcal{F}_{Y^{\ast} Z^{\ast}} \mathcal{F}_{Z^{\ast} X^{\ast}}=\mathcal{F}_{Y^{\ast} Z^{\ast}} \mathcal{F}_{Z^{\ast} X^{\ast}}=\id$.  On the right hand side, $ \widetilde{c}_{X^{\ast}}(g_2 g_1^{-1}, g_1)=1$, so the result holds.\\
2) If $r=0$,  then $Z^{\ast}=X^{\ast}$, $\mathcal{F}_{X^{\ast} Y^{\ast}} \mathcal{F}_{Y^{\ast} Z^{\ast}}=\mathcal{F}_{X^{\ast} Y^{\ast}} \mathcal{F}_{Y^{\ast} X^{\ast}}=\id$, $\mathcal{F}_{Z^{\ast} X^{\ast}}=\id$. Moreover, $  \widetilde{c}_{X^{\ast}}(g_2 g_1^{-1}, g_1)=1$, so the result holds.\\
3) Note that $g_2g_1^{-1}=\begin{pmatrix} \ast &\ast\\- cs+dr& \ast\end{pmatrix}$.  If $cs-dr=0$, then $Y^{\ast}=Z^{\ast}$. Then $\mathcal{F}_{X^{\ast} Y^{\ast}} \circ \mathcal{F}_{Y^{\ast} Z^{\ast}} \circ \mathcal{F}_{Z^{\ast} X^{\ast}}=\id$. On the other hand, $ \widetilde{c}_{X^{\ast}}(g_2 g_1^{-1}, g_1)=1$.\\
4) Assume $cr(cs-dr)\neq 0$. Then $Z^{\ast}\neq X^{\ast}\neq Y^{\ast}$, and $Y^{\ast}\neq Z^{\ast}$.   Note that
$$\begin{pmatrix}
c& d \\
r & s
\end{pmatrix}=\begin{pmatrix}
\tfrac{c}{r}& 0 \\
1 & \tfrac{r}{c} (cs-rd)
\end{pmatrix}\begin{pmatrix}
r& \tfrac{d}{c}r\\
0 & r^{-1}
\end{pmatrix}.$$
Hence there exists $p=\begin{pmatrix}
x& y \\
0& x^{-1}
\end{pmatrix}\in \SL_2(\R)$ such that
$$(c,d)p=(c',0), \quad \quad (r,s) p=(1,s')$$
for some positive $c'$, and non-zero $s'$.  Let $\begin{pmatrix} e_1'\\e_1^{'\ast}\end{pmatrix}=p\begin{pmatrix} e_1\\e_1^{\ast}\end{pmatrix}$. Under the basis of $\{e_1', e_1^{'\ast}\}$, $$e_{Y^{\ast}}=(c,d)\begin{pmatrix} e_1\\e_1^{\ast}\end{pmatrix}=(c,d)p\begin{pmatrix} e_1'\\e_1^{'\ast}\end{pmatrix}=(c', 0)\begin{pmatrix} e_1'\\e_1^{'\ast}\end{pmatrix},$$
$$e_{Y^{\ast}}=(r,s)\begin{pmatrix} e_1\\e_1^{\ast}\end{pmatrix}=(r,s)p\begin{pmatrix} e_1'\\e_1^{'\ast}\end{pmatrix}=(1, s')\begin{pmatrix} e_1'\\e_1^{'\ast}\end{pmatrix}.$$
By Lemmas \ref{pi1}, \ref{pi2}, we have:
$$\mathcal{F}_{ X^{\ast}Z^{\ast}} \circ \mathcal{F}_{Z^{\ast}Y^{\ast} } \circ \mathcal{F}_{Y^{\ast}X^{\ast} }= e^{\tfrac{\pi i \sgn(-s')}{4}} \id_{\mathcal{H}(X^{\ast})}.$$
In this case, $$g_1p=\begin{pmatrix}
\ast &-c^{'-1}\\
c' & 0
\end{pmatrix}, \quad  g_2p=\begin{pmatrix}
\ast &\ast\\
1 & s'
\end{pmatrix}, \quad g_2g_1^{-1}=g_2p (g_1p)^{-1}=\begin{pmatrix}
\ast &\ast\\
-c's'& \ast
\end{pmatrix}.$$
Hence $\widetilde{c}_{X^{\ast}}(g_2 g_1^{-1}, g_1)=\widetilde{c}_{X^{\ast}}(g_2p (g_1p)^{-1}, g_1p)=e^{\tfrac{\pi i\sgn(c' 1 (-c')s')}{4}}=e^{ \tfrac{\pi i\sgn(-s')}{4}}$. So the result is right.
\end{proof}

\subsection{Lattice model}
Let $L=\Z e_1+\Z e_1^{\ast}$ be a self-dual lattice with respect to $\psi$. Then $\pi_{\psi}\simeq \Ind_{\Ha(L)}^{\Ha(W)} \psi$. By Example \ref{laex}, we let $\mathcal{H}(L)$ be the set of measurable  functions $f: W \longrightarrow \C$ such that
\begin{itemize}
\item[(i)] $f(l+w)=\psi(-\tfrac{\langle x_{l}, x^{\ast}_l\rangle}{2}-\tfrac{\langle l, w\rangle}{2}) f(w)$, for all $l=x_l+x_l^{\ast}\in L=(X\cap L) \oplus (X^{\ast}\cap L)$, $w\in W$;
\item[(ii)] $\int_{L\setminus W} ||f(\dot{w})||^2 d\dot{w}<+\infty$.
\end{itemize}
Here, we choose a $W$-right invariant measure $d\dot{w}$ on $L\setminus W$. By the basis $\{e_1, e_1^{\ast}\}$, $L\setminus W \simeq (\R/\Z)^2$, so we choose the measure $d\dot{w}$, which corresponds to the product   normal  Haar measure on $(\R/\Z)^2$. Following \cite[p.149]{LiVe}, let $Y^{\ast}$ be a Lagrangian plane of $W$ such that  $Y^{\ast}\cap L\neq 0$.  Then there exists two coprime integers $c, d$ such that $ce_1+de_1^{\ast}\in Y^{\ast}$. Let us choose two integers $a, b$ such that $ad-bc=1$, i.e. $ g=\begin{pmatrix} a &b\\ c& d\end{pmatrix}\in \SL_2(\Z)$. Let $e_{Y}=ae_1+be_1^{\ast}$, $e_{Y^{\ast}}=ce_1+de_1^{\ast}$, $Y=\R e_Y$.  So $W=Y\oplus Y^{\ast}$ and $L=\Z e_Y\oplus \Z e_{Y^{\ast}}$.

Let us consider $2\mid cd$. By \cite[pp.164-165]{We}, or \cite[pp.142-145]{LiVe}, there exists  a pair of explicit isomorphisms between   $\mathcal{H}(Y^{\ast})\simeq L^2(Y)$ and  $\mathcal{H}(L)$, given as follows:
 \begin{equation}\label{eq1}
\theta_{L, Y^{\ast}}(f')( w)=\sum_{l\in L/L\cap Y^{\ast}} f'( w+l)\psi(\tfrac{\langle l, w\rangle}{2}+\tfrac{\langle y_{l}, y^{\ast}_l\rangle}{2}),
 \end{equation}
 \begin{equation}\label{eq2}
 \theta_{Y^{\ast}, L}(f)(y)=  \int_{Y^{\ast}/Y^{\ast}\cap L} f([\dot{y}^{\ast},0]+[y, 0]) d\dot{y}^{\ast},
      \end{equation}
    for $ w=x+ x^{\ast} \in W$, $y\in Y$,  $f'\in S(Y)\subseteq L^2(Y)$,  $f\in L^1(W)\cap \mathcal{H}(L)$. It can be checked that $\theta_{L, Y^{\ast}}$ and $\theta_{Y^{\ast}, L}$ define a pair of inverse intertwining operators between $ \Ind_{\Ha(L)}^{\Ha(W)} \psi$ and  $\Ind_{Y^{\ast} \times \R}^{\Ha(W)} \psi$.

Let us consider  $g_2=\begin{pmatrix} p &q\\ r& s\end{pmatrix} \in \SL_2(\Z)$. Let $e_Z=pe_1+qe_1^{\ast}$, $e_{Z^{\ast}}=re_1+se_1^{\ast}=e_1^{\ast}g_2$. Let $Z=\R e_Z$, $Z^{\ast}=\R e_{Z^{\ast}}$. Following \cite[p.149]{LiVe}, let  $\beta(Y^{\ast}, Z^{\ast})$ be an element in $T$ such that
     $$ \theta_{Y^{\ast}, L}  \circ  \theta_{ L, Z^{\ast}}=\beta(Y^{\ast}, Z^{\ast}) \mathcal{F}_{Y^{\ast}Z^{\ast}}.$$
By Lemma \ref{three}, we have:
$$\beta(X^{\ast}, Z^{\ast})\beta(Z^{\ast},Y^{\ast})\beta(Y^{\ast}, X^{\ast})=\widetilde{c}_{X^{\ast}}(g_2g_1^{-1}, g_1)^{-1}.$$
Let us define $$\widetilde{\beta}(g_1)=\beta(X^{\ast}g_1, X^{\ast}).$$
\begin{lemma}
$\widetilde{c}_{X^{\ast}}(h_1,h_2)=\widetilde{\beta}(h_1)^{-1}\widetilde{\beta}(h_2)^{-1}\widetilde{\beta}(h_1h_2)$,  for $h_i=\begin{pmatrix} a_i& \beta_i\\c_i& d_i\end{pmatrix}\in \SL_2(\Z)$ with $2\mid c_id_i$.
\end{lemma}
\begin{proof}
$\widetilde{\beta}(g_2)^{-1}=\beta(X^{\ast}, Z^{\ast})=\widetilde{\beta}([g_2g_1^{-1}]g_1)^{-1}$,
$\beta(Z^{\ast}, Y^{\ast})=\beta(X^{\ast}g_2 g_1^{-1}, X^{\ast})=\widetilde{\beta}(g_2g_1^{-1})$, $\beta(Y^{\ast}, X^{\ast})=\widetilde{\beta}(g_1)$. Hence $\widetilde{\beta}(g_2)^{-1}\widetilde{\beta}(g_2^{-1} g_1)\widetilde{\beta}(g_1)=\widetilde{c}_{X^{\ast}}(g_2g_1^{-1},g_1)^{-1}$. Then write $h_1=g_2g_1^{-1}$, $h_2=g_1$, $h_1h_2=g_2$.
\end{proof}
Following \cite{LiVe}, let $\breve{\Gamma}(2)$ be the group generated by $\Gamma(2)$ and $\begin{pmatrix} 0& -1\\1& 0\end{pmatrix} $.  It is clear that $\breve{\Gamma}(2)\simeq  \langle \begin{pmatrix} 0& -1\\1& 0\end{pmatrix} \rangle \ltimes \Gamma(2)$.
\begin{lemma}\label{trii}
The restriction of $[\widetilde{c}_{X^{\ast}}]$ on  $\breve{\Gamma}(2)$ is trivial, with an explicit trivialization:
$$
 \widetilde{\beta}:  \breve{\Gamma}(2) \longrightarrow \mu_8;
  g \longmapsto   \widetilde{\beta}(g),$$
such that $\widetilde{c}_{X^{\ast}}(g_1, g_2) = \widetilde{\beta}(g_1)^{-1} \widetilde{\beta}(g_2)^{-1} \widetilde{\beta}(g_1g_2)$, for $g_i\in \breve{\Gamma}(2)$.
\end{lemma}
\begin{proof}
It following from the above lemma.
\end{proof}
\subsection{The expression of $\widetilde{\beta}(g)$} Let $c, d$ be two coprime integers with $c d\neq  0$.  Following \cite[p.162]{LiVe}, let us define
$$ \betaup(c,d)\stackrel{\Delta}{=}|d|^{-\tfrac{1}{2}}\sum_{n=0}^{|d|-1} e^{-\tfrac{\pi i c n^2}{d}}.$$
\begin{lemma}
If  $cd $ is even, $\beta(d,c)\beta(c,d)=e^{-\tfrac{\pi i}{4} \sgn(dc)}$.
\end{lemma}
\begin{proof}
See \cite[p.170, Prop.]{LiVe}.
\end{proof}

\begin{lemma}
If  $d>0$, $c=2c'$, then
$$\betaup(d,c)=\left\{\begin{array}{cc}
 (\tfrac{c'}{|d|}) e^{-\tfrac{\pi i}{4} \sgn(c)} & \textrm{ if } d\equiv 1(\bmod 4),\\ (\tfrac{c'}{|d|} )i  e^{-\tfrac{\pi i}{4} \sgn(c)}   & \textrm{ if } d\equiv 3(\bmod 4).\end{array}\right.$$
 If $d<0$,  $c=2c'$, then $$\betaup(d,c)=\left\{\begin{array}{cc}
 (\tfrac{c'}{|d|}) e^{-\tfrac{\pi i}{4} \sgn(cd)} & \textrm{ if } d\equiv 3(\bmod 4),\\ (\tfrac{c'}{|d|} )(-i)  e^{-\tfrac{\pi i}{4} \sgn(cd)}   & \textrm{ if } d\equiv 1(\bmod 4).\end{array}\right.$$
\end{lemma}
\begin{proof}
Assume $d>0$. By the above lemma,  $\betaup(c,d)\betaup(d,c)=e^{-\tfrac{\pi i}{4} \sgn(cd)}=e^{-\tfrac{\pi i}{4} \sgn(c)}$.  Moreover, by \cite[Thm. 1.5.2]{BeEvWi}, we have: $$\betaup(c,d)=\left\{\begin{array}{cc}
 (\tfrac{c'}{d})  & \textrm{ if } d\equiv 1(\bmod 4),\\ -(\tfrac{c'}{d})i  & \textrm{ if } d\equiv 3(\bmod 4).\end{array}\right.$$
\end{proof}

Let  us calculate the constant $\widetilde{\beta}(g)$, for $g=\begin{pmatrix} a& b\\c& d\end{pmatrix} \in \SL_2(\Z)$ with $2\mid cd$. In this case, $X^{\ast}g=Y^{\ast}=\R (c e_1+de_1^{\ast})$.
By the basis $\{e_1, e_1^{\ast}\}$, we identity $W$ with $\R^2$, $X^{\ast}$ with $\R$, $X$ with $\R$.\\
 \begin{align}
\theta_{Y^{\ast}, L}(f)(xe_Y)&=  \int_{Y^{\ast}/Y^{\ast}\cap L} f([te_{Y^{\ast}},0]+[xe_Y, 0]) dt\\
 &=\int_{t\in \R/\Z} f(xe_y+te_{Y^{\ast}})\psi(-\tfrac{xt}{2})dt.
      \end{align}
 \begin{align}
\theta_{L, X^{\ast}}(f')( ye_1+ze_1^{\ast})&=\sum_{t\in \Z} f'(ye_1+ze_1^{\ast}+te_1)\psi(\tfrac{\langle te_1, ye_1+ze_1^{\ast}\rangle}{2})\\
&=\sum_{t\in \Z} f'(ye_1+ze_1^{\ast}+te_1)\psi(\tfrac{ tz}{2}).
 \end{align}
 \begin{align}
\mathcal{F}_{Y^{\ast}, X^{\ast}}(f')(xe_Y)&=\int_{te_{Y^{\ast}}\in Y^{\ast}}  f'([te_{Y^{\ast}}, 0]+[xe_Y,0])  |A_{Y^{\ast}X^{\ast}}|^{1/2}dt\\
&=\int_{t\in \R}  f'([te_{Y^{\ast}}, 0]+[xe_Y,0])  |A_{Y^{\ast}X^{\ast}}|^{1/2}dt\\
&=\int_{t\in \R}  f'([(td+xb)e_1^{\ast}, \tfrac{(tc+xa)(td+xb)-tx}{2}][(tc+xa)e_1, 0])|c|^{1/2}dt\\
&=\int_{t\in \R} \psi(\tfrac{(tc+xa)(td+xb)-tx}{2}) f'([(tc+xa)e_1, 0])|c|^{1/2}dt.
 \end{align}
 \begin{align}
& \theta_{Y^{\ast}, L}[\theta_{L, X^{\ast}}(f')](xe_Y)\\
 &=\int_{t\in \R/\Z} \theta_{L, X^{\ast}}(f')(xe_y+te_{Y^{\ast}})\psi(-\tfrac{xt}{2})dt\\
 &=\int_{t\in \R/\Z} \theta_{L, X^{\ast}}(f')((xa+tc)e_1+(xb+td)e_{2}^{\ast})\psi(-\tfrac{xt}{2})dt\\
 &=\int_{t\in \R/\Z} \sum_{s\in \Z} f'((xa+tc)e_1+(xb+td)e_1^{\ast}+se_1)\psi(\tfrac{ s(xb+td)}{2})\psi(-\tfrac{xt}{2})dt
 \end{align}
 \begin{align}
  &=\int_{t\in \R/\Z} \sum_{s\in \Z}  f'([(td+xb)e_1^{\ast}, \tfrac{(tc+xa+s)(td+xb)}{2}][(tc+xa+s)e_1, 0])\psi(\tfrac{ s(xb+td)}{2})\psi(-\tfrac{xt}{2})dt\\
  &=\int_{t\in \R/\Z} \sum_{s\in \Z}  f'([(tc+xa+s)e_1, 0])\psi(\tfrac{(tc+xa)(td+xb)-xt}{2} +s(xb+td))dt\\
  &= \int_{t\in \R/\Z}\sum_{l\in \Z/c\Z} \sum_{s\in \Z}  f'([(tc+xa+cs+l)e_1, 0])\psi(\tfrac{(tc+xa)(td+xb)-xt}{2} +(cs+l)(xb+td))dt\\
 & = \int_{t\in \R/\Z}\sum_{l\in \Z/c\Z} \sum_{s\in \Z}  f'([((t+s)c+xa+l)e_1, 0])\psi(\tfrac{((t+s)c+xa)((t+s)d+xb)-x(t+s)}{2} +l(xb+(t+s)d))dt\\
   & = \int_{t\in \R}\sum_{l\in \Z/c\Z}   f'([(tc+xa+l)e_1, 0])\psi(\tfrac{(tc+xa)(td+xb)-xt}{2} +l(xb+td))dt\\
   & =\sum_{l\in \Z/c\Z} \int_{t\in \R}  f'([(tc+xa)e_1, 0])\psi(\tfrac{([t-\tfrac{l}{c}]c+xa)([t-\tfrac{l}{c}]d+xb)-x[t-\tfrac{l}{c}]}{2}+l[t-\tfrac{l}{c}]d+lxb)dt\\
    & =\int_{t\in \R}  f'([(tc+xa)e_1, 0])\psi(\tfrac{(tc+xa)(td+xb)-xt}{2})[\sum_{l\in \Z/c\Z} \psi(-\tfrac{dl^2}{2c})]dt.
 \end{align}
Hence
$$\widetilde{\beta}(g)=|c|^{-1/2}\sum_{l\in \Z/c\Z} \psi(-\tfrac{dl^2}{2c})=\beta(d,c).$$
\begin{lemma}
 The restriction of $[\overline{c}_{X^{\ast}}]$ on $\Gamma(2) $ is trivial, with an explicit trivialization:
 $$\begin{array}{rlcl}
\overline{s}: & \Gamma(2) &\longrightarrow &\mu_8; \\
  &  \begin{pmatrix}
a & b\\
c& d\end{pmatrix} &\longmapsto  & \left\{\begin{array}{cc} e^{-\tfrac{\pi i[1-\sgn(d)]}{4}}& \textrm{ if } c=0,\\ \beta(d,c)e^{\tfrac{\pi i \sgn(c)}{4}} & \textrm{ if } c\neq 0.\end{array}\right.
\end{array}$$
 such that $\overline{c}_{X^{\ast}}(g_1, g_2) =\overline{ \beta}(g_1)^{-1}\overline{ \beta}(g_2)^{-1}\overline{ \beta}(g_1g_2)$, for $g_i\in \Gamma(2) $.
\end{lemma}
\begin{proof}
\begin{align*}
\overline{c}_{X^{\ast}}(g_1, g_2)&=m_{X^{\ast}}(g_1g_2)^{-1} m_{X^{\ast}}(g_1) m_{X^{\ast}}(g_2) \widetilde{c}_{ {X^{\ast}}}(g_1,g_2)\\
&=[ \widetilde{\beta}(g_1g_2)m_{X^{\ast}}(g_1g_2)^{-1}] [ \widetilde{\beta}(g_1)m_{X^{\ast}}(g_1)^{-1}]^{-1} [\widetilde{\beta}(g_2)m_{X^{\ast}}(g_2)^{-1}]^{-1}.
\end{align*}
So we can take $\overline{ \beta}(g)= \widetilde{\beta}(g)m_{X^{\ast}}(g)^{-1}$.  \\
1) If $c=0$, $ \widetilde{\beta}(g)=1$, $m_{X^{\ast}}(g)^{-1}=e^{-\tfrac{\pi i[1-\sgn(d)]}{4}}$.\\
2) If $c\neq 0$, $ \widetilde{\beta}(g)=\beta(d,c)$, $m_{X^{\ast}}(g)^{-1}=e^{\tfrac{\pi i \sgn(c)}{4}}$.
\end{proof}
Note that $\Gamma^{\pm}(2)\cap \SO_2(\R)=\langle \pm I\rangle$. By calculation,  one can see that the two trivialization maps on $\langle \pm I\rangle$ are different, up to a sign character. As $ \Gamma(2) \simeq \langle \pm I\rangle \times \langle  n_{2}, n^-_{2}\rangle $,  for $n_2 =\begin{pmatrix} 1& 2 \\ 0 &1\end{pmatrix}$, $n^-_{2}=\begin{pmatrix} 1& 0 \\ 2 &1\end{pmatrix}$, let us modify the  trivialization map on $\Gamma(2)$ by the sign character and write it by $\overline{\beta}^-$. Similarly, we can modify  the trivialization map $\widetilde{\beta}$ on $\Gamma(2)$ by the sign character and write it by $\widetilde{\beta}^-$.
\subsection{The splitting of $\widehat{\SL_2}(\Z)$}
  In this section, we will let
$g=\begin{pmatrix}
a & b\\ c & d\end{pmatrix}, h= \begin{pmatrix}
p& q\\ r & s\end{pmatrix}$ be two elements of $\SL_2(\Z)$. Let us first recall some results of Dedekind sums from \cite{RaWh},\cite{Po}.
\begin{itemize}
\item $((x))=\left\{\begin{array}{ll} x-[x]-\tfrac{1}{2} & \textrm{ if } x\in \R\setminus \Z,\\ 0 & \textrm{ if } x\in \Z.\end{array}\right.$
\item $s(d,c)=\sum_{k\bmod(|c|)} ((\tfrac{k}{c}))((\tfrac{kd}{c}))$, for two coprime integers $c,d$.(Dedekind sum)
\item $s(-d, c)=-s(d,c)$, $s(d,-c)=s(d,c)$.
\end{itemize}
\begin{theorem}
\begin{itemize}
\item[(a)] For positive odd $d$, we have $$12 ds(c,d)\equiv d+1-(\tfrac{c}{d}) (\bmod 8),$$ where $(\tfrac{c}{d})$ is the Jacobi symbol.\\
\item[(b)] If $c>0, d>0$ and $c,d $ are coprime, we have
$$12s(d,c)+12s(c,d)=-3+ \tfrac{d}{c}+\tfrac{c}{d}+ \tfrac{1}{cd}.$$
\end{itemize}
\end{theorem}
\begin{proof}
See \cite[Theorems 3, 18]{RaWh}.
\end{proof}
\begin{theorem}
$s(c,d)+s(r,s)=s(cu-dv,cs+dr)-\tfrac{1}{4}+\tfrac{1}{12}(\tfrac{d}{st}+\tfrac{s}{td}+\tfrac{t}{ds})$, for all positive $c,d,r,s$, where  $t=cs+dr$ and  $u,v$ are any two integers such that $ru+sv=1$.
\end{theorem}
\begin{proof}
See \cite[Theorem 7]{Po}.
\end{proof}
Let $g=\begin{pmatrix}
a & b\\ c& d\end{pmatrix} \in \SL_2(\Z)$. In \cite{As}, Asai gave the following function:
$$\nu(g)=\left\{\begin{array}{ll} \tfrac{b}{12d}+\tfrac{1-sgn(d)}{4}  & \textrm{ if } c=0,\\ \tfrac{a+d}{12c}-\sgn(c)\bigg( \tfrac{1}{4}+s(d, |c|)\bigg)& \textrm{ if } c\neq 0.\end{array}\right.$$

By \cite{As},  the  signature cocycle on $\SL_2(\Z)$ with values on $\Z$ can be splitting by the above function. In fact, the  signature  cocycle  differs from   the Rao-Perrin cocycle only by a character $e^{-\pi i (-)}$. So let us define:
\begin{itemize}
\item $\overline{\beta_1}(g)=e^{-\pi i \nu(g)}$.
\item $\widetilde{\beta_1}(g_i)=\overline{\beta_1}(g_1)m_{X^{\ast}}(g_1)=\left\{ \begin{array}{lr} e^{-\pi i \tfrac{b}{12a}}&  c=0,\\ e^{-\pi i \Big[\tfrac{a+d}{12c}-\sgn(c) s(d, |c|)\Big]}& c\neq 0. \end{array}\right.$
\end{itemize}
 Let $g_i= \begin{pmatrix} a_i& b_i\\c_i& d_i\end{pmatrix} \in \SL_2(\Z)$.
\begin{lemma}
$\overline{c}_{X^{\ast}}(g_1,g_2)=\overline{\beta_1}(g_1)^{-1}\overline{\beta_1}(g_2)^{-1}\overline{\beta_1}(g_1g_2)$.
\end{lemma}
\begin{proof}
See \cite[2-3]{As}.
\end{proof}
\begin{lemma}
$\widetilde{c}_{X^{\ast}}(g_1,g_2)=\widetilde{\beta_1}(g_1)^{-1}\widetilde{\beta_1}(g_2)^{-1}\widetilde{\beta_1}(g_1g_2)$.
\end{lemma}
\begin{proof}
\begin{align*}
\widetilde{c}_{ {X^{\ast}}}(g_1,g_2) &=m_{X^{\ast}}(g_1g_2) m_{X^{\ast}}(g_1)^{-1} m_{X^{\ast}}(g_2)^{-1}\overline{c}_{X^{\ast}}(g_1, g_2) \\
&=[\overline{\beta_1}(g_1g_2)m_{X^{\ast}}(g_1g_2)] [\overline{\beta_1}(g_1)m_{X^{\ast}}(g_1)]^{-1} [\overline{\beta_1}(g_2)m_{X^{\ast}}(g_2)]^{-1}.
\end{align*}
\end{proof}
Now let us consider the subgroup $\Gamma(2) \subseteq \SL_2(\Z)$.  According to Lemma \ref{trii}, the cocycle $\widetilde{c}_{ {X^{\ast}}}$ on $\Gamma(2)$ can also be splitting by the function $\widetilde{\beta}$. So $\widetilde{\beta}$   differs from  $\widetilde{\beta_1}$ by a character $\chi$. Let us formulate this character explicitly.
Let $\chi(g)=\widetilde{\beta_1}(g)/\widetilde{\beta}(g)$, for $g=\begin{pmatrix} a & b\\ c& d\end{pmatrix}\in \Gamma(2)$.
\begin{itemize}
\item \begin{align*}
12\sgn(c)s(d, |c|)&=\sgn(d)12\sgn(c)s(|d|, |c|)\\
&=\sgn(d)[-3\sgn(c)+ \tfrac{|d|}{c}+\tfrac{c}{|d|}+ \tfrac{1}{c|d|}-12s(c,|d|)]\\
&=-3\sgn(cd)+ \tfrac{d}{c}+\tfrac{c}{d}+ \tfrac{1}{cd}-12\sgn(d)s(c,d).
\end{align*}
 \begin{align*}
\tfrac{a+d}{12c}-\sgn(c) s(d, |c|)&=\tfrac{1}{12}[\tfrac{a}{c}+\tfrac{d}{c}-12\sgn(c)s(d, |c|)]\\
&=\tfrac{1}{12c}[a+d+3c\sgn(cd)-d-\tfrac{c^2}{d}- \tfrac{1}{d}+12\sgn(d)cs(c,d)]\\
&=\tfrac{1}{12}[3\sgn(cd)-\tfrac{c}{d}+\tfrac{b}{d}+12\sgn(d)s(c,d)].
\end{align*}
\item If  $d>0$, $c=2c'$, then
$$\betaup(d,c)=\left\{\begin{array}{cc}
 (\tfrac{c'}{|d|}) e^{-\tfrac{\pi i}{4} \sgn(c)} & \textrm{ if } d\equiv 1(\bmod 4),\\ (\tfrac{c'}{|d|} )i  e^{-\tfrac{\pi i}{4} \sgn(c)}  & \textrm{ if } d\equiv 3(\bmod 4).\end{array}\right.$$
\item If $d<0$,  $c=2c'$, then $$\betaup(d,c)=\left\{\begin{array}{cc}
 (\tfrac{c'}{|d|}) e^{-\tfrac{\pi i}{4} \sgn(cd)} & \textrm{ if } d\equiv 3(\bmod 4),\\ (\tfrac{c'}{|d|} )(-i)  e^{-\tfrac{\pi i}{4} \sgn(cd)}  & \textrm{ if } d\equiv 1(\bmod 4).\end{array}\right.$$
\item If $c=0$, $\chi(g)=e^{-\pi i \tfrac{b}{12a}}$.
\item If $c\neq 0$, $d>0$,  $d\equiv 1(\bmod 4)$, and $c=2c'$,  then:
\begin{align*}
\chi(g)&= e^{-\pi i \Big[\tfrac{a+d}{12c}-\sgn(c) s(d, |c|)\Big]}\beta^{-1}(d,c)\\
&= (\tfrac{c'}{|d|})e^{\tfrac{-\pi i}{12}[-\tfrac{c}{d}+\tfrac{b}{d}+12s(c,d)]}.
 \end{align*}
 \item If $c\neq 0$, $d>0$,  $d\equiv 3(\bmod 4)$, and $c=2c'$,  then:
\begin{align*}
\chi(g)&= e^{-\pi i \Big[\tfrac{a+d}{12c}-\sgn(c) s(d, |c|)\Big]}\beta^{-1}(d,c)\\
&=  (\tfrac{c'}{|d|} )(-i) e^{\tfrac{-\pi i}{12}[-\tfrac{c}{d}+\tfrac{b}{d}+12s(c,d)]}.
 \end{align*}
 \item If $c\neq 0$, $d<0$,  $d\equiv 3(\bmod 4)$, and $c=2c'$,  then:
\begin{align*}
\chi(g)&= e^{-\pi i \Big[\tfrac{a+d}{12c}-\sgn(c) s(d, |c|)\Big]}\beta^{-1}(d,c)\\
&= e^{-\pi i \tfrac{1}{12}[-3\sgn(c)-\tfrac{c}{d}+\tfrac{b}{d}-12s(c,d)]} (\tfrac{c'}{|d|}) e^{\tfrac{-\pi i}{4} \sgn(c)} \\
&=(\tfrac{c'}{|d|})e^{-\pi i \tfrac{1}{12}[-\tfrac{c}{d}+\tfrac{b}{d}-12s(c,d)]}.
 \end{align*}
  \item If $c\neq 0$, $d<0$,  $d\equiv 1(\bmod 4)$, and $c=2c'$,  then:
\begin{align*}
\chi(g)&=e^{-\pi i \Big[\tfrac{a+d}{12c}-\sgn(c) s(d, |c|)\Big]}\beta^{-1}(d,c)\\
&=e^{-\pi i \tfrac{1}{12}[-3\sgn(c)-\tfrac{c}{d}+\tfrac{b}{d}-12s(c,d)]} (\tfrac{c'}{|d|} )i  e^{-\tfrac{\pi i}{4} \sgn(c)} \\
&=(\tfrac{c'}{|d|})ie^{-\pi i \tfrac{1}{12}[-\tfrac{c}{d}+\tfrac{b}{d}-12s(c,d)]}.
 \end{align*}
\end{itemize}
Let $n_{2k}=\begin{pmatrix} 1& 2k \\ 0 &1\end{pmatrix}$, $n^-_{2k}=\begin{pmatrix} 1& 0 \\ 2k &1\end{pmatrix}$. Note that $\chi(n_{2k})=e^{-\tfrac{k\pi i}{6}}$, $\chi(n^-_{2k})=e^{\tfrac{k\pi i}{6}}$.
\begin{lemma}
Let $g=\begin{pmatrix} a & b\\c& d\end{pmatrix}\in \Gamma(2)$, and $c=2c'$.
\begin{itemize}
\item[(1)] $\chi(\pm g)=\chi(g)$.
\item[(2)] $\chi( g n_{2})=\chi( g)\chi(n_{2})$
\item[(3)] $\chi( g n^-_{2})=\chi( g)\chi(n^-_{2})$.
\end{itemize}
\end{lemma}
\begin{proof}
1) If $c=0$, $\chi(- g)=e^{-\pi i \tfrac{-b}{-12a}}=\chi(g)$. \\
If $c\neq 0$,  then
\begin{align*}
\chi(- g)&= e^{-\pi i \Big[\tfrac{-a-d}{-12c}-\sgn(-c) s(-d, |c|)\Big]}\beta^{-1}(-d,-c)\\
&=e^{-\pi i \Big[\tfrac{a+d}{12c}-\sgn(c) s(d, |c|)\Big]}\beta^{-1}(-d,-c)\\
&=\chi( g)\beta^{-1}(-d,-c)/\beta^{-1}(d,c)\\
&=\chi(g).
\end{align*}
2) Note that $gn_2=\begin{pmatrix} a & b+2a\\ c & d+2c\end{pmatrix}$. \\
If $c=0$, $\chi(gn_2)=e^{-\pi i \tfrac{b+2a}{12a}}=\chi(g)\tfrac{-2\pi i}{12}=\chi(g)\chi(n_2) $.\\
If $c\neq 0$,  then:
\begin{align*}
\chi( gn_2)&= e^{-\pi i \Big[\tfrac{a+d+2c}{12c}-\sgn(c) s(d+2c, |c|)\Big]}\beta^{-1}(d+2c,c)\\
&= e^{-\pi i\tfrac{1}{6}} e^{-\pi i \Big[\tfrac{a+d}{12c}-\sgn(c) s(d, |c|)\Big]}\beta^{-1}(d,c)\\
&=\chi(n_2)\chi( g).
\end{align*}
3) Note that $g n^-_{2}=\begin{pmatrix} a+2b & b\\ c+2d & d\end{pmatrix}$.\\
If $c=0$,  $g n^-_{2}=\begin{pmatrix} a & b\\ 2d & d\end{pmatrix}$. Then
\begin{align*}
\chi(g n^-_{2})&= e^{\tfrac{-\pi i}{12}[3\sgn(2dd)-\tfrac{2d}{d}+\tfrac{b}{d}+12\sgn(d)s(2d,d)]}\beta^{-1}(d,2d)\\
&=e^{\tfrac{-\pi i}{12}[3-2+\tfrac{b}{d}]}\beta^{-1}(1,2)\\
&=e^{\tfrac{-\pi i}{12}\tfrac{b}{d}} e^{\tfrac{-\pi i}{12}}e^{\tfrac{\pi i}{4}}\\
&=e^{\tfrac{-\pi i}{12}\tfrac{b}{d}} e^{\tfrac{\pi i}{6}}\\
&=\chi(g )\chi(n^-_{2}).
\end{align*}
If $c\neq0$, by Part (1), we can assume $d>0$. \\
a)  $c+2d=0$, $d=1$, $c=-2$, $a+2b=1$. Then:
$$\chi(g n^-_{2})=e^{-\pi i \tfrac{b}{12}}, \quad \quad  \chi( n^-_{2})=e^{\tfrac{\pi i}{6}};$$
\begin{align*}
\chi(g )&=e^{\tfrac{-\pi i}{12}\tfrac{1}{12}[3\sgn(cd)-\tfrac{c}{d}+\tfrac{b}{d}+12\sgn(d)s(c,d)]}\betaup(d,c)^{-1}\\
&=e^{\tfrac{-\pi i}{12}\tfrac{1}{12}[-3+2+b]}\betaup(1,-2)^{-1}\\
&=e^{-\pi i\tfrac{b}{12}}e^{\tfrac{\pi i}{12}}e^{-\tfrac{\pi i}{4}}\\
&=e^{-\pi i \tfrac{b}{12}}e^{-\tfrac{\pi i}{6}}\\
&=\chi(g n^-_{2}) \chi(n^-_{2})^{-1}.
\end{align*}
b) $c+2d\neq 0$, $d\equiv 1(\bmod 4)$, $c=2c'$.
\begin{align*}
\chi(gn_2^-)&= (\tfrac{c'+d}{d})e^{\tfrac{-\pi i}{12}[-\tfrac{ c+2d}{d}+\tfrac{b}{d}+12s(c+2d,d)]}\\
&=(\tfrac{c'}{d})e^{\tfrac{-\pi i}{12}[-\tfrac{ c}{d}+\tfrac{b}{d}+12s(c,d)]}e^{\tfrac{\pi i}{6}}\\
&=\chi(g)\chi(n_2^-).
 \end{align*}
c) $c+2d\neq 0$, $d\equiv 3(\bmod 4)$, $c=2c'$.
\begin{align*}
\chi(gn_2^-)&=  (\tfrac{c'+d}{d} )(-i) e^{\tfrac{-\pi i}{12}[-\tfrac{c+2d}{d}+\tfrac{b}{d}+12s(c+2d,d)]}\\
&=  (\tfrac{c'}{|d|} )(-i) e^{\tfrac{-\pi i}{12}[-\tfrac{c}{d}+\tfrac{b}{d}+12s(c,d)]}e^{\tfrac{\pi i}{6}}\\
&=\chi(g)\chi(n_2^-).
 \end{align*}
\end{proof}
 According to A. Putman's answer on the question titled 'Generators for congruence group $\Gamma(2)$' on MathOverflow, $\Gamma(2)/\{\pm I\}$ is a free group of two generators  $n_2,n_2^-$.  So by the above lemma, $\chi$ is a character of $\Gamma(2)$.  Moreover, this character is determined by $\chi(\pm I)=1$, $\chi(n_2)=e^{-\tfrac{\pi i}{6}}$, $\chi(n_2^-)=e^{\tfrac{\pi i}{6}}$.
\begin{question}
\begin{itemize}
\item Does  everything can be done  for higher rank symplectic groups?
\item  Moreover, one can  also connect them  with  A. Putman,  J.E. Pommersheim,  R.Kirby, P.Melvin  R.Brooks, or  others'  works in some sense?
\end{itemize}
\end{question}
\subsection{Two models} Go back to the formulas (\ref{eq1})(\ref{eq2}). Take $Y^{\ast}=X^{\ast}, Y=X$.  Through the isomorphisms $\theta_{L, X^{\ast}}$, $\theta_{ X^{\ast},L}$, we can transfer the actions of $\widehat{\Sp}(W)$ from  $\mathcal{H}(X^{\ast})$ to  $\mathcal{H}(L)$.  For an element   $f\in \mathcal{H}(L)$, let $f'=\theta_{ X^{\ast},L}(f)\in \mathcal{H}(X^{\ast})$. Then
$f=\theta_{L, X^{\ast}}(f')$.  For $g\in  \widehat{\Sp}(W)$, we can define
     \begin{equation}
     \pi_{\psi}(g)f=\theta_{L, X^{\ast}}[\pi_{\psi}(g)(f')]=\theta_{L, X^{\ast}}[\pi_{\psi}(g)\theta_{X^{\ast},L}(f)].
     \end{equation}
     Under such action, for $g_1, g_2\in \widehat{\Sp}(W)$,
     \begin{equation}
     \begin{split}
     \pi_{\psi}(g_1)[\pi_{\psi}(g_2)f]&=\theta_{L, X^{\ast}}[\pi_{\psi}(g_1)\theta_{X^{\ast},L}]([\pi_{\psi}(g_2)f])\\
     &=\theta_{L, X^{\ast}}[\pi_{\psi}(g_1)\theta_{X^{\ast},L}\theta_{L, X^{\ast}}
     [\pi_{\psi}(g_2)\theta_{X^{\ast},L}(f)]))\\
     &=\theta_{L, X^{\ast}}[\pi_{\psi}(g_1)\pi_{\psi}(g_2)\theta_{X^{\ast},L}(f)]\\
     &=\widetilde{c}_{X^{\ast}}(g_1, g_2)\pi_{\psi}(g_1g_2)f.
     \end{split}
     \end{equation}
\subsection{$\SL_2(\Z)$}\label{sl2z} In particular, let us consider  $g\in \Sp(L)$. Under the basis $\{ e_1,  e_1^{\ast}\}$, $\Sp(L)\simeq \SL_2(\Z)$.  Let $w=x+x^{\ast} \in W$.  \\
Case 1: $g=u(b)\in \SL_2(\Z)$ with  $2\mid b$.
\begin{equation*}
\begin{split}
\pi_{\psi}(g)f( w)& =\sum_{l\in L/L\cap X^{\ast}}\pi_{\psi}(g)( f')( w+l)\psi(\tfrac{\langle l, w\rangle}{2}+\tfrac{\langle x_{l}, x^{\ast}_l\rangle}{2})\\
&=\sum_{l\in L\cap X}\pi_{\psi}(g)(f')(w+l)\psi(\tfrac{\langle l, w\rangle}{2}) \\
&=\sum_{l\in L\cap X}\pi_{\psi}(g)(f')(x+l)\psi(\langle l, w\rangle)\psi(\tfrac{\langle x, x^{\ast}\rangle}{2}) \\
&=\sum_{l\in L\cap X}\psi(\tfrac{1}{2}\langle (x+l),(x+l)b\rangle) f'(x+l)\psi(\langle l, w\rangle)\psi(\tfrac{\langle x, x^{\ast}\rangle}{2})\\
&= \sum_{l\in L\cap X} \int_{X^{\ast}/X^{\ast}\cap L}  \psi(\tfrac{1}{2}\langle (x+l), (x+l)b\rangle)f(x+l+\dot{y}^{\ast}) \psi(-\tfrac{\langle x+l, \dot{y}^{\ast}\rangle}{2})  \psi(\langle l, w\rangle)\psi(\tfrac{\langle x, x^{\ast}\rangle}{2})d\dot{y}^{\ast}\\
&=\sum_{l\in L\cap X} \int_{X^{\ast}/X^{\ast}\cap L} f(x+l+\dot{y}^{\ast}) \psi(-\tfrac{\langle x+l, \dot{y}^{\ast}\rangle}{2})   \psi(\langle l, x^{\ast}+xb\rangle)\psi(\tfrac{\langle x, x^{\ast}+xb\rangle}{2})d \dot{y}^{\ast}\\
&= f( wg).
\end{split}
\end{equation*}
Case 2: $g=h(a) \in \SL_2(\Z)$, $a\in \{\pm 1\}$.
\begin{equation}
\begin{split}
\pi_{\psi}(g)f(w)&=\sum_{l\in L/L\cap X^{\ast}}\pi_{\psi}(g)( f')( w+l)\psi(\tfrac{\langle l, w\rangle}{2}+\tfrac{\langle x_{l}, x^{\ast}_l\rangle}{2})\\
&=\sum_{l\in L\cap X}\pi_{\psi}(g)(f')(x+l)\psi(\langle l, w\rangle)\psi(\tfrac{\langle x, x^{\ast}\rangle}{2}) \\
&=\sum_{l\in L\cap X}f'(-x-l)\psi(\langle l, w\rangle)\psi(\tfrac{\langle x, x^{\ast}\rangle}{2}) \\
&= \sum_{l\in L\cap X}\int_{X^{\ast}/X^{\ast}\cap L} f(-x-l+\dot{y}^{\ast}) \psi(-\tfrac{\langle -x-l, \dot{y}^{\ast}\rangle}{2}) \psi(\langle l, w\rangle)\psi(\tfrac{\langle x, x^{\ast}\rangle}{2}) d\dot{y}^{\ast}\\
&= \sum_{l\in L\cap X}\int_{X^{\ast}/X^{\ast}\cap L} f(-x+l+\dot{y}^{\ast}) \psi(-\tfrac{\langle -x+l, \dot{y}^{\ast}\rangle}{2}) \psi(\langle l, -w\rangle)\psi(\tfrac{\langle -x, -x^{\ast}\rangle}{2}) d\dot{y}^{\ast}\\
&=f(wg).
\end{split}
\end{equation}
Case 3: $g=\omega\in  \SL_2(\Z)$.
\begin{equation}\label{eq3}
\begin{split}
\pi_{\psi}(g)f(w)&=\sum_{l\in L/L\cap X^{\ast}}\pi_{\psi}(g)( f')( w+l)\psi(\tfrac{\langle l, w\rangle}{2}+\tfrac{\langle x_{l}, x^{\ast}_l\rangle}{2})\\
&=\sum_{l\in L\cap X}\pi_{\psi}(g)(f')(x+l)\psi(\langle l, w\rangle)\psi(\tfrac{\langle x, x^{\ast}\rangle}{2}) \\
&=\sum_{l\in L\cap X} \int_{X^{\ast}}  f'(y^{\ast}\omega^{-1}) \psi(\langle x+l, y^{\ast}\rangle) \psi(\langle l, w\rangle)\psi(\tfrac{\langle x, x^{\ast}\rangle}{2})dy^{\ast}\\
&=\sum_{l\in L\cap X} \int_{X^{\ast}}  f'(y^{\ast}\omega^{-1}) \psi(\langle x, y^{\ast}\rangle) \psi(\langle l, x^{\ast}+y^{\ast}\rangle)\psi(\tfrac{\langle x, x^{\ast}\rangle}{2})dy^{\ast}\\
&=\sum_{l\in L\cap X} \int_{X^{\ast}}  \int_{X^{\ast}/X^{\ast}\cap L}   f(y^{\ast}\omega^{-1}+\dot{z}^{\ast})\psi(-\tfrac{\langle y^{\ast}\omega^{-1}, \dot{z}^{\ast}\rangle}{2}) \psi(\langle x+l, y^{\ast}\rangle) \psi(\langle l, w\rangle)\psi(\tfrac{\langle x, x^{\ast}\rangle}{2})dz^{\ast}dy^{\ast}\\
&= \sum_{l^{\ast}\in L\cap X^{\ast}}\int_{X^{\ast}/X^{\ast}\cap L} f((-x^{\ast}+l^{\ast})\omega^{-1}+\dot{z}^{\ast})\psi(-\tfrac{\langle (-x^{\ast}+l^{\ast})\omega^{-1}, \dot{z}^{\ast}\rangle}{2}) \psi(\langle x, -x^{\ast}+l^{\ast}\rangle) \psi(\tfrac{\langle x, x^{\ast}\rangle}{2})dz^{\ast}\\
&= \sum_{l\in L\cap X}\int_{X^{\ast}/X^{\ast}\cap L} f((-x^{\ast}\omega^{-1}+l+\dot{z}^{\ast})\psi(-\tfrac{\langle (-x^{\ast})\omega^{-1}+l, \dot{z}^{\ast}\rangle}{2}) \psi(\langle x, -x^{\ast}+l\omega\rangle) \psi(\tfrac{\langle x, x^{\ast}\rangle}{2})dz^{\ast}\\
&= \sum_{l\in L\cap X}\int_{X^{\ast}/X^{\ast}\cap L} f(x^{\ast}\omega+l+\dot{z}^{\ast})\psi(-\tfrac{\langle x^{\ast}\omega+l, \dot{z}^{\ast}\rangle}{2}) \psi(\langle l, x\omega\rangle) \psi(-\tfrac{\langle x, x^{\ast}\rangle}{2})dz^{\ast}\\
&= \sum_{l\in L\cap X}\int_{X^{\ast}/X^{\ast}\cap L} f(x^{\ast}\omega+l+\dot{z}^{\ast})\psi(-\tfrac{\langle x^{\ast}\omega+l, \dot{z}^{\ast}\rangle}{2}) \psi(\langle l, x\omega\rangle) \psi(\tfrac{\langle  x^{\ast}\omega, x\omega\rangle}{2})dz^{\ast}\\
&=f(wg).
\end{split}
\end{equation}
Case 4: $g=u_-(c)=\omega^{-1} u(-c) \omega\in \SL_2(\Z)$ with $2\mid c$. Then:
\begin{equation}\label{eq4}
\begin{split}
\pi_{\psi}(g)f(w )&=\pi_{\psi}(\omega^{-1} u(-c) \omega) f(w)\\
&=\widetilde{c}_{X^{\ast}}(\omega^{-1}u(-c), \omega)^{-1}\pi_{\psi}(\omega^{-1})\pi_{\psi}(u(-c)) \pi_{\psi}(\omega) f(w)\\
&=\widetilde{c}_{X^{\ast}}(\omega^{-1}u(-c), \omega)^{-1} \pi_{\psi}(\omega) f(wg)\\
&=e^{\tfrac{\pi i}{4} \sgn(c)} f(wg).
\end{split}
\end{equation}
Recall $n_{2}=\begin{pmatrix} 1& 2 \\ 0 &1\end{pmatrix}$, $n^-_{2}=\begin{pmatrix} 1& 0 \\ 2 &1\end{pmatrix}$. Moreover, $\Gamma(2)/{\pm I}$ is a free group with the generators $n_{2}$, $n^-_{2}$. For any $g\in \Gamma(2)$, let us write it in the form
\begin{align}\label{kkk}
g=(-I)^{\epsilon} n_{2}^{k_1} (n^-_{2})^{k_2}\cdots n_{2}^{k_l} (n^-_{2})^{k_{l+1}}.
\end{align}
Hence for any $g$, there exists $\epsilon_g\in \mu_8$, which is a product of Weil indexes such that
\begin{equation}\label{eq5}
 \pi_{\psi}(g)f(w)= \epsilon_g f(wg).
 \end{equation}
 Note that $\epsilon_g$ is determined by $(\ref{eq5})$, and can be calculated  from (\ref{kkk}). Then: for $g_1, g_2\in \Gamma(2)$, we have:
 \begin{align*}
 \pi_{\psi}(g_1)\pi_{\psi}(g_2)f(w)&=\epsilon_{g_1} \epsilon_{g_2}f(wg_1g_2)\\
 &=\widetilde{c}_{X^{\ast}}(g_1,g_2)\pi_{\psi}(g_1g_2)f(w)\\
 &= \widetilde{c}_{X^{\ast}}(g_1,g_2)\epsilon_{g_1g_2}f(wg_1g_2).
 \end{align*}
 Hence
 $$\widetilde{c}_{X^{\ast}}(g_1,g_2)=\epsilon_{g_1} \epsilon_{g_2}\epsilon_{g_1g_2}^{-1}.$$
 According to Lemma \ref{trii},   $\widetilde{\beta}$ differs from $\epsilon_{g}^{-1}$ by a character $\chiup$, where $\chiup(g)=\widetilde{\beta}(g)\epsilon_{g}$.\\
 1) If $g=-I$, $\chiup(g)=1$.\\
 2) If $g=n_2$, $\chiup(g)=1$.\\
 3) If $g=n_2^-$, $\chiup(g)=\widetilde{\beta}(g)\epsilon_{g} =\beta(1,2)e^{\tfrac{\pi i}{4} \sgn(2)}=1$.
 Hence $\widetilde{\beta}(g)=\epsilon_{g}^{-1}$.

\labelwidth=4em
\addtolength\leftskip{25pt}
\setlength\labelsep{0pt}
\addtolength\parskip{\smallskipamount}

\end{document}